\documentclass{article}

\usepackage[utf8]{inputenc}
\usepackage[T1]{fontenc}

\usepackage{amsmath,amsfonts,amssymb,amsthm,cases,graphicx,authblk}
\usepackage{tikz}
\usepackage{tikz-3dplot}
\usepackage{pgfplots}
\usepackage{hyperref}
\usepackage{color}
\usepackage{vmargin}
\usepackage{nccmath}
\usepackage{verbatim}
\usepackage{float}
\usepackage{url}
\setmarginsrb{2cm}{1cm}{2cm}{3cm}{1cm}{1cm}{2cm}{2cm}

\newcommand{\Cof}{\operatorname{Cof}}

\renewcommand{\leq}{\leqslant}

\renewcommand{\geq}{\geqslant}

\everymath{\displaystyle}
\numberwithin{equation}{section}

\newtheorem{Theorem}{Theorem}[section]

\newtheorem{Proposition}[Theorem]{Proposition}
\newtheorem{Lemma}[Theorem]{Lemma}
\newtheorem{Remark}[Theorem]{Remark}

\title{Local null controllability of a fluid-rigid body interaction problem with Navier slip boundary conditions} 

\author[1]{Imene Aicha Djebour}
\affil[1]{Universit\'e de Lorraine, CNRS, Inria, IECL, F-54000 Nancy}

\date{\today}
\pgfplotsset{width=7cm,compat=1.8}
\begin{document}
	\maketitle	
\begin{abstract}
The aim of this work is to show the local null controllability of a fluid-solid interaction system by using a distributed control located in the fluid. 
The fluid is modeled by the incompressible Navier-Stokes system with Navier slip boundary conditions and the rigid body is governed by the Newton laws.
Our main result yields that we can drive the velocities of the fluid and of the structure to 0 and we can control exactly the position of the rigid body.
One important ingredient consists in a new Carleman estimate for a linear fluid-rigid body system with Navier boundary conditions. This work is done without imposing any geometrical conditions on the rigid body. 
\end{abstract} 
	
	\vspace{1cm}
	
	

	\tableofcontents

\section{Introduction}
Let $\Omega$ be a bounded, non empty open subset of $\mathbb{R}^2$ with a regular boundary. We assume that $\Omega$ contains a rigid body and  an incompressible viscous fluid. 
At each time $t>0$, the domain of the rigid body is denoted by $\mathcal{S}(t)\subset \Omega$ that is assumed to be compact with non empty interior and regular. The fluid domain is denoted by $\mathcal{F}(t)=\Omega\backslash \mathcal{S}(t)$, and is assumed to be connected.

We consider the following system describing the evolution of the fluid which is governed by the incompressible Navier-Stokes system
\begin{equation}
\label{NSE}
\left\{
\begin{array}{rl}
\partial_t U+(U\cdot \nabla)U-\nabla \cdot \mathbb{T}(U,P)=v^*1_{\mathcal{O}}  & \text { in }(0,T)\times \mathcal{F}(t),\\
\nabla \cdot U=0& \text { in }(0,T)\times \mathcal{F}(t).\\
\end{array}
\right.
\end{equation}
In the above system, we have denoted by $U$ the fluid velocity, $P$ the fluid pressure and by $v^*$ the control acting on the system through $\mathcal{O}\subset \mathbb{R}^2$, where $\mathcal{O}$ is a non empty open subset such that $\overline{\mathcal{O}}\subset \mathcal{F}(t)$.

The Cauchy stress tensor $\mathbb{T}(U,P)$ is defined by
$$ \mathbb{T}(U,P)=-P I_2 +2 \nu D(U),\quad D(U)_{i,j}=\frac{1}{2}\left( \frac{\partial  U_i}{\partial x_j}+\frac{\partial  U_j}{\partial x_i}\right),$$
where $\nu$ is the viscosity of the fluid. We denote for each time $t$, the position of the structure by $h(t)\in \mathbb{R}^2$ and by $R_{\theta(t)}$ the rotation matrix of angle $\theta$ of the solid defined by
$$
R_{\theta(t)}=
\begin{pmatrix}
\cos \theta(t) & -\sin\theta(t)\\
\sin\theta(t) & \cos\theta(t)
\end{pmatrix}.
$$
Then, the flow of the structure is given by $X_\mathcal{S}(t,\cdot):\mathcal{S}\longrightarrow \mathcal{S}(t)$ where
\begin{equation}
\label{strucflow}
X_\mathcal{S}(t,y)= h(t)+R_{\theta(t)}y,\quad t\in(0,T),\quad y\in \mathcal{S},
\end{equation}
where $\mathcal{S}$ is a fixed subset of $\mathbb{R}^2$, non empty, compact with a regular boundary.

We notice that $X_\mathcal{S}(t,\cdot)$ is invertible and a $C^\infty$-diffeomorphism, we denote its inverse by $Y_\mathcal{S}(t,\cdot):\mathcal{S}(t)\longrightarrow \mathcal{S}$ where
\begin{equation*}
Y_\mathcal{S}(t,x)= R_{\theta(t)}^{-1}(x-h(t)),\quad t\in(0,T),\quad x\in \mathcal{S}(t).
\end{equation*}
Thus, the Eulerian velocity of the structure is given by
$$
U_\mathcal{S}(t,x)=h'(t)+R_{\theta(t)}'R_{\theta(t)}^{-1}(x-h(t)),\quad t\in(0,T),\quad x\in \mathcal{S}(t).
$$
We denote by $a^\perp$, the vector $\begin{pmatrix}
-a_2\\a_1
\end{pmatrix} $, for any $a=\begin{pmatrix}
a_1\\a_2
\end{pmatrix}\in \mathbb{R}^2 $.
We notice that $R'_{\theta(t)}R_{\theta(t)}^{-1}$ is a skew-symmetric matrix, then the Eulerian velocity of the structure writes 
$$
U_\mathcal{S}(t,x)=h'(t)+\omega(t)(x-h(t))^\perp,
$$
where $\omega(t)=\theta'(t)$ represents the angular velocity of the rigid body.

We denote by $\mathcal{S}_{h,\theta}$ the set 
$$
\mathcal{S}_{h,\theta}=h+R_\theta \mathcal{S},
$$
and we define the corresponding fluid domain
$$
\mathcal{F}_{h,\theta}=\Omega\backslash \mathcal{S}_{h,\theta},
$$
for any $h\in\mathbb{R}^2$, $\theta \in \mathbb{R}$. Then, with these notations, we have
$$
\mathcal{S}(t)=\mathcal{S}_{h(t),\theta(t)},\quad \mathcal{F}(t)=\mathcal{F}_{h(t),\theta(t)}.
$$
We point out that the fluid domain is depending on the displacement of the solid structure, consequently, it depends on time.
 
We denote by $\widehat{n}$ the outward unit normal to $\partial\mathcal{F}(t)$, where $\partial\mathcal{F}(t)=\partial\Omega\cup \partial \mathcal{S}(t)$. 

The motion of the structure is governed by the balance equations for linear and angular momenta
\begin{equation}
\label{RBE}
\left\{
\begin{array}{rl}
mh''(t)=-\int_{\partial \mathcal{S}(t)}\mathbb{T}(U,P)\widehat{n} \ d\Gamma  & t\in(0,T),\\
J\omega'(t)=-\int_{\partial \mathcal{S}(t)}(x-h(t))^\perp\cdot\mathbb{T}(U,P)\widehat{n} \ d\Gamma& t\in(0,T).\\
\end{array}
\right.
\end{equation}

We complete \eqref{NSE} and \eqref{RBE} by the Navier slip boundary conditions. 
In order to write these boundary conditions, we need to introduce some notations. 
Let $\widehat{\tau}$ be a tangent vector to $\partial\mathcal{F}(t)$. We denote by $a_{\widehat{n}}$ and $a_{\widehat{\tau}}$ the normal and the tangential parts of $a\in \mathbb{R}^2$:
\begin{equation*}
\label{tak0.1}
a_{\widehat{n}}= (a\cdot \widehat{n})\widehat{n}, \quad a_{\widehat{\tau}} =a-a_{\widehat{n}}.
\end{equation*}
Then, the boundary conditions write as follows
\begin{equation}
\label{NBC}
\left\{
\begin{array}{rl}
U_{\widehat{n}}= 0  &  \text { on }(0,T)\times\partial\Omega,\\
\left[ 2\nu D(U) \widehat{n}+\beta_{\Omega} U\right]_{\widehat{\tau}}=0&  \text { on }(0,T)\times\partial\Omega,\\
(U-U_\mathcal{S})_{\widehat{n}}=0&   \text { on }(0,T)\times\partial \mathcal{S}(t),\\
\left[ 2\nu D(U) \widehat{n} + \beta_{\mathcal{S}} \left(U-U_\mathcal{S} \right)\right]_{\widehat{\tau}}=0 &  \text { on }(0,T)\times\partial \mathcal{S}(t),
\end{array}
\right.
\end{equation}
where $\beta_\Omega\geq 0$ and $\beta_{\mathcal{S}}\geq 0$ are the friction coefficients.

Let $h^0,\,\widetilde{\ell^0}\in\mathbb{R}^2$, $\theta^0,\;\omega^0\in \mathbb{R}$ and $u^0\in [H^1(\mathcal{F}_{h^0,\theta^0})]^2$. We furnish the following initial conditions
\begin{equation}
\label{IC}
U(0,x) = u^0(x),\; x\in \mathcal{F}_{h^0,\theta^0}, \quad h'(0)=\widetilde{\ell^0},\quad \omega(0)=\omega^0,\quad
h(0)=h^0,\quad \theta(0)=\theta^0,
\end{equation} 
such that the following compatibility conditions are satisfied
\begin{equation}
\label{CC}
\left\{
\begin{array}{rl}
\nabla \cdot u^0=0&\;\text { in }\mathcal{F}_{h^0,\theta^0}, \\
 u^0_{\widehat{n}}=0& \text { on }\partial\Omega,\\  
\left( u^0-u^0_\mathcal{S}\right) _{\widehat{n}}=0& \text{ on } \partial \mathcal{S}_{h^0,\theta^0},
\end{array}
\right.
\end{equation}
where $u^0_\mathcal{S}(x)=\widetilde{\ell^0}+\omega^0(x-h^0)^\perp$.
Without loss of generality, we assume that the center of gravity of $\mathcal{S}$ is at the origin. Then, $h(t)$ will be the position of the center of mass of the rigid body $\mathcal{S}(t)$.

Our main objective in this paper is to look for a control $v^*$ acting on $\mathcal{O}$ such that for any $(h_T,\theta_T)\in\mathbb{R}^2\times \mathbb{R}$
with
\begin{equation}
\label{non-collisiont}
\mathcal{S}_{h_T,\theta_T}\subset\Omega\; \backslash \;\overline{\mathcal{O}},
\end{equation}
we get that $h(T)=h_T$, $\theta(T)=\theta_T$ and the velocities of the fluid and of the rigid body are equal to $0$ at time $T$.

The main result of this paper is stated below:
\begin{Theorem}
\label{tp}
Assume that $\beta_\mathcal{S}>0$ and let $(h_T,\theta_T)$ that satisfies \eqref{non-collisiont}. Then, there exists $\varepsilon>0$ such that for any $(u^0,h^0,\widetilde{\ell^0},\omega^0,\theta^0)$ that satisfies \eqref{CC} and
\begin{equation*}
\label{petitesse}
\left\|u^0\right\|_{[H^1(\mathcal{F}_{h^0,\theta^0})]^2}+\left|h^0-h_T\right|+\left|\widetilde{\ell^0}\right|+\left|\omega^0\right|+\left|\theta^0-\theta_T\right|\leq \varepsilon,
\end{equation*}
there exists a control $v^*\in L^2(0,T;[L^2(\mathcal{O})]^2)$ such that 
$$
U(T,\cdot)=0\text{ in }\mathcal{F}_{h_T,\theta_T},\quad h(T)=h_T,\quad h'(T)=0,\quad \omega(T)=0,\quad \theta(T)=\theta_T.
$$
\end{Theorem}
 Without loss of generality, we can always assume that
$$
h_T=0, \quad \theta_T=0, \quad\text{and thus}\quad \mathcal{S}_{h_T,\theta_T}=\mathcal{S}, \quad \mathcal{F}_{h_T,\theta_T}=\mathcal{F}.
$$	
In fact: in general, we have 
\begin{equation*}
X_\mathcal{S}(t,y)= h(t)+R_{\theta(t)-\theta_T}(y-h_T),\quad t\in(0,T),\quad y\in \mathcal{S}.
\end{equation*}
and in this case, we set
$$
\mathcal{S}_{h,\theta}=h+R_{\theta-\theta_T} (\mathcal{S}-h_T),
$$
then, $\mathcal{S}_{h_T,\theta_T}=\mathcal{S}$.
Let take $z=R_{-\theta_T}(y-h_T)$, hence
\begin{equation*}
X_\mathcal{S}(t,z)= h(t)+R_{\theta(t)}z,\quad t\in(0,T),\quad z\in\mathcal{S}_{0,0} ,
\end{equation*}
where
$$
\mathcal{S}_{0,0}=R_{-\theta_T}(\mathcal{S}-h_T).
$$
Then, we notice that we are reduced to the case \eqref{strucflow}. 
Thus, by translation of vector $-h_T$ and rotation of angle $-\theta_T$, one can reduce the controllability problem to the case $h_T=0$ and $\theta_T=0$. 
In what follows, the vectors $n$ and $\tau$ stand respectively for the outer unit normal and the unit tangent vector to $\partial\mathcal{F}$.

Several works were devoted to the study of fluid-rigid body interaction systems, in particular, when the fluid is governed  by the  Navier-Stokes system. Existence results concerning this kind of systems with Dirichlet boundary conditions were considered in \cite{MR1759801}, \cite{MR1682663}, \cite{MR1765138},\cite{MR2029294}, \cite{MR2027754}, \cite{MR1763528}, \cite{MR1781915}, \cite{MR1870954} etc.
For the case of the Navier slip boundary conditions \eqref{NBC}, the existence of weak solutions is proved in \cite{MR3272367} and the existence of strong solutions is obtained in \cite{MR3266142}.
In \cite{MR3266142} and \cite{MR3281946}, the authors proved that collisions can occur in final time between the rigid body and the domain cavity with some assumptions on the solid geometry.

Concerning the controllability, let us mention \cite{MR2103189} and \cite{MR1804497}, where the authors obtained the local exact controllability of the 2D or 3D Navier-Stokes equations with Dirichlet boundary conditions considering distributed controls. The local exact controllability of the Navier Stokes system with nonlinear Navier boundary conditions with distributed controls was studied in \cite{MR2224824}. Moreover, in \cite{MR3784805}, the authors established the local controllability with $N-1$ scalar controls. 
 With Navier-slip conditions on the fluid equations, global null controllability is obtained for the weak solution in \cite{MR4081730} such that the controls are only located on a small part of the domain boundary.
Concerning controllability results of fluid-structure systems with Dirichlet boundary conditions, in dimension 2, we mention the paper \cite{MO}, where the authors proved the null controllability in velocity and the exact controllability for the position of the rigid body assuming some geometric properties for the solid and  provided that the initial conditions are small enough, more precisely a condition of smallness on the $H^3$ norm of the initial fluid velocity is needed. The authors used the Kakutani's fixed point theorem to deduce the null controllability of the nonlinear system. We have also the paper \cite{MR2317341} where the authors considered the structure of a rigid ball, their result relies on semigroup theory. In the latest paper, only an assumption on the $H^1$ norm of the initial fluid velocity is needed. In dimension 3, we mention \cite{MG}, the same result was proved without any assumptions on the solid geometry while a condition of smallness on the $H^2$ norm of the initial fluid velocity is needed. We also mention \cite{royT}, where the authors considered the interaction between a viscous and incompressible fluid modeled by the Boussinesq system and a rigid body with arbitrary shape, they proved null controllability of the associated system.
In the case of the stabilization of fluid-solid ineraction systems, we have \cite{MR3181675}, \cite{MR3261920}.

In this paper, we prove the local null controllability of the system  \eqref{NSE}, \eqref{RBE}, \eqref{NBC}, \eqref{IC}, that is the case of the Navier slip boundary conditions  in the presence of a rigid structure of arbitrary shape. We follow the same method as \cite{MR2317341}: we use a change of variables to write our system in a fixed domain and use a fixed point argument to reduce our problem to the null controllability of a linear fluid-rigid body system, that is coupling the Stokes system with ODE for the structure velocity. To do this we derive a Carleman estimates for the corresponding system.

One of the main difficulties to obtain such an estimate is to manage the boundary conditions and more precisely to obtain estimates of the rigid velocity with the good weights. An important step for this calculation is a Carleman estimates for the Laplacian equation with divergence free condition and Navier slip boundary conditions, which is given in section \ref{cestim}.
We emphasize that this is the first result concerning the null controllability of a fluid-structure interaction system with boundary conditions different from the standard no-slip ones. Note that with the Navier boundary conditions considered here, one of the additional difficulties with respect to the Dirichlet boundary conditions lies on the fact that in the Carleman estimate, it is more complictaed to estimate the structure velocities from the fluid velocity. There are several possible extensions to this work. First let us recall that in  \cite{MR4081730}, the authors obtain the global
 exact controllability of the Navier-Stokes system with Navier boundary conditions. One of their ingredients is to use the local exact null controllability of \cite{MR2224824}. Here, one can also consider the global exact controllability but the arguments of \cite{MR4081730} may be difficult to adapt due to the presence of the structure.
Second, one can also consider a heat conducting fluid and remplace the Navier-Stokes system by the Boussinesq system. This has been done for instance in 
\cite{royT} with a rigid body and Dirichlet boundary conditions. Our method here should be adapted to this case and we would obtain a similar result.
Finally, one can try to reduce the number of controls as it is done in \cite{MR3784805} for the Navier-Stokes system with Navier boundary conditions. However, let us note that due to the presence of the structure velocities in the boundary conditions, some parts of the proof in \cite{MR3784805}, might be difficult to adapt, mainly the manipulation of the curl of the fluid velocity on the boundary.

The outline of this paper is as follows:
in section \ref{notation}, we give some preliminaries. We emphasize that one of the main difficulties in this problem is that we are dealing with a coupled system set on a non cylindrical domain. Then, in section \ref{chg}, we remap the problem into an equivalent system given in a fixed geometry. In section \ref{carlm}, we establish a new carleman inequality. In section \ref{contrl}, we prove the null controllability of the linearized system. Finally, in section \ref{fip}, we prove Theorem \ref{tp} and deduce the null controllability of the system by applying a fixed-point argument.

\section{Preliminaries}
\label{notation}
In this section, we prove some regularity results of an associated linearized problem. 
We consider the following linear system 
\begin{equation}
\label{Sl1}
\left\{
\begin{array}{rl}
\partial_t \overline{w}-\nabla \cdot \mathbb{T}(\overline{w},\overline{\pi})=F_1  & \text { in }(0,T)\times\mathcal{F},\\
\nabla \cdot \overline{w}=0& \text { in }(0,T)\times\mathcal{F},\\
m\ell_{\overline{w}}'(t)=-\int_{\partial \mathcal{S}}\mathbb{T}( \overline{w}, \overline{\pi})n\ d\Gamma+F_2  & t\in(0,T),\\
J(k_{\overline{w}}'(t))=-\int_{\partial \mathcal{S}}y^\perp\cdot\mathbb{T}( \overline{w}, \overline{\pi})n \ d\Gamma+F_3& t\in(0,T),\\
\end{array}
\right.
\end{equation}
with the boundary conditions
\begin{equation}
\label{SNBCl1}
\left\{
\begin{array}{rl}
\overline{w}_n = 0  &  \text { on }(0,T)\times\partial\Omega,\\
\left[ 2\nu D( \overline{w}) n+\beta_{\Omega}  \overline{w}\right]_{\tau}=0&  \text { on }(0,T)\times\partial\Omega,\\
( \overline{w}- \overline{w}_\mathcal{S})_n=0&   \text { on }(0,T)\times\partial \mathcal{S},\\
\left[ 2\nu D( \overline{w}) n + \beta_{\mathcal{S}} ( \overline{w}- \overline{w}_\mathcal{S})\right]_{\tau}=0 &  \text { on }(0,T)\times\partial \mathcal{S},
\end{array}
\right.
\end{equation}
where $ \overline{w}_\mathcal{S}(y)=\ell_{\overline{w}}+k_{\overline{w}} y^\perp$, completed with the initial conditions
\begin{equation}
\label{SIl1}
\overline{w}(0,\cdot) = w^0,\; \text{ in }\mathcal{F},\quad
\ell_{\overline{w}}(0)=\ell_w^0,\quad k_{\overline{w}}(0)=k_w^0.
\end{equation}
We have the following regularity result for the system \eqref{Sl1}, \eqref{SNBCl1} and \eqref{SIl1} which is proved in \cite{MR3266142}.
\begin{Theorem}
\label{regu}
Let $T>0$. Suppose that $F_1\in L^2(0,T;[L^2(\mathcal{F})]^2)$, $F_2\in [L^2(0,T)]^2$ and $F_3\in L^2(0,T)$ are given functions and $ w^0\in[H^1(\mathcal{F})]^2$ such that
\begin{equation*}
\nabla\cdot w^0=0, \;  \text{ in }\mathcal{F},\quad
w^0_n = 0,  \; \text{ on }\partial\Omega,\quad
w^0_n(y)=(\ell_w^0+k_w^0 y^\perp)_n,  \; y\in \partial \mathcal{S}.
\end{equation*} Then, there exists a unique solution to problem \eqref{Sl1}, \eqref{SNBCl1} and \eqref{SIl1} such that
\begin{equation*}
\overline{w} \in L^2(0,T;[H^2(\mathcal{F})]^2)\cap C([0,T];[H^1(\mathcal{F})]^2 )\cap H^1(0,T;[L^2(\mathcal{F})]^2),\quad  \overline{\pi} \in L^2(0,T;H^1(\mathcal{F}) \slash \mathbb{R}),
\end{equation*}
$$
(\ell_{\overline{w}},k_{\overline{w}})\in [H^1(0,T)]^2\times H^1(0,T).
$$
Moreover, it satisfies the following estimate
\begin{multline}
\label{reg1}
\left\| \overline{w}\right\|_{L^2(0,T;[H^2(\mathcal{F})]^2)\cap C([0,T];[H^1(\mathcal{F})]^2)\cap H^1(0,T;[L^2(\mathcal{F})]^2)}+ \left\| \overline{\pi}\right\|_{L^2(0,T;H^1(\mathcal{F}) \slash \mathbb{R})}+\left\|\ell_{\overline{w}}\right\|_{[H^1(0,T)]^2}+\left\|k_{\overline{w}}\right\|_{ H^1(0,T)} \\\leq C\left( \left\|F_1\right\|_{L^2(0,T;[L^2(\mathcal{F})]^2)}+\left\|F_2\right\|_{[L^2(0,T)]^2}+\left\|F_3\right\|_{L^2(0,T)}+\left\| w^0\right\|_{[H^1(\mathcal{F})]^2}+\left|\ell^0_w \right|+\left|k_w^0\right|  \right). 
\end{multline}
\end{Theorem}
\begin{proof}
The proof of the above theorem is based on semigroup theory. For the sake of completeness, we just recall the main ideas of the proof.
 
We note that $w^0$ and $\overline{w}$ are extended by $\ell_w^0+k_w^0y^\perp$ and $\ell_{\overline{w}}+k_{\overline{w}}y^\perp$ on $\mathcal{S}$ respectively.
Let define the following Hilbert spaces
\begin{equation*}
\mathbb{H}=\left\lbrace w\in [L^2(\Omega)]^2\;|\; \nabla\cdot w=0,\quad w_n=0,\text{ on }\partial\Omega,\quad D(w)=0\text{ in }\mathcal{S} \right\rbrace,
\end{equation*}
\begin{equation*}
\mathbb{V}=\left\lbrace w\in \mathbb{H}\;|\;  w|_\mathcal{F}\in [H^1(\mathcal{F})]^2 \right\rbrace.
\end{equation*}
We notice that the condition $D(w)=0$ on $\mathcal{S}$ is equivalent to $w=w_\mathcal{S}\in \mathcal{R}$ on $\mathcal{S}$ where 
$$
\mathcal{R}=\left\lbrace v\in\mathbb{R}^2\;|\; \text{there exist}\;\ell_v\in\mathbb{R}^2,\; k_v\in\mathbb{R}\;\text{such that}\; v(y)=\ell_v+k_v y^\perp\right\rbrace .
$$
For $w,v\in\mathbb{H}$, we define the inner product on $\mathbb{H}$ by
$$
\left\langle w,v \right\rangle = \int_\mathcal{F}w\cdot v \ dy+m\ell_v\cdot\ell_w+Jk_vk_w.
$$
Let define also the orthogonal projector $\mathbb{P}: [L^2(\Omega)]^2\longrightarrow \mathbb{H}$.

The system \eqref{Sl1}, \eqref{SNBCl1} and \eqref{SIl1} can be reduced to the following form
\begin{equation}
\overline{w} '=A \overline{w}+F,\quad  \overline{w}(0)=w^0,
\end{equation}
where the operator $A$ is defined by 
\begin{equation}
\mathcal{A}w=
\left\{
\begin{array}{r}
\nu\Delta w\quad\text{in}\quad\mathcal{F}, \\
-\frac{2\nu}{m}\int_{\partial \mathcal{S}} D(w)n \ d\Gamma-\left( \frac{2\nu }{J}\left[ \int_{\partial \mathcal{S}}y^\perp\cdot D(w)n \ d\Gamma\right] \right) y^\perp  \quad \text{on}\quad \partial \mathcal{S},
\end{array}
\right.
\end{equation}
\begin{multline*}
A=\mathbb{P}\mathcal{A},\quad \mathcal{D}(A)=\mathcal{D}(\mathcal{A})=\{ w\in \mathbb{V},\;w|_{\mathcal{F}}\in [H^2(\mathcal{F})]^2,\;\left[ 2\nu D(w) n+\beta_{\Omega} w\right]_{\tau}=0,\ \text{  on }\partial\Omega,\\
\left[ 2\nu D(w) n + \beta_{\mathcal{S}} (w-w_\mathcal{S} )\right]_{\tau}=0,  \ \text{ on } \partial \mathcal{S}  \},
\end{multline*}
and
$$
F=\mathbb{P}\left( F_11_{\mathcal{F}}+\left( \frac{F_2}{m}+\frac{F_3y^\perp}{J}\right) 1_{\mathcal{S}}\right).
$$ 
In \cite[Lemma 3.1]{MR3266142}, it is proved that the operator $A$ is self-adjoint and it generates a semigroup of contractions on $\mathbb{H}$. Thus, we deduce Theorem \ref{regu} (see \cite[Proposition 3.3]{MR3266142}).
\end{proof}
We note here that since $A$ is a self-adjoint operator, then for any $w\in \mathcal{D}(A)$, we have $$\left\|w\right\|_{\mathcal{D}((-A)^{1/2})}^2=\left\langle w,w \right\rangle +\left\langle w,Aw \right\rangle.$$
We also need some regularity results on the linear system \eqref{Sl1}, \eqref{SNBCl1}, \eqref{SIl1}.

Let $w^0\in [H^3(\mathcal{F})]^2$, $\pi^0\in H^2(\mathcal{F})$, $(\ell_w^0,k_w^0)\in\mathbb{R}^3$ and we set $w^1\in [H^1(\mathcal{F})]^2$, $(\ell_w^1,k_w^1)\in\mathbb{R}^3$ such that
$$
w^1=\nabla\cdot\mathbb{T}(w^0,\pi^0)+F_1(0,\cdot) 
$$
$$
\ell_w^1=-\frac{1}{m}\int_{\partial \mathcal{S}}\mathbb{T}( w^0, \pi^0)n\ d\Gamma+\frac{1}{m}F_2(0) , 
$$
and 
$$
k_w^1=-\frac{1}{J}\int_{\partial \mathcal{S}}y^\perp\cdot\mathbb{T}( w^0, \pi^0)n \ d\Gamma+\frac{1}{J}F_3(0).
$$

Moreover, we suppose that $\pi^0$ satisfies the following system
\begin{equation*}
	\left\{ 
	\begin{array}{cc}
		\Delta \pi^0=\nabla\cdot F_1(0,\cdot) & \text{ in } \mathcal{F},\\
		\frac{\partial\pi^0}{\partial n}=(\nu \Delta w^0+F_1(0,\cdot))\cdot n- (\ell^1_w+k^1_wy^\perp)\cdot n1_{\partial\mathcal{S}}  & \text { on }\partial\mathcal{F}.
	\end{array}
	\right.
\end{equation*}
By the Lax-Milgram theorem, the above system admits a unique solution $\pi^0\in H^1(\mathcal{F})\slash \mathbb{R}$, such that
\begin{multline*}
\left\|\nabla\pi^0\right\|_{[L^2(\mathcal{F})]^2}+\left|\ell^1_w \right|+\left|k_w^1\right|\leq C \bigg(\left\| w^0\right\|_{[H^3(\mathcal{F})]^2}+\left\|F_1\right\|_{L^2(0,T;[H^2(\mathcal{F})]^2)\cap H^1(0,T;[L^2(\mathcal{F})]^2)}\\+\left\|F_2\right\|_{[H^1(0,T)]^2}+\left\|F_3\right\|_{H^1(0,T)}\bigg).  
\end{multline*}
Since $\partial \mathcal{F}$ is sufficiently regular, we get that 
$$(\nu \Delta w^0+F_1(0,\cdot))\cdot n- (\ell^1_w+k^1_wy^\perp)\cdot n1_{\partial\mathcal{S}}\in H^{1/2}(\partial\mathcal{F})
.$$  
Then, using classical elliptic estimate of the Neumann system, we obtain 
\begin{multline*}
	\left\|\pi^0\right\|_{H^2(\mathcal{F})}+\left|\ell^1_w \right|+\left|k_w^1\right|\leq C \bigg(\left\| w^0\right\|_{[H^3(\mathcal{F})]^2}+\left\|F_1\right\|_{L^2(0,T;[H^2(\mathcal{F})]^2)\cap H^1(0,T;[L^2(\mathcal{F})]^2)}\\+\left\|F_2\right\|_{[H^1(0,T)]^2}+\left\|F_3\right\|_{H^1(0,T)}   \bigg).    
\end{multline*}
We set the compatibility conditions
\begin{equation}
\label{comp}
\nabla\cdot w^1=0, \;  \text{ in }\mathcal{F},\quad
w^1_n = 0,  \; \text{ on }\partial\Omega,\quad
w^1_n(y)=(\ell_w^1+k_w^1 y^\perp)_n,  \; y\in \partial \mathcal{S}.
\end{equation}
\begin{Proposition}
\label{prop}
Let $T>0$. Suppose that $F_1\in L^2(0,T;[H^2(\mathcal{F})]^2)\cap H^1(0,T;[L^2(\mathcal{F})]^2)$, $F_2\in [H^1(0,T)]^2$, 
$F_3\in H^1(0,T)$ and let $(w^0, \ell^0_w, k_w^0)\in [H^3(\mathcal{F})]^2\times\mathbb{R}^2\times\mathbb{R}$ such that the compatibility conditions \eqref{comp} hold.
Then, there exists a unique solution to the problem \eqref{Sl1}, \eqref{SNBCl1} and \eqref{SIl1} such that
$$
\overline{w} \in L^2(0,T;[H^4(\mathcal{F})]^2)\cap H^2(0,T;[L^2(\mathcal{F})]^2),\quad  \overline{\pi} \in L^2(0,T;H^3(\mathcal{F}))\cap H^1(0,T;H^1(\mathcal{F})),
$$
$$
(\ell_{\overline{w}},k_{\overline{w}})\in [H^2(0,T)]^2\times H^2(0,T).
$$
Moreover, it satisfies the following estimate
\begin{multline}
\label{reg2}
\left\| \overline{w}\right\|_{L^2(0,T;[H^4(\mathcal{F})]^2)\cap H^2(0,T;[L^2(\mathcal{F})]^2)}+ \left\| \overline{\pi}\right\|_{L^2(0,T;H^3(\mathcal{F}))\cap H^1(0,T;H^1(\mathcal{F}))}+\left\|\ell_{\overline{w}}\right\|_{[H^2(0,T)]^2}+\left\|k_{\overline{w}}\right\|_{H^2(0,T)} \\\leq C\bigg( \left\|F_1\right\|_{L^2(0,T;[H^2(\mathcal{F})]^2)\cap H^1(0,T;[L^2(\mathcal{F})]^2)}+\left\|F_2\right\|_{[H^1(0,T)]^2}+\left\|F_3\right\|_{H^1(0,T)}\\+\left\| w^0\right\|_{[H^3(\mathcal{F})]^2}+\left|\ell^0_w \right|+\left|k_w^0\right| \bigg). 
\end{multline}
\end{Proposition}
\begin{proof}
we differentiate in time the system \eqref{Sl1}, \eqref{SNBCl1}.  we get
\begin{equation}
\label{Sl}
\left\{
\begin{array}{rl}
\partial_{tt} \overline{w}-\nabla \cdot \mathbb{T}(\partial_t\overline{w},\partial_t\overline{\pi})=\partial_tF_1  & \text { in }(0,T)\times\mathcal{F},\\
\nabla \cdot (\partial_t\overline{w})=0& \text { in }(0,T)\times\mathcal{F},\\
m\ell_{\overline{w}}''(t)=-\int_{\partial \mathcal{S}}\mathbb{T}( \partial_t\overline{w}, \partial_t\overline{\pi})n\ d\Gamma+\partial_tF_2  & t\in(0,T),\\
J(k_{\overline{w}}''(t))=-\int_{\partial \mathcal{S}}y^\perp\cdot\mathbb{T}( \partial_t\overline{w}, \partial_t\overline{\pi})n \ d\Gamma+\partial_tF_3& t\in(0,T),\\
\end{array}
\right.
\end{equation}
with the boundary conditions
\begin{equation}
\label{Cl1}
\left\{
\begin{array}{rl}
\partial_t\overline{w}_n = 0  &  \text { on }(0,T)\times\partial\Omega,\\
\left[ 2\nu D(\partial_t\overline{w}) n+\beta_{\Omega}  \partial_t\overline{w}\right]_{\tau}=0&  \text { on }(0,T)\times\partial\Omega,\\
( \partial_t\overline{w}- \overline{w}'_\mathcal{S})_n=0&   \text { on }(0,T)\times\partial \mathcal{S},\\
\left[ 2\nu D(\partial_t\overline{w}) n + \beta_{\mathcal{S}} ( \partial_t\overline{w}- \overline{w}'_\mathcal{S})\right]_{\tau}=0 &  \text { on }(0,T)\times\partial \mathcal{S},
\end{array}
\right.
\end{equation}
where $ \overline{w}'_\mathcal{S}=\ell'_{\overline{w}}+k'_{\overline{w}} y^\perp$
with the initial conditions 
\begin{equation}
\label{il1}
\partial_t\overline{w}(0,\cdot) = w^1,\; \text{ in }\mathcal{F},\quad
\ell'_{\overline{w}}(0)=\ell_w^1,\quad k'_{\overline{w}}(0)=k_w^1,\quad \ell_{\overline{w}}(0)=\ell_w^0,\quad k_{\overline{w}}(0)=k_w^0.
\end{equation}
Since \eqref{comp} is satisfied, we can apply Theorem \ref{regu}, we get
\begin{multline}
\label{eq16}
\left\| \partial_t\overline{w}\right\|_{L^2(0,T;[H^2(\mathcal{F})]^2)\cap H^1(0,T;[L^2(\mathcal{F})]^2)}+ \left\| \partial_t\overline{\pi}\right\|_{L^2(0,T;H^1(\mathcal{F}))}+\left\|\ell'_{\overline{w}}\right\|_{[H^1(0,T)]^2}+\left\|k'_{\overline{w}}\right\|_{H^1(0,T)} \\\leq C\bigg( \left\|F_1\right\|_{L^2(0,T;[H^2(\mathcal{F})]^2)\cap H^1(0,T;[L^2(\mathcal{F})]^2)}+\left\|F_2\right\|_{[H^1(0,T)]^2}+\left\|F_3\right\|_{H^1(0,T)}+\left\|w^0 \right\|_{[H^3(\mathcal{F})]^2}\bigg). 
\end{multline}
We obtain from \eqref{eq16} that $\ell_{\overline{w}}\in [H^2(0,T)]^2$ and $k_{\overline{w}}\in H^2(0,T)$.
Then, using the regularity results for the unstationary Stokes system with Navier boundary conditions proved in \cite{GS}, combined  with \eqref{eq16} and \eqref{reg1}, we get
\begin{multline*}
\left\|\overline{w}\right\|_{L^2(0,T;[H^4(\mathcal{F})]^2)\cap H^2(0,T;[L^2(\mathcal{F})]^2)}+ \left\| \overline{\pi}\right\|_{L^2(0,T;H^3(\mathcal{F}))\cap H^1(0,T;H^1(\mathcal{F}))}+\left\|\ell_{\overline{w}}\right\|_{[H^2(0,T)]^2}+\left\|k_{\overline{w}}\right\|_{H^2(0,T)} \\\leq C\bigg( \left\|F_1\right\|_{L^2(0,T;[H^2(\mathcal{F})]^2)\cap H^1(0,T;[L^2(\mathcal{F})]^2)}+\left\|F_2\right\|_{[H^1(0,T)]^2}+\left\|F_3\right\|_{H^1(0,T)}\\+\left\|w^0 \right\|_{[H^3(\mathcal{F})]^2}+\left|\ell^0_w\right|+\left|k^0_w\right|\bigg). 
\end{multline*}
Then, we obtain \eqref{reg2}.
\end{proof}
\section{Change of variables}
\label{chg}
To treat the free boundary problem \eqref{NSE}, \eqref{RBE}, \eqref{NBC}, \eqref{IC}, we consider an equivalent system written in a fixed domain using a change of variables that was already introduced in \cite{MR2029294}. In fact, we construct an extension of the structure flow \eqref{strucflow} over $\Omega$ by a regular and incompressible flow. First, we need to control the distance between the structure and the boundary $\partial(\Omega\backslash\overline{\mathcal{O}})$. 

The condition \eqref{non-collisiont} implies that there exists $d>0$ such that 
\begin{equation*}
d(\partial (\Omega\backslash\overline{\mathcal{O}}), \mathcal{S})= d.
\end{equation*}
Let $t\in [0,T]$. We have for a fixed $y\in \mathcal{S}$
\begin{equation*}
d(y,\mathcal{S}(t))\leq \left|h(t)+R_{\theta(t)}y-y\right|.
\end{equation*}
Then, we get
\begin{equation}
d(\mathcal{S}(t),\mathcal{S})\leq C\left( \left|h(t)\right| +\left| R_{\theta(t)}-I_2\right|\right).
\end{equation}
If 
\begin{equation}
\label{smallness}
\left|h(t)\right|+\left| R_{\theta(t)}-I_2\right|\leq \frac{d}{2C},
\end{equation} we get
\begin{equation*}
d(\mathcal{S}(t),\mathcal{S})\leq \frac{d}{2}.
\end{equation*}
Thus,
we obtain
$$
d(\partial (\Omega\backslash \overline{\mathcal{O}}), \mathcal{S}(t))\geq  \frac{d}{2},\quad t\in[0,T].
$$
In other words, we only assume that no collision occurs between the structure and the boundary $\partial (\Omega\backslash \overline{\mathcal{O}})$ at time $T$. In fact, if the initial data are small enough, then the displacement of the structure remains small, then \eqref{smallness} is satisfied. Thus, no contact can occur between the solid and the boundary for any $t\in [0,T]$.

Following \cite{MR2029294}, we can construct a change of variables $X$ and $Y$ with the following properties
\begin{itemize}
\item For any $t\in [0,T]$, $X$ and $Y$ are $C^\infty$ diffeomorphisms from $\Omega$ into itself,
\item The function $X$ is invertible of inverse $Y$,
\item In a neighborhood of $\mathcal{S}=\mathcal{S}(T)$, $X(t,y)=X_\mathcal{S}(t,y)=h(t)+R_{\theta(t)}y$,
\item  In a neighborhood of $\partial\Omega$ and of $\mathcal{O}$, $X(t,y)=y$,
\item $\det\nabla X (t,y)=1$, for all $y\in \Omega$,
\item In a neighborhood of $\mathcal{S}$, $\nabla X(t,y)=R_{\theta(t)}$ and $\nabla Y(t,X(t,y))=R_{\theta(t)}^{-1}$.

Moreover, we have
\begin{equation}
\label{im}
\left\|X\right\|_{H^2(0,T;[C^2(\overline{\Omega})]^2)}+ \left\|Y\right\|_{H^2(0,T;[C^2(\overline{\Omega})]^2)}\leq C (\left\|h\right\|_{[H^2(0,T)]^2}+\left\|\theta\right\|_{H^2(0,T)} ),
\end{equation}
where $C$ depends on $T$.
\end{itemize}
Now, we set 
\begin{equation*}
u(t,y)=\Cof(\nabla X(t,y))^*U(t,X(t,y)),\quad P(t,y)=p(t,X(t,y)).
\end{equation*}
Then, we have
$$
u_\mathcal{S}(t,y)=R_{\theta(t)}^{-1}(h'(t)+\omega(t)(x-h(t))^\perp)=\ell(t)+\omega(t)y^\perp,
$$
where $\ell(t)=R_{\theta(t)}^{-1}h'(t)$.
We transform the system \eqref{NSE}, \eqref{RBE}, \eqref{NBC}, \eqref{IC} by using this change of variables. Calculations of this type are already done in \cite[Lemma 2.1]{MR3266142}. 
Thus, the system \eqref{NSE}, \eqref{RBE}, \eqref{NBC}, \eqref{IC} is equivalent to
\begin{equation}
\label{sf1}
\left\{
\begin{array}{rl}
\partial_t u-\nu\mathcal{L}u+\mathcal{M}u+\mathcal{N}u+\mathcal{G}p= v^*1_{\mathcal{O}}  & \text { in }(0,T)\times\mathcal{F},\\
\nabla \cdot u=0& \text { in }(0,T)\times\mathcal{F},\\
\end{array}
\right.
\end{equation}
\begin{equation}
\label{sf2}
\left\{
\begin{array}{rl}
m\ell'(t)=-\int_{\partial \mathcal{S}}\mathbb{T}(u,p)n \ d\Gamma -m\omega \ell(t)^\perp & t\in(0,T),\\
J\omega'(t)=-\int_{\partial \mathcal{S}}y^\perp\cdot\mathbb{T}(u,p)n\ d\Gamma& t\in(0,T),\\
\end{array}
\right.
\end{equation}
\begin{equation}
\label{sf3}
\left\{
\begin{array}{rl}
u_{n}= 0  &  \text { on }(0,T)\times\partial\Omega,\\
\left[ 2\nu D(u)n+\beta_{\Omega} u\right]_{\tau}=0&  \text { on }(0,T)\times\partial\Omega,\\
(u-u_\mathcal{S})_{n}=0&   \text { on }(0,T)\times\partial \mathcal{S},\\
\left[ 2\nu D(u) n + \beta_{\mathcal{S}} \left(u-u_\mathcal{S} \right)\right]_{\tau}=0 &  \text { on }(0,T)\times\partial \mathcal{S},
\end{array}
\right.
\end{equation}
where $n$ and $\tau$ respectively stand for the normal and the tangential vectors on $\partial\mathcal{F}$, with 
\begin{equation*}
\left[ \mathcal{L}u\right]_i=\sum_{j,k}\frac{\partial}{\partial y_j}\left(g^{jk}\frac{\partial u_i}{\partial y_k}\right)+ 2\sum_{j,k,l} g^{kl}\Gamma^i_{jk}\frac{\partial u_j}{\partial y_l}+\sum_{j,k,l}\left(\frac{\partial}{\partial y_k}(g^{kl}\Gamma^i_{jl})+\sum_m g^{kl}\Gamma^m_{jl}\Gamma^i_{km} \right)u_j,   
\end{equation*}
\begin{equation*}
\left[ \mathcal{M}u\right]_i=\sum_{j}\frac{\partial Y_j}{\partial t}\frac{\partial u_i}{\partial y_j}+ \sum_{j,k}\left(  \Gamma^i_{jk}\frac{\partial Y_k}{\partial t}+\frac{\partial Y_i}{\partial x_k}\frac{\partial^2X_k}{\partial t\partial y_j} \right)u_j,   
\end{equation*}
\begin{equation*}
\left[ \mathcal{N}u\right]_i=\sum_{j}u_j\frac{\partial u_i}{\partial y_j}+\sum_{j,k} \Gamma^i_{jk}u_ju_k,   
\end{equation*}
\begin{equation*}
\left[ \mathcal{G}p\right]_i=\sum_{j}g^{ij}\frac{\partial p}{\partial y_j},  
\end{equation*}
where 
$$
g^{ij}=\sum_k \frac{\partial Y_i}{\partial x_k} \frac{\partial Y_j}{\partial x_k},\quad g_{ij}=\sum_k \frac{\partial X_i}{\partial y_k} \frac{\partial X_j}{\partial y_k},
$$
and 
$$
\Gamma^{k}_{ij}=\frac{1}{2}\sum _l g^{kl}\left( \frac{\partial g_{il}}{\partial y_j}+\frac{\partial g_{jl}}{\partial y_i}+\frac{\partial g_{ij}}{\partial y_l} \right). 
$$
Finally, we set the initial conditions for $y\in \mathcal{F}$

\begin{equation}
\label{sf4}
u(0,y)=\Cof(\nabla X(0,y))^*u^0(X(0,y))=u_0(y),\quad \ell(0)=R^{-1}_{\theta^0}\widetilde{\ell^0}=\ell^0,\quad \omega(0)=\omega^0.
\end{equation}

\section{Carleman estimate for the Laplacian problem with Navier slip boundary conditions}
\label{cestim}

We prove first, a Carleman inequality for the Laplacian problem with non-homogeneous Navier boundary conditions.
From \cite[Lemma 1.1]{CDI}, we can construct a function $\eta\in C^2(\overline{\mathcal{F}})$	such that 

\begin{equation}
\label{eta}
\eta>0 \text{ in } \mathcal{F},\quad \eta=0 \text{ on } \partial \mathcal{F},\quad  \left| \nabla \eta \right|>0  \text{ in } \overline{\mathcal{F} \; \backslash\; \mathcal{O}_\eta} \text{ where } \mathcal{O}_\eta \subset \subset \mathcal{O},\quad \nabla\eta\cdot n <0,\text{ on } \partial\mathcal{F}.
\end{equation}
Let $\lambda>0$ and let take $\alpha=e^{\lambda \eta}$. We have the following proposition.   
\begin{Proposition}
\label{CS}
Let $\mathcal{F}$ and $\mathcal{O}$ be two open sets such that $\mathcal{O}\subset \mathcal{F}$. Suppose that the friction coefficient $\beta$ is a positive constant, then there exist $C=C(\mathcal{F},\mathcal{O})>0$, $s_1$ and $\lambda_1$ where $s_1=s_1(\mathcal{F},\mathcal{O})$, $\lambda_1= \lambda_1(\mathcal{F},\mathcal{O})$,
such that the solution $\psi\in [H^2(\mathcal{F})]^2$ of the system
\begin{equation}
\label{SN}
\left\{
\begin{array}{rl}
-\Delta \psi=f&\text{in}\quad\mathcal{F}, \\
\nabla \cdot \psi=0& \text{in}\quad\mathcal{F},\\  
\psi_n=a_n&  \text{on}\quad\partial \mathcal{F},\\
\left[ 2 D(\psi) n+\beta \psi\right]_{\tau}=b& \text{on}\quad \partial \mathcal{F},
\end{array}
\right.
\end{equation}
satisfies the inequality 
\begin{multline}
\label{carlS}
s^2\lambda^2\int_{\mathcal{F}}e^{2s\alpha}\alpha^2 \left|\nabla \psi\right|^2   \ dy	+s^4\lambda^4\int_\mathcal{F} e^{2s\alpha}\alpha^4\left|\psi\right|^2  \ dy +\beta s^3\lambda^2e^{2s}\int_{\partial\mathcal{F}}\left|\psi_\tau\right|^2d\Gamma \\\leq C\bigg(s^4\lambda^4\int_{\mathcal{O}}e^{2s\alpha}\alpha^4\left| \psi\right|^2   \ dy+s\int_{\mathcal{F}} e^{2s\alpha}\alpha\left| f\right|^2  \ dy +s^3\lambda^2e^{2s}\int_{\partial\mathcal{F}}\left| \nabla_\tau (a\cdot n)\right| ^2 d\Gamma\\+ s^4\lambda^2e^{2s}\left\|a_n\right\|^2_{[H^{1/2}(\partial\mathcal{F})]^2}+s^3\lambda^2e^{2s}\int_{\partial\mathcal{F}} \left|b\right|^2 \ d\Gamma\bigg),
\end{multline}
for any $s\geq s_1$ and $\lambda\geq \lambda_1$, where  $f\in [L^2(\mathcal{F})]^2$, $a\in [H^{3/2}(\partial\mathcal{F})]^2$, $b\in [H^{1/2}(\partial\mathcal{F})]^2$.
\end{Proposition}
\begin{Remark}
Since $\beta>0$, then for any $f\in [L^2(\mathcal{F})]^2$, $a\in [H^{3/2}(\partial\mathcal{F})]^2$, $b\in [H^{1/2}(\partial\mathcal{F})]^2$, the system \eqref{SN} admits a unique solution $\psi\in [H^2(\mathcal{F})]^2$. See \cite{BdV} for a complete proof.
\end{Remark}
\begin{proof}
The proof is inspired from \cite{MR2224824} where in our case, we need to take into account the non homogeneous Navier slip boundary conditions and thus, one need to manipulate carefully the surface integrals that appear.
	
\textbf{Step 1:}
Let  $w=e^{s\alpha}\psi$.
The first equation of the system \eqref{SN} becomes
\begin{equation}
\label{eq}
-\Delta w-s^2\lambda^2\alpha^2\left|\nabla \eta\right|^2w+2s\lambda \alpha \nabla w\nabla \eta+s\lambda^2\alpha\left|\nabla\eta\right|^2w+s\lambda \alpha \Delta \eta w =e^{s\alpha}f,
\end{equation} 
We write $\Delta w=\nabla \cdot (\nabla w+(\nabla w)^*)-\nabla (\nabla\cdot w)$.
Using that $\nabla\cdot w=s\lambda\alpha\nabla\eta\cdot w$, we get
$$
-\Delta w=-\nabla \cdot (\nabla w+(\nabla w)^*) +s\lambda\alpha (\nabla w)^* \nabla\eta+s\lambda\alpha\nabla^2\eta w+s\lambda^2\alpha(\nabla\eta\cdot w)\nabla\eta.
$$ 
Then, \eqref{eq} can be written as
\begin{multline}
\label{eqq}
-\nabla \cdot (\nabla w+(\nabla w)^*) +s\lambda\alpha (\nabla w)^*\nabla\eta-s^2\lambda^2\alpha^2\left|\nabla \eta\right|^2w+2s\lambda \alpha \nabla w\nabla \eta+s\lambda^2\alpha\left|\nabla\eta\right|^2w+s\lambda \alpha \Delta \eta w \\+s\lambda\alpha\nabla^2\eta w+s\lambda^2\alpha(\nabla\eta\cdot w)\nabla\eta=e^{s\alpha}f.
\end{multline} 
We multiply \eqref{eqq} by $\alpha^{1/2}$, then \eqref{eqq} is equivalent to	
\begin{equation}
\label{Ceq3}
Mw+Nw=\widetilde{g},
\end{equation}
where 
\begin{equation}
\label{Mw}
Mw=-\alpha^{1/2}\nabla\cdot(\nabla w+(\nabla w)^*) -s^2\lambda^2\alpha^{5/2}\left|\nabla \eta\right|^2w+s\lambda\alpha^{3/2} (\nabla w)^* \nabla\eta,
\end{equation}
\begin{equation}
\label{Nw}
Nw=2s\lambda \alpha^{3/2} \nabla w \nabla\eta+4s\lambda^2\alpha^{3/2}\left|\nabla\eta\right|^2w,
\end{equation}	
and 
\begin{equation}
\widetilde{g}=\alpha^{1/2}e^{s\alpha}f+3s\lambda^2\alpha^{3/2}\left|\nabla\eta\right|^2w-s\lambda \alpha^{3/2} \Delta \eta w-s\lambda\alpha^{3/2}\nabla^2\eta w-s\lambda^2\alpha^{3/2}(\nabla\eta\cdot w)\nabla\eta.
\end{equation}
Multiplying \eqref{Ceq3} by its self, we notice that we only need to consider the terms $\sum_{i,j}\left\langle( M w)_i,(Nw)_j \right\rangle$.
First, we have
\begin{multline}
\left\langle( M w)_1,(Nw)_1 \right\rangle=-2s\lambda\int_{\mathcal{F}}\alpha^2\nabla\cdot(\nabla w+(\nabla w)^*)\cdot\nabla w \nabla\eta  \ dy\\=-2s\lambda\int_{\partial\mathcal{F}}(\nabla\eta\cdot n)(\nabla w+(\nabla w)^*)n\cdot(\nabla w n)  \ d\Gamma+2s\lambda\int_{\mathcal{F}}\alpha^2(\nabla w \nabla^2\eta):(\nabla w+(\nabla w)^*) \ dy\\+4s\lambda^2\int_\mathcal{F} \alpha^2(\nabla w+(\nabla w)^*)\nabla\eta\cdot\nabla w \nabla\eta \ dy+2s\lambda\sum_{i,j,k}\int_{\mathcal{F}}\alpha^2 (\partial_jw_i+\partial_iw_j)\partial^2_{kj}w_i\partial_k\eta \ dy.
\end{multline}
We set
\begin{multline}
A=4s\lambda^2\int_\mathcal{F} \alpha^2(\nabla w+(\nabla w)^*)\nabla\eta\cdot\nabla w \nabla\eta \ dy=4s\lambda^2\int_\mathcal{F} \alpha^2\left|\nabla w \nabla \eta\right|^2  \ dy\\+4s\lambda^2\int_\mathcal{F}\alpha^2((\nabla w)^*\nabla\eta)\cdot(\nabla w\nabla\eta)  \ dy=A_1+A_2.
\end{multline}
We obtain
\begin{multline}
A_2=4s\lambda^2\int_{\partial\mathcal{F}}\left|\nabla\eta \right|^2(\nabla wn)\cdot n(w\cdot n) \ d\Gamma-4s\lambda^2\sum_{i,j,k}\int_{\mathcal{F}}\partial_i(\alpha^2\partial_k\eta\partial_j\eta)\partial_j w_i w_k  \ dy\\-4s^2\lambda^3\int_\mathcal{F}\alpha^2\nabla\eta\cdot\nabla(\alpha\nabla\eta\cdot w)(\nabla \eta\cdot w ) \ dy,
\end{multline}
where we have used that $\nabla\cdot w=s\lambda\alpha\nabla\eta\cdot w$. An integration by parts for the last term gives
\begin{multline}
A_2=4s\lambda^2\int_{\partial\mathcal{F}}\alpha^2\left|\nabla\eta \right|^2(\nabla wn)\cdot n(w\cdot n) \ d\Gamma-4s\lambda^2\sum_{i,j,k}\int_{\mathcal{F}}\partial_i(\alpha^2\partial_k\eta\partial_j\eta)\partial_j w_i w_k  \ dy\\+2s^2\lambda^3\int_\mathcal{F}\alpha^3\Delta\eta\left| \nabla\eta\cdot w\right| ^2  \ dy+2s^2\lambda^4\int_\mathcal{F}\left|\nabla\eta\right|^2 \alpha^3 \left|\nabla \eta \cdot w\right|^2   \ dy -2s^2\lambda^3\int_{\partial\mathcal{F}}(\nabla\eta\cdot n)\left|\nabla\eta\cdot w\right|^2 \ d\Gamma. 
\end{multline}
On the other hand, we have
\begin{multline}
B=2s\lambda\sum_{i,j,k}\int_{\mathcal{F}}\alpha^2 (\partial_jw_i+\partial_iw_j)\partial^2_{kj}w_i\partial_k\eta \ dy\\=s\lambda\int_\mathcal{F} \alpha^2\nabla\eta\cdot\nabla\left|\nabla w\right|^2  \ dy +2s\lambda\sum_{i,j,k}\int_\mathcal{F}\alpha^2 \partial_iw_j\partial^2_{kj}w_i\partial_k\eta  \ dy=B_1+B_2.
\end{multline}
An integration by parts for the terms $B_1$ and $B_2$, gives
\begin{equation}
B_1=s\lambda\int_{\partial\mathcal{F}}(\nabla\eta\cdot n)\left|\nabla w\right|^2 \ d\Gamma-s\lambda\int_{\mathcal{F}}\alpha^2\Delta \eta\left|\nabla w\right|^2  \ dy-2s\lambda^2\int_{\mathcal{F}}\left|\nabla\eta\right|^2\alpha^2\left|\nabla w\right|^2  \ dy,    
\end{equation}
\begin{multline}
B_2=2s\lambda\int_{\partial\mathcal{F}}(\nabla\eta\cdot n)\nabla w:(\nabla w)^* \ d\Gamma-2s\lambda\int_\mathcal{F} \Delta\eta \alpha^2\nabla w: (\nabla w)^* \ dy\\-4s\lambda^2\int_\mathcal{F}\alpha^2\left|\nabla\eta\right|^2 \nabla w:(\nabla w)^*  \ dy-2s\lambda\sum_{i,j,k}\int_\mathcal{F}\alpha^2 \partial^2_{ki}w_j\partial_{j}w_i\partial_k\eta  \ dy=B_{21}+B_{22}+B_{23}+B_{24}.
\end{multline}
We make an integration by parts for $B_{24}$, we get
\begin{multline}
\label{b24}
B_{24}=-2s\lambda\int_{\partial\mathcal{F}}(\nabla\eta\cdot n)(\nabla w n)\cdot((\nabla w)^*n) \ d\Gamma+2s\lambda\int_{\mathcal{F}}\alpha^2(\nabla^2\eta\nabla w):\nabla w \ dy\\+4s\lambda^2\int_\mathcal{F}\alpha^2((\nabla w)^*\nabla\eta)\cdot((\nabla w)\nabla\eta)  \ dy+2s^2\lambda^2\int_\mathcal{F}\alpha^3(\nabla^2\eta w)\cdot (\nabla w \nabla\eta)  \ dy\\+\sum_{i,j,k}s^2\lambda^3\int_{\mathcal{F}} \alpha^3 \partial_k\eta\partial_j\eta\partial_i\eta (\partial_k(w_iw_j))  \ dy+2s^2\lambda^2 \int_{\mathcal{F}}\alpha^3(\nabla w \nabla \eta)\cdot((\nabla w)^*\nabla \eta)  \ dy.
\end{multline}
We notice that the third term in \eqref{b24} corresponds to $A_2$, while the fifth term in \eqref{b24}  denoted by $B_{245}$  gives
\begin{multline*}
B_{245}=s^2\lambda^3\sum_{i,j,k}\int_{\mathcal{F}} \alpha^3 \partial_k\eta\partial_j\eta\partial_i\eta \partial_k(w_iw_j)  \ dy=-\sum_{i,j,k}s^2\lambda^3\int_{\mathcal{F}}\partial_k( \alpha^3 \partial_k\eta\partial_j\eta\partial_i\eta) (w_iw_j)  \ dy\\+s^2\lambda^3\int_{\partial\mathcal{F}} ( \nabla \eta\cdot n)^3 \left|w\cdot n\right|^2  \ d\Gamma,
\end{multline*}
we note that
$$
-\sum_{i,j,k}s^2\lambda^3\int_{\mathcal{F}}\partial_k( \alpha^3 \partial_k\eta\partial_j\eta\partial_i\eta) (w_iw_j)  \ dy\geq -C s^2\lambda^3(1+\lambda)\int_\mathcal{F} \alpha^4\left|w\right|^2   \ dy.
$$
We treat the sixth term in \eqref{b24}, we obtain 
\begin{multline}
B_{246}=2s^2\lambda^2 \int_{\mathcal{F}}\alpha^3(\nabla w \nabla \eta)\cdot((\nabla w)^*\nabla \eta)  \ dy=-2s^2\lambda^2\int_\mathcal{F}\alpha^3\partial_i(\partial_j\eta\partial_k\eta)\partial_jw_iw_k  \ dy \\-3s^2\lambda^3\sum_{i,j,k}\int_{\mathcal{F}} \alpha^3 \partial_k\eta\partial_j\eta\partial_i\eta \partial_k(w_iw_j)  \ dy-s^3\lambda^3\int_{\mathcal{F}}\alpha^2\nabla\eta \cdot\nabla \left|\alpha \nabla\eta \cdot w\right|^2  \ dy\\+2s^2\lambda^2\int_{\partial\mathcal{F}}\left|\nabla \eta \right|^2((\nabla w)^*n)\cdot n(w\cdot n) \ d\Gamma=C_1-3B_{245}+C_3+C_4. 
\end{multline}
\begin{multline}
C_3=s^3\lambda^3\int_\mathcal{F}\alpha^4\Delta\eta \left|w\cdot\nabla \eta\right|^2  \ dy+2s^3\lambda^4\int_\mathcal{F}\alpha^4\left|\nabla\eta\right|^2\left|w\cdot\nabla\eta\right|^2  \ dy-s^3\lambda^3\int_{\partial\mathcal{F}}(\nabla\eta\cdot n)^3\left|w\cdot n\right|^2 \ d\Gamma. 
\end{multline}
Thus, we obtain
\begin{multline}
\left\langle( M w)_1,(Nw)_1 \right\rangle\geq -2s\lambda\int_{\partial\mathcal{F}}(\nabla\eta\cdot n)(\nabla w+(\nabla w)^*)n\cdot(\nabla w n)  \ d\Gamma+s\lambda\int_{\partial\mathcal{F}}(\nabla\eta\cdot n)\left|\nabla w\right|^2 \ d\Gamma\\+2s\lambda\int_{\partial\mathcal{F}}(\nabla\eta\cdot n)\nabla w:(\nabla w)^* \ d\Gamma-2s\lambda\int_{\partial\mathcal{F}}(\nabla\eta\cdot n)(\nabla w n)\cdot((\nabla w)^*n) \ d\Gamma\\+8s\lambda^2\int_{\partial\mathcal{F}}\left|\nabla\eta \right|^2(\nabla wn)\cdot n(w\cdot n) \ d\Gamma+2s^2\lambda^2\int_{\partial\mathcal{F}}\left|\nabla \eta \right|^2 ((\nabla w)^*n)\cdot n(w\cdot n) \ d\Gamma\\-6s^2\lambda^3\int_{\partial\mathcal{F}}(\nabla\eta\cdot n)\left|\nabla\eta\cdot w\right|^2 \ d\Gamma-s^3\lambda^3\int_{\partial\mathcal{F}}(\nabla\eta\cdot n)^3\left|w\cdot n\right|^2 \ d\Gamma\\+2s^3\lambda^4\int_\mathcal{F}\left|\nabla\eta\right|^2\alpha^4\left|w\cdot\nabla\eta\right|^2  \ dy-4s\lambda^2\int_\mathcal{F}\alpha^2\left|\nabla\eta\right|^2 \nabla w:(\nabla w)^*  \ dy-2s\lambda^2\int_{\mathcal{F}}\left|\nabla\eta\right|^2\alpha^2\left|\nabla w\right|^2  \ dy\\-\varepsilon\left( s\lambda^2\int_{\mathcal{F}}\alpha^2\left|\nabla w\right|^2  \ dy+s^3\lambda^4\int_{\mathcal{F}}\alpha^4\left| w\right|^2  \ dy\right) ,
\end{multline}
for all $\varepsilon>0$ and for $s>1$, $\lambda>1$.
We get also
\begin{multline}
\left\langle(Mw)_1,(Nw)_2\right\rangle=-4s\lambda^2\int_\mathcal{F}\left|\nabla\eta\right|^2\alpha^2\left[ \nabla\cdot(\nabla w+(\nabla w)^*)\right] \cdot w \ dy\\=-4s\lambda^2\int_{\partial\mathcal{F}}\left|\nabla\eta \right|^2((\nabla w)+(\nabla w)^*)n\cdot w  \ d\Gamma+8s\lambda^2 \int_\mathcal{F}\alpha^2(\nabla^2\eta\nabla\eta)(\nabla w+(\nabla w)^*)\cdot w \ dy \\+8s\lambda^3\int_\mathcal{F}\left|\nabla\eta\right|^2\alpha^2(\nabla\eta(\nabla w+(\nabla w)^*))\cdot w \ dy+4s\lambda^2\int_\mathcal{F}\left|\nabla\eta\right|^2\alpha^2\left|\nabla w\right|^2  \ dy\\+4s\lambda^2\int_\mathcal{F} \left|\nabla\eta\right|^2\alpha^2\nabla w:(\nabla w )^*  \ dy.  
\end{multline}
On the other hand, we have
\begin{multline}
\left\langle(Mw)_2,(Nw)_1\right\rangle=-s^3\lambda^3\int_\mathcal{F}\left|\nabla\eta\right|^2\alpha^4\nabla\eta\cdot\nabla\left|w\right|^2  \ dy=-s^3\lambda^3\int_{\partial\mathcal{F}}(\nabla\eta\cdot n )^3\left| w\right|^2   d\Gamma \\+s^3\lambda^3\int_\mathcal{F}\nabla\cdot(\left|\nabla\eta\right|^2\nabla\eta)\alpha^4\left|  w\right| ^2 \ dy+4s^3\lambda^4\int_\mathcal{F}\left|\nabla\eta\right|^4\alpha^4\left|w\right|^2  \ dy.
\end{multline}
We get
\begin{equation}
\left\langle(Mw)_2,(Mw)_2\right\rangle=-4s^3\lambda^4\int_\mathcal{F}\left|\nabla\eta\right|^4\alpha^4\left|w\right|^2  \ dy.
\end{equation}
We obtain
\begin{multline}
\left\langle (Mw)_3,(Nw)_1 \right\rangle=B_{246}\geq  -3s^2\lambda^3\int_{\partial\mathcal{F}}(\nabla\eta\cdot n)^3 \left|w \cdot n\right|^2 \ d\Gamma+2s^2\lambda^2\int_{\partial\mathcal{F}}\left|\nabla \eta \right|^2 ((\nabla w)^*n)\cdot n(w\cdot n) \ d\Gamma
\\+2s^3\lambda^4\int_\mathcal{F}\left|\nabla\eta\right|^2\alpha^4\left|w\cdot\nabla\eta\right|^2  \ dy-s^3\lambda^3\int_{\partial\mathcal{F}}(\nabla\eta\cdot n)^3\left|w\cdot n\right|^2 \ d\Gamma\\-\varepsilon\left( s\lambda^2\int_{\mathcal{F}}\alpha^2\left|\nabla w\right|^2  \ dy+s^3\lambda^4\int_{\mathcal{F}}\alpha^4\left| w\right|^2  \ dy\right). 
\end{multline}
We have also
\begin{multline}
\left\langle (M w)_3,(Nw)_2 \right\rangle =4s^2\lambda^3\int_\mathcal{F}\left|\nabla\eta\right|^2\alpha^3(\nabla w)^*\nabla\eta\cdot w  \ dy=4s^2\lambda^3\int_{\partial\mathcal{F}}(\nabla\eta\cdot n)^3\left|w\cdot n \right|^2  \ d\Gamma\\-4s^2\lambda^3\int_\mathcal{F}(\nabla(\left|\nabla \eta \right|^2\alpha^3\nabla\eta )\cdot w)\cdot w  \ dy -4s^3\lambda^4\int_\mathcal{F}\left|\nabla\eta\right|^2\alpha^4\left|w\cdot\nabla\eta\right|^2  \ dy.
\end{multline}
Then, we get
\begin{multline}
\label{step2}
\langle(Mw),(Nw)\rangle\geq -2s\lambda\int_{\partial\mathcal{F}}(\nabla\eta\cdot n)(\nabla w+(\nabla w)^*)n\cdot(\nabla w n)  \ d\Gamma+s\lambda\int_{\partial\mathcal{F}}(\nabla\eta\cdot n)\left|\nabla w\right|^2 \ d\Gamma\\+2s\lambda\int_{\partial\mathcal{F}}(\nabla\eta\cdot n)\nabla w:(\nabla w)^* \ d\Gamma-2s\lambda\int_{\partial\mathcal{F}}(\nabla\eta\cdot n)(\nabla w n)\cdot((\nabla w)^*n) \ d\Gamma\\-4s\lambda^2\int_{\partial\mathcal{F}}\left|\nabla\eta \right|^2(\nabla wn+(\nabla w)^*n)_\tau\cdot w_\tau \ d\Gamma+4s^2\lambda^2\int_{\partial\mathcal{F}}\left|\nabla \eta \right|^2 ((\nabla w)^*n)\cdot n(w\cdot n) \ d\Gamma\\-2s^3\lambda^3\int_{\partial\mathcal{F}}(\nabla\eta\cdot n)^3\left|w\cdot n\right|^2 \ d\Gamma-s^3\lambda^3\int_{\partial\mathcal{F}}(\nabla\eta\cdot n )^3\left| w\right|^2   d\Gamma\\+2s\lambda^2\int_{\mathcal{F}}\left|\nabla\eta\right|^2\alpha^2\left|\nabla w\right|^2  \ dy-5s^2\lambda^3\int_{\partial\mathcal{F}}(\nabla\eta\cdot n)^3\left|w\cdot n\right|^2 \ d\Gamma\\-\varepsilon\left( s\lambda^2\int_{\mathcal{F}}\alpha^2\left|\nabla w\right|^2  \ dy+s^3\lambda^4\int_{\mathcal{F}}\alpha^4\left| w\right|^2  \ dy\right) ,
\end{multline}
where we have used 
\begin{multline*}
-4s\lambda^2\int_{\partial\mathcal{F}}\left|\nabla\eta \right|^2((\nabla w)+(\nabla w)^*)n\cdot w  \ d\Gamma+8s\lambda^2 \int_{\partial\mathcal{F}}\left|\nabla\eta \right|^2(\nabla w)^*n\cdot n (w\cdot n)  \ d\Gamma\\=-4s\lambda^2\int_{\partial\mathcal{F}}\left|\nabla\eta \right|^2(((\nabla w)+(\nabla w)^*)n)_\tau\cdot w_\tau  \ d\Gamma.
\end{multline*}
\textbf{Step 2:} We derive a Carleman estimate for $\widetilde{w}=e^{s\widetilde{\alpha}}\psi$ with $\widetilde{\alpha}=e^{-\lambda\eta}$, the calculus will be analogous and we will get the same terms up to a sign. 
We obtain
\begin{multline}
\label{step3}
\langle(\widetilde{M}\widetilde{w}),(\widetilde{N}\widetilde{w})\rangle\geq 2s\lambda\int_{\partial\mathcal{F}}(\nabla\eta\cdot n)(\nabla \widetilde{w}+(\nabla \widetilde{w})^*)n\cdot(\nabla \widetilde{w} n)  \ d\Gamma-s\lambda\int_{\partial\mathcal{F}}(\nabla\eta\cdot n)\left|\nabla \widetilde{w}\right|^2 \ d\Gamma\\-2s\lambda\int_{\partial\mathcal{F}}(\nabla\eta\cdot n)\nabla \widetilde{w}:(\nabla \widetilde{w})^* \ d\Gamma+2s\lambda\int_{\partial\mathcal{F}}(\nabla\eta\cdot n)(\nabla \widetilde{w} n)\cdot((\nabla \widetilde{w})^*n) \ d\Gamma\\-4s\lambda^2\int_{\partial\mathcal{F}}\left|\nabla\eta \right|^2(\nabla \widetilde{w}n+(\nabla \widetilde{w})^*n)_\tau\cdot \widetilde{w}_\tau \ d\Gamma+4s^2\lambda^2\int_{\partial\mathcal{F}}\left|\nabla \eta \right|^2 ((\nabla \widetilde{w})^*n)\cdot n(\widetilde{w}\cdot n) \ d\Gamma\\+2s^3\lambda^3\int_{\partial\mathcal{F}}(\nabla\eta\cdot n)^3\left|\widetilde{w}\cdot n\right|^2 \ d\Gamma+s^3\lambda^3\int_{\partial\mathcal{F}}(\nabla\eta\cdot n )^3\left| \widetilde{w}\right|^2   d\Gamma\\+2s\lambda^2\int_{\mathcal{F}}\left|\nabla\eta\right|^2\widetilde{\alpha}^2\left|\nabla \widetilde{w}\right|^2  \ dy+5s^2\lambda^3\int_{\partial\mathcal{F}}(\nabla\eta\cdot n)^3\left|\widetilde{w}\cdot n\right|^2 \ d\Gamma\\-\varepsilon\left( s\lambda^2\int_{\mathcal{F}}\widetilde{\alpha}^2\left|\nabla \widetilde{w}\right|^2  \ dy+s^3\lambda^4\int_{\mathcal{F}}\widetilde{\alpha}^4\left| \widetilde{w}\right|^2  \ dy\right).
\end{multline}
\textbf{Step 3:}
We deal with the surface integrals. We note that on the boundary $\partial\mathcal{F}$, we have
\begin{equation}
\label{b1}
\nabla w_i=e^{s}(\nabla\psi+s\lambda\nabla\eta\psi_i),\quad
\nabla\widetilde{w}_i=e^{s}(\nabla\psi-s\lambda\nabla\eta\psi_i),\quad \text{on }\partial\mathcal{F}.
\end{equation}
Then,
\begin{equation}
\label{b2}
\nabla wn=e^{s}(\nabla\psi n+s\lambda(\nabla\eta\cdot n)\psi),\quad\nabla \widetilde{w}n=e^{s}(\nabla\psi n-s\lambda(\nabla\eta\cdot n)\psi),\quad \text{on }\partial\mathcal{F},
\end{equation}
and
\begin{equation}
\label{b3}
(\nabla w)^*n=e^{s}((\nabla\psi)^* n+s\lambda(\psi\cdot n)\nabla\eta),\quad (\nabla\widetilde{w})^*n=e^{s}((\nabla\psi)^* n-s\lambda(\psi\cdot n)\nabla\eta),\quad \text{on }\partial\mathcal{F}.
\end{equation}
The boundary terms in \eqref{step2} are reduced to
\begin{multline*}
S=-2s\lambda\int_{\partial\mathcal{F}}(\nabla\eta\cdot n)(\nabla w+(\nabla w)^*)n\cdot(\nabla w n)  \ d\Gamma+s\lambda\int_{\partial\mathcal{F}}(\nabla\eta\cdot n)\left|\nabla w\right|^2 \ d\Gamma\\+2s\lambda\int_{\partial\mathcal{F}}(\nabla\eta\cdot n)\nabla w:(\nabla w)^* \ d\Gamma-2s\lambda\int_{\partial\mathcal{F}}(\nabla\eta\cdot n)(\nabla w n)\cdot((\nabla w)^*n) \ d\Gamma\\-4s\lambda^2\int_{\partial\mathcal{F}}\left|\nabla\eta \right|^2(\nabla wn+(\nabla w)^*n)_\tau\cdot w_\tau \ d\Gamma+4s^2\lambda^2\int_{\partial\mathcal{F}}\left|\nabla \eta \right|^2 ((\nabla w)^*n)\cdot n(w\cdot n) \ d\Gamma\\-2s^3\lambda^3\int_{\partial\mathcal{F}}(\nabla\eta\cdot n)^3\left|w\cdot n\right|^2 \ d\Gamma-s^3\lambda^3\int_{\partial\mathcal{F}}(\nabla\eta\cdot n )^3\left| w\right|^2  d\Gamma-5s^2\lambda^3\int_{\partial\mathcal{F}}(\nabla\eta\cdot n)^3\left|w\cdot n\right|^2 \ d\Gamma.
\end{multline*}
The boundary terms in \eqref{step3} write
\begin{multline*}
\widetilde{S}=2s\lambda\int_{\partial\mathcal{F}}(\nabla\eta\cdot n)(\nabla \widetilde{w}+(\nabla \widetilde{w})^*)n\cdot(\nabla \widetilde{w} n)  \ d\Gamma-s\lambda\int_{\partial\mathcal{F}}(\nabla\eta\cdot n)\left|\nabla \widetilde{w}\right|^2 \ d\Gamma\\-2s\lambda\int_{\partial\mathcal{F}}(\nabla\eta\cdot n)\nabla \widetilde{w}:(\nabla \widetilde{w})^* \ d\Gamma+2s\lambda\int_{\partial\mathcal{F}}(\nabla\eta\cdot n)(\nabla \widetilde{w} n)\cdot((\nabla \widetilde{w})^*n) \ d\Gamma\\-4s\lambda^2\int_{\partial\mathcal{F}}\left|\nabla\eta \right|^2(\nabla \widetilde{w}n+(\nabla \widetilde{w})^*n)_\tau\cdot \widetilde{w}_\tau \ d\Gamma+4s^2\lambda^2\int_{\partial\mathcal{F}}\left|\nabla \eta \right|^2 ((\nabla \widetilde{w})^*n)\cdot n(\widetilde{w}\cdot n) \ d\Gamma\\+2s^3\lambda^3\int_{\partial\mathcal{F}}(\nabla\eta\cdot n)^3\left|\widetilde{w}\cdot n\right|^2 \ d\Gamma+s^3\lambda^3\int_{\partial\mathcal{F}}(\nabla\eta\cdot n )^3\left| \widetilde{w}\right|^2   d\Gamma d\Gamma+5s^2\lambda^3\int_{\partial\mathcal{F}}(\nabla\eta\cdot n)^3\left|\widetilde{w}\cdot n\right|^2 \ d\Gamma.
\end{multline*}
Using that $\alpha=\widetilde{\alpha}=1$ and $w=\widetilde{w}$ on $\partial\mathcal{F}$, the boundary terms are reduced to
\begin{multline*}
S+\widetilde{S}=-2s\lambda\left( \int_{\partial\mathcal{F}}(\nabla\eta\cdot n)(\nabla w+(\nabla w)^*)n\cdot(\nabla w n)  \ d\Gamma- \int_{\partial\mathcal{F}}(\nabla\eta\cdot n)(\nabla \widetilde{w}+(\nabla \widetilde{w})^*)n\cdot(\nabla \widetilde{w} n)  \ d\Gamma\right)\\ +s\lambda\left( \int_{\partial\mathcal{F}}(\nabla\eta\cdot n)\left|\nabla w\right|^2 \ d\Gamma- \int_{\partial\mathcal{F}}(\nabla\eta\cdot n)\left|\nabla \widetilde{w}\right|^2 \ d\Gamma\right)\\+2s\lambda\left( \int_{\partial\mathcal{F}}(\nabla\eta\cdot n)\nabla w:(\nabla w)^* \ d\Gamma- \int_{\partial\mathcal{F}}(\nabla\eta\cdot n)\nabla \widetilde{w}:(\nabla \widetilde{w})^* \ d\Gamma\right)\\ -2s\lambda\left( \int_{\partial\mathcal{F}}(\nabla\eta\cdot n)(\nabla w n)\cdot((\nabla w)^*n) \ d\Gamma- \int_{\partial\mathcal{F}}(\nabla\eta\cdot n)(\nabla \widetilde{w} n)\cdot((\nabla \widetilde{w})^*n) \ d\Gamma\right)\\ -4s\lambda^2\left( \int_{\partial\mathcal{F}}\left|\nabla\eta \right|^2(\nabla wn+(\nabla w)^*n)_\tau\cdot w_\tau \ d\Gamma+ \int_{\partial\mathcal{F}}\left|\nabla\eta \right|^2(\nabla \widetilde{w}n+(\nabla \widetilde{w})^*n)_\tau\cdot \widetilde{w}_\tau \ d\Gamma\right)\\ +4s^2\lambda^2\left( \int_{\partial\mathcal{F}}\left|\nabla \eta \right|^2 ((\nabla w)^*n)\cdot n(w\cdot n) \ d\Gamma+ \int_{\partial\mathcal{F}}\left|\nabla \eta \right|^2((\nabla \widetilde{w})^*n)\cdot n(\widetilde{w}\cdot n) \ d\Gamma\right)=\sum_{i=1}^6 I_i. 
\end{multline*}
Using \eqref{b1}, \eqref{b2}, \eqref{b3} and \eqref{SN}$_4$, we get
\begin{equation*}
I_1+I_4=-8s^2\lambda^2\int_{\partial\mathcal{F}}e^{2s}\left|\nabla \eta \right|^2(\nabla\psi n+(\nabla\psi)^* n )\cdot\psi \ d\Gamma-8s^2\lambda^2\int_{\partial\mathcal{F}}e^{2s}\left|\nabla \eta \right|^2(\nabla\psi n\cdot n)(\psi\cdot n) \ d\Gamma.
\end{equation*}
\begin{equation*}
I_2=4s^2\lambda^2\int_{\partial\mathcal{F}}e^{2s}\left|\nabla \eta \right|^2\nabla\psi n \cdot\psi \ d\Gamma,
\end{equation*}
\begin{equation*}
I_3=8s^2\lambda^2\int_{\partial\mathcal{F}}e^{2s}\left|\nabla \eta \right|^2(\nabla\psi)^* n \cdot\psi \ d\Gamma,
\end{equation*}
\begin{equation*}
I_5=8\beta s\lambda^2\int_{\partial\mathcal{F}}e^{2s}\left|\nabla \eta \right|^2\left|\psi_\tau\right|^2 \ d\Gamma-8 s\lambda^2\int_{\partial\mathcal{F}}e^{2s}\left|\nabla \eta \right|^2b\cdot \psi \ d\Gamma,
\end{equation*}
\begin{equation*}
I_6=8s^2\lambda^2\int_{\partial\mathcal{F}}e^{2s}\left|\nabla \eta \right|^2(\nabla\psi n\cdot n)(\psi\cdot n) \ d\Gamma.
\end{equation*}
Then, we get
\begin{multline*}
I_1+I_4=-8s^2\lambda^2\int_{\partial\mathcal{F}}e^{2s}\left|\nabla \eta \right|^2(\nabla\psi n+(\nabla\psi)^* n )_\tau\cdot\psi_\tau \ d\Gamma-24s^2\lambda^2\int_{\partial\mathcal{F}}e^{2s}\left|\nabla \eta \right|^2(\nabla\psi n\cdot n)(\psi\cdot n) \ d\Gamma\\=8\beta s^2\lambda^2\int_{\partial\mathcal{F}}e^{2s}\left|\nabla \eta \right|^2\left|\psi_\tau\right|^2 \ d\Gamma-8 s^2\lambda^2\int_{\partial\mathcal{F}}e^{2s}\left|\nabla \eta \right|^2b\cdot \psi \ d\Gamma-24s^2\lambda^2\int_{\partial\mathcal{F}}e^{2s}\left|\nabla \eta \right|^2(\nabla\psi n\cdot n)(\psi\cdot n) \ d\Gamma,
\end{multline*}
and
\begin{multline*}
I_2=4s^2\lambda^2\int_{\partial\mathcal{F}}e^{2s}\left|\nabla \eta \right|^2(\nabla\psi n)_\tau \cdot\psi_\tau \ d\Gamma+4s^2\lambda^2\int_{\partial\mathcal{F}}e^{2s}\left|\nabla \eta \right|^2(\nabla\psi n)\cdot n (\psi\cdot n) \ d\Gamma\\=-4s^2\lambda^2\int_{\partial\mathcal{F}}e^{2s}\left|\nabla \eta \right|^2(\nabla\psi)^*n_\tau \cdot\psi_\tau \ d\Gamma-4s^2\lambda^2\beta\int_{\partial\mathcal{F}}e^{2s}\left| \psi_\tau\right|^2  \ d\Gamma\\+4s^2\lambda^2\int_{\partial\mathcal{F}}e^{2s}\left|\nabla \eta \right|^2b_\tau \cdot\psi_\tau \ d\Gamma+4s^2\lambda^2\int_{\partial\mathcal{F}}e^{2s}\left|\nabla \eta \right|^2(\nabla\psi n)\cdot n (\psi\cdot n) \ d\Gamma.
\end{multline*}
We notice that
$$
-4s^2\lambda^2\int_{\partial\mathcal{F}}e^{2s}\left|\nabla \eta \right|^2\left[ (\nabla\psi)^*n\right] _\tau \cdot\psi_\tau \ d\Gamma=-4s^2\lambda^2\int_{\partial\mathcal{F}}e^{2s}\left|\nabla \eta \right|^2\left[  \nabla_\tau (a\cdot n)\cdot \psi_\tau-(\nabla n\tau\cdot \psi)(\psi\cdot\tau)\right] \ d\Gamma.
$$
Using the inequality in \cite[Theorem II.4.1]{galdi} with $r=2$, $q=2$, we obtain
$$
s^2\lambda^2\int_{\partial\mathcal{F}}\left|\psi\right|^2 \ d\Gamma \leq C \left( s\lambda \int_\mathcal{F} \left|\nabla \psi\right|^2 \ dy + s^3\lambda^3\int_\mathcal{F}\left|\psi\right|^2  \ dy\right) .
$$
Applying the same arguments, we get for $I_3$
\begin{multline*}
I_3=8s^2\lambda^2\int_{\partial\mathcal{F}}e^{2s}\left|\nabla \eta \right|^2((\nabla\psi)^* n)_\tau \cdot\psi_\tau \ d\Gamma+8s^2\lambda^2\int_{\partial\mathcal{F}}e^{2s}\left|\nabla \eta \right|^2(\nabla\psi n)\cdot n (\psi\cdot n) \ d\Gamma\\=8s^2\lambda^2\int_{\partial\mathcal{F}}e^{2s}\left|\nabla \eta \right|^2\left[  \nabla_\tau (a\cdot n)\cdot \psi_\tau-(\nabla n\tau\cdot \psi)(\psi\cdot\tau)\right] \ d\Gamma+8s^2\lambda^2\int_{\partial\mathcal{F}}e^{2s}\left|\nabla \eta \right|^2(\nabla\psi n)\cdot n (\psi\cdot n) \ d\Gamma.
\end{multline*}
Then, combining all these inequalities, we get 
\begin{multline*}
s\lambda^2\int_{\mathcal{F}}e^{2s\alpha}\alpha^2\left| \nabla\eta\right|^2 \left|\nabla \psi\right|^2   \ dy +s^2\lambda^2\beta e^{2s}\int_{\partial\mathcal{F}}\left|\psi_\tau\right|^2d\Gamma\\\leq  \varepsilon s^3\lambda^4\int_{\mathcal{F}}e^{2s\alpha} \alpha^4\left|\psi \right|^2  \ dy +\varepsilon s\lambda^2\int_{\mathcal{F}}e^{2s\alpha} \alpha^2 \left|\nabla \psi \right|^2  \ dy+C\bigg(\int_{\mathcal{F}} e^{2s\alpha}\alpha\left| f\right|^2  \ dy +s^2\lambda^2e^{2s}\int_{\partial\mathcal{F}}\left| \nabla (a\cdot n)\tau\right| ^2 d\Gamma\\+s^2\lambda^2e^{2s}\int_{\partial\mathcal{F}} \left|b\right|^2 \ d\Gamma\bigg)+4s^2\lambda^2e^{2s}\int_{\partial\mathcal{F}}\left|\nabla \eta \right|^2(\nabla\psi n\cdot n)(\psi\cdot n) \ d\Gamma.
\end{multline*}
Using the fact that $\left| \nabla\eta\right|>0 $ on $\mathcal{F} \backslash \mathcal{O}_\eta$, we obtain 
\begin{multline*}
s\lambda^2\int_{\mathcal{F}\backslash\mathcal{O}_\eta}e^{2s\alpha}\alpha^2 \left|\nabla \psi\right|^2   \ dy +s^2\lambda^2\beta e^{2s}\int_{\partial\mathcal{F}}\left|\psi_\tau\right|^2d\Gamma\\\leq   \varepsilon s^3\lambda^4\int_{\mathcal{F}}e^{2s\alpha} \alpha^4\left|\psi \right|^2  \ dy +\varepsilon s\lambda^2\int_{\mathcal{F}}e^{2s\alpha} \alpha^2 \left|\nabla \psi \right|^2  \ dy+C\bigg(\int_{\mathcal{F}} e^{2s\alpha}\alpha\left| f\right|^2  \ dy +s^2\lambda^2\int_{\partial\mathcal{F}}e^{2s}\left| \nabla (a\cdot n)\tau\right| ^2 d\Gamma\\+s^2\lambda^2e^{2s}\int_{\partial\mathcal{F}} \left|b\right|^2 \ d\Gamma\bigg)+4s^2\lambda^2e^{2s}\int_{\partial\mathcal{F}}\left|\nabla \eta \right|^2(\nabla\psi n\cdot n)(\psi\cdot n) \ d\Gamma.
\end{multline*}
We add the term $s\lambda^2\int_{\mathcal{O}_\eta}e^{2s\alpha}\alpha^2 \left|\nabla \psi\right|^2   \ dy$ in the both sides of the last equation, to get
\begin{multline*}
s\lambda^2\int_{\mathcal{F}}e^{2s\alpha}\alpha^2 \left|\nabla \psi\right|^2   \ dy +s^2\lambda^2\beta\int_{\partial\mathcal{F}}e^{2s}\left|\psi_\tau\right|^2d\Gamma \\\leq  \varepsilon s^3\lambda^4\int_{\mathcal{F}}e^{2s\alpha} \alpha^4\left|\psi \right|^2  \ dy +\varepsilon s\lambda^2\int_{\mathcal{F}}e^{2s\alpha} \alpha^2 \left|\nabla \psi \right|^2  \ dy+C\bigg(\int_{\mathcal{F}} e^{2s\alpha}\alpha\left| f\right|^2  \ dy +s^2\lambda^2e^{2s}\int_{\partial\mathcal{F}}\left| \nabla (a\cdot n)\tau\right| ^2 d\Gamma\\+s^2\lambda^2e^{2s}\int_{\partial\mathcal{F}}\left|b\right|^2 \ d\Gamma\bigg)+4s^2\lambda^2e^{2s}\int_{\partial\mathcal{F}}\left|\nabla \eta \right|^2(\nabla\psi n\cdot n)(\psi\cdot n) \ d\Gamma.
\end{multline*}
We use the following inequality that is proved in \cite[Lemma 3]{MR2503028}
\begin{equation}
s^3\lambda^4\int_\mathcal{F} e^{2s\alpha}\alpha^4\left|\psi\right|^2   \ dy\leq C \left( 	s\lambda^2\int_{\mathcal{F}}e^{2s\alpha}\alpha^2 \left|\nabla \psi\right|^2   \ dy+s^3\lambda^4\int_{\mathcal{O}_\eta}e^{2s\alpha}\alpha^4\left| \psi\right|^2   \ dy \right). 
\end{equation}
Then, we get for sufficiently large $s$ and $\lambda$
\begin{multline}
\label{eq18}
s\lambda^2\int_{\mathcal{F}}e^{2s\alpha}\alpha^2 \left|\nabla \psi\right|^2   \ dy	+s^3\lambda^4\int_\mathcal{F} e^{2s\alpha}\alpha^4\left|\psi\right|^2  \ dy +s^2\lambda^2\beta e^{2s}\int_{\partial\mathcal{F}}\left|\psi_\tau\right|^2d\Gamma \\\leq C\bigg(s^3\lambda^4\int_{\mathcal{O}_\eta}e^{2s\alpha}\alpha^4\left| \psi\right|^2   \ dy+ s\lambda^2\int_{\mathcal{O}_\eta}e^{2s\alpha} \alpha^2 \left|\nabla \psi \right|^2  \ dy+\int_{\mathcal{F}} e^{2s\alpha}\alpha\left| f\right|^2  \ dy +s^2\lambda^2e^{2s}\int_{\partial\mathcal{F}}\left| \nabla (a\cdot n)\tau\right| ^2 d\Gamma\\+s^2\lambda^2e^{2s}\int_{\partial\mathcal{F}} \left|b\right|^2 \ d\Gamma\bigg)+4s^2\lambda^2e^{2s}\int_{\partial\mathcal{F}}\left|\nabla \eta \right|^2(\nabla\psi n\cdot n)(\psi\cdot n) \ d\Gamma.
\end{multline}
We recall that 
$$
\nabla\times \psi=\frac{\partial\psi_2}{\partial y_1}-\frac{\partial\psi_1}{\partial y_2} .
$$
Since $\psi$ is divergence free, we have that 
$$
\Delta\psi=-\nabla\times(\nabla\times \psi).
$$
We have used the fact that for any scalar function $a:\mathbb{R}^2\longrightarrow \mathbb{R}$, we have 
$$\nabla\times a=\begin{pmatrix}
\partial_2 a\\-\partial_1 a
\end{pmatrix}.
$$
We recall the Green formula
$$
\int_\mathcal{F} \Delta \psi\cdot\widehat{v} \ dy=-\int_\mathcal{F} (\nabla\times\psi)\cdot(\nabla \times \widehat{v}) \ dy -\int_{\partial\mathcal{F}}(\widehat{v}\cdot\tau)(\nabla\times\psi)\ d\Gamma,\quad \forall\widehat{v}\in [H^1(\mathcal{F})]^2,
$$
with $\tau=\begin{pmatrix}
-n_2\\ n_1
\end{pmatrix}$.
Then, we obtain
\begin{multline*}
\int_{\partial\mathcal{F}}\left|\nabla \eta \right|^2(\nabla\psi n\cdot n)(\psi\cdot n)\ d\Gamma =\int_{\partial\mathcal{F}}\left|\nabla \eta \right|^2\nabla\psi n\cdot\psi_n\ d\Gamma\\=\int_\mathcal{F} \nabla \psi : \nabla \widehat{v}  \ dy-\int_\mathcal{F} (\nabla\times\psi) \cdot (\nabla \times \widehat{v} ) \ dy-\int_{\partial\mathcal{F}}(\widehat{v}\cdot\tau)(\nabla\times\psi)\ d \Gamma,
\end{multline*}
where $\widehat{v}$ is a $[H^1(\mathcal{F})]^2$ such that $\widehat{v}=\left|\nabla \eta \right|^2\psi_n$ on $\partial\mathcal{F}$ and
$$
\left\|\widehat{v}\right\|_{[H^1(\mathcal{F})]^2} \leq C \left\|\psi_n\right\|_{[H^{1/2}(\partial\mathcal{F})]^2} .
$$ 
Thus, the last term in the right hand side of the inequality \eqref{eq18} gives
\begin{equation*}
s^2\lambda^2e^{2s}\int_{\partial\mathcal{F}}\left|\nabla \eta \right|^2(\nabla\psi n\cdot n)(\psi\cdot n)\ d\Gamma \leq  \varepsilon s\lambda
^2e^{2s}\int_\mathcal{F} \left|\nabla \psi\right|^2  \ dy  + Cs^3\lambda^2e^{2s}\left\|\psi_n\right\|^2_{[H^{1/2}(\partial\mathcal{F})]^2}.
\end{equation*}
Then, we obtain
\begin{multline*}
s\lambda^2\int_{\mathcal{F}}e^{2s\alpha}\alpha^2 \left|\nabla \psi\right|^2   \ dy	+s^3\lambda^4\int_\mathcal{F} e^{2s\alpha}\alpha^4\left|\psi\right|^2  \ dy +s^2\lambda^2\beta e^{2s}\int_{\partial\mathcal{F}}\left|\psi_\tau\right|^2d\Gamma \\\leq C\bigg(s^3\lambda^4\int_{\mathcal{O}_\eta}e^{2s\alpha}\alpha^4\left| \psi\right|^2   \ dy+ s\lambda^2\int_{\mathcal{O}_\eta} e^{2s\alpha}\alpha^2 \left|\nabla \psi \right|^2  \ dy+\int_{\mathcal{F}} e^{2s\alpha}\alpha\left| f\right|^2  \ dy +s^2\lambda^2e^{2s}\int_{\partial\mathcal{F}}\left| \nabla (a\cdot n)\tau\right| ^2 d\Gamma\\+s^3\lambda^2e^{2s}\left\|\psi_n\right\|^2_{[H^{1/2}(\partial\mathcal{F})]^2}+s^2\lambda^2e^{2s}\int_{\partial\mathcal{F}} \left|b\right|^2 \ d\Gamma\bigg).
\end{multline*}
To adsorb the second term of the right hand side, we proceed like \cite[inequality (1.62)]{FS} which shows that the integral of $e^{2s\alpha}\alpha^2\left| \nabla \psi\right|^2 $  over $\mathcal{O}_\eta$ can be estimated by $e^{2s\alpha}\alpha^4\left| \psi\right|^2 $ over a larger set $\mathcal{O}$.
	
Indeed, we define $\theta\in C^2_0(\overline{\mathcal{O}})$ such that $\theta\equiv1$ in $\mathcal{O}_\eta$ and $0\leq\theta\leq1$. 
We obtain
\begin{equation*}
s\lambda^2\int_{\mathcal{O}_\eta}e^{2s\alpha}\alpha^2\left|\nabla\psi \right|^2  \ dy\leq  s\lambda^2\int_{\mathcal{O}}e^{2s\alpha}\theta\alpha^2\left|\nabla\psi \right|^2  \ dy,
\end{equation*} 
whence
\begin{multline*}
s\lambda^2\int_{\mathcal{O}}e^{2s\alpha}\theta\alpha^2\left|\nabla\psi \right|^2  \ dy=-s\lambda^2\int_{\mathcal{O}}e^{2s\alpha}\theta\alpha^2\Delta\psi\cdot\psi  \ dy-s\lambda^2\int_{\mathcal{O}}e^{2s\alpha}\alpha^2(\nabla\theta\cdot\nabla)\psi\cdot\psi \ dy\\-2s\lambda^3\int_{\mathcal{O}}e^{2s\alpha}\theta\alpha^2(\nabla\eta\cdot\nabla)\psi\cdot\psi  \ dy-2s^2\lambda^3\int_{\mathcal{O}}e^{2s\alpha}\theta\alpha^3\nabla\eta\nabla\psi\cdot\psi  \ dy\\\leq C\bigg(\int_\mathcal{O}e^{2s\alpha}\left|\Delta\psi\right|^2  \ dy+s^3\lambda^4\int_{\mathcal{O}}e^{2s\alpha}\alpha^4\left|\psi\right|^2  \ dy+\varepsilon s\lambda^2\int_{\mathcal{O}}e^{2s\alpha}\theta\alpha^2\left|\nabla\psi \right|^2  \ dy\bigg).
\end{multline*}
Thus, we get
\begin{multline*}
s\lambda^2\int_{\mathcal{F}}e^{2s\alpha}\alpha^2 \left|\nabla \psi\right|^2   \ dy	+s^3\lambda^4\int_\mathcal{F} e^{2s\alpha}\alpha^4\left|\psi\right|^2  \ d\Gamma +s^2\lambda^2\beta e^{2s}\int_{\partial\mathcal{F}}\left|\psi_\tau\right|^2d\Gamma \\\leq C\bigg(s^3\lambda^4\int_{\mathcal{O}}e^{2s\alpha}\alpha^4\left| \psi\right|^2   \ dy+\int_{\mathcal{F}} e^{2s\alpha}\alpha\left| f\right|^2  \ dy +s^2\lambda^2e^{2s}\int_{\partial\mathcal{F}}\left| \nabla (a\cdot n)\tau\right| ^2 d\Gamma\\+s^3\lambda^2e^{2s}\left\|\psi_n\right\|^2_{[H^{1/2}(\partial\mathcal{F})]^2}+s^2\lambda^2e^{2s}\int_{\partial\mathcal{F}} \left|b\right|^2 \ d\Gamma\bigg),
\end{multline*}
for $\lambda$ and $s$ sufficiently large. Thus, we obtain \eqref{carlS}.
\end{proof}
\section{Carleman estimate for the linearized system}
\label{carlm}
We consider the following adjoint system
\begin{equation}
\label{aSl1}
\left\{
\begin{array}{rl}
-\partial_t v-\nabla \cdot \mathbb{T}(v,q)=F_1  & \text { in }(0,T)\times\mathcal{F},\\
\nabla \cdot v=0& \text { in }(0,T)\times\mathcal{F},\\
m\ell_v'(t)=\int_{\partial \mathcal{S}}\mathbb{T}( v, q)n\ d\Gamma+F_2  & t\in(0,T),\\
J(k_v'(t))=\int_{\partial \mathcal{S}}y^\perp\cdot\mathbb{T}( v, q)n \ d\Gamma+F_3& t\in(0,T),\\
\end{array}
\right.
\end{equation}
with the boundary conditions
\begin{equation}
\label{aSNBCl1}
\left\{
\begin{array}{rl}
v_n = 0  &  \text { on }(0,T)\times\partial\Omega,\\
\left[ 2\nu D( v) n+\beta_{\Omega}  v\right]_{\tau}=0&  \text { on }(0,T)\times\partial\Omega,\\
( v- v_\mathcal{S})_n=0&   \text { on }(0,T)\times\partial \mathcal{S},\\
\left[ 2\nu D( v) n + \beta_{\mathcal{S}} ( v- v_\mathcal{S})\right]_{\tau}=0 &  \text { on }(0,T)\times\partial \mathcal{S},
\end{array}
\right.
\end{equation}
where $ v_\mathcal{S}(y)=\ell_v+k_v y^\perp$, completed with the initial condition
\begin{equation}
\label{aSIl1}
v(T,\cdot) = v_T,\; \text { in }\mathcal{F},\quad
\ell_v(T)=\ell_T,\quad k_v(T)=k_T.
\end{equation}
Let $\eta\in C^2(\overline{\mathcal{F}})$ which verifies \eqref{eta} with $\mathcal{O}_\eta\subset\subset\mathcal{O}$ a non empty open set. 

Let $\lambda>0$ and 

\begin{equation}
\label{beta}
\beta(t,y)=\frac{e^{\lambda(2N+2) \left\|\eta\right\|_{L^\infty(\Omega)}} -e^{\lambda(2N\left\|\eta\right\|_{L^\infty(\Omega)}+\eta(y))}}{t^N(T-t)^N},
\end{equation}
$$
\widehat{\beta}(t)=\max_{y\in \overline{\mathcal{F}}} \beta(t,y)=\frac{e^{\lambda(2N+2) \left\|\eta\right\|_{L^\infty(\Omega)}} -e^{\lambda2N\left\|\eta\right\|_{L^\infty(\Omega)}}}{t^N(T-t)^N},\quad
\beta^*(t)=\min_{y\in \overline{\mathcal{F}}} \beta(t,y),
$$

\begin{equation}
\label{ksi}
\xi(t,y)=\frac{e^{\lambda(2N\left\|\eta\right\|_{L^\infty(\Omega)}+\eta(y))}}{t^N(T-t)^N},\quad \xi^*(t)=\min_{y\in \overline{\mathcal{F}}}\xi(t,y)=\frac{e^{\lambda 2N\left\|\eta\right\|_{L^\infty(\Omega)}}}{t^N(T-t)^N},\quad \widehat{\xi}(t)=\max_{y\in \overline{\mathcal{F}}}\xi(t,y).
\end{equation}
with $N > 0$ an integer number to be defined later
on.

\begin{Proposition}
Assume that $\beta_\mathcal{S}>0$. 
	
There exist $C=C(\Omega,\mathcal{O})$ and $\overline{C}=\overline{C}(\Omega,\mathcal{O})$ such that for any $F_1\in L^2(0,T;[L^2(\mathcal{F})]^2)$, $F_2\in [L^2(0,T)]^2$ and $F_3\in L^2(0,T)$ and for any $v_T\in [L^2(\mathcal{F})]^2$, $\ell_T\in\mathbb{R}^2$ and $k_T\in\mathbb{R}$, the solution $(v,\ell_v,k_v)$ of the system \eqref{aSl1}-\eqref{aSIl1} verifies the inequality 
\begin{multline}
\label{carl}
s^3\lambda^4\int_0^T\int_{\mathcal{F}}e^{-5s\widehat{\beta}}(\xi^*)^3\left|\nabla  v\right|^2   \ dydt+s^4\lambda^4\int_0^T\int_{\mathcal{F}}e^{-5s\widehat{\beta}}(\xi^*)^4\left| v\right|^2  \ dydt\\+s^3\lambda^3\int_0^Te^{-5s\widehat{\beta}}(\xi^*)^3\left( \left|\ell_v(t)\right|^2+\left| k_v(t) \right|^2  \right) dt\leq \overline{C} \bigg(s^5\lambda^6 \int_0^T\int_{\mathcal{O}} e^{-2s\beta^*-3s\widehat{\beta}}(\widehat{\xi})^5\left| v \right|^2  \ dydt\\+\int_0^T\int_{\mathcal{F}}e^{-3s\widehat{\beta}}\left| F_1 \right|^2  \ dydt+ \int_0^Te^{-3s\widehat{\beta}}\left| F_2 \right|^2 \ dt+ \int_0^Te^{-3s\widehat{\beta}}\left| F_3 \right|^2 \ dt\bigg) .
\end{multline}
for all $\lambda\geq C$ and $s\geq C(T^N+T^{2N})$ with $N\geq 4$.
\end{Proposition}
\begin{proof}

\textbf{Step 1: Decomposition of the solution}
	
Let $\rho(t)=e^{-\frac{3}{2}s\widehat{\beta}(t)}$ and let us write 
\begin{equation}
\label{decomp}
\rho v=\varphi+z, \quad\rho \ell_v=\ell_\varphi+\ell_z,\quad\rho k_v=k_\varphi+k_z,\quad\rho q=q_\varphi+q_z,
\end{equation} where
\begin{equation}
\label{system1}
\left\{
\begin{array}{rl}
-\partial_t \varphi-\nabla \cdot \mathbb{T}(\varphi,q_\varphi)=-\rho'v  & \text { in }(0,T)\times\mathcal{F},\\
\nabla \cdot \varphi=0& \text { in }(0,T)\times\mathcal{F},\\
m\ell_\varphi'(t)=\int_{\partial \mathcal{S}}\mathbb{T}( \varphi, q_\varphi)n\ d\Gamma+m\rho'\ell_v & t\in(0,T),\\
J(k_\varphi'(t))=\int_{\partial \mathcal{S}}y^\perp\cdot\mathbb{T}( \varphi, q_\varphi)n \ d\Gamma+J\rho'k_v & t\in(0,T),\\
\end{array}
\right.
\end{equation}
with the boundary conditions
\begin{equation}
\label{boundary1}
\left\{
\begin{array}{rl}
\varphi_n = 0  &  \text { on }(0,T)\times\partial\Omega,\\
\left[ 2\nu D(\varphi) n+\beta_{\Omega}  \varphi\right]_{\tau}=0&  \text { on }(0,T)\times\partial\Omega,\\
(\varphi- \varphi_\mathcal{S})_n=0&   \text { on }(0,T)\times\partial \mathcal{S},\\
\left[ 2\nu D(\varphi) n + \beta_{\mathcal{S}} (\varphi- \varphi_\mathcal{S})\right]_{\tau}=0 &  \text { on }(0,T)\times\partial \mathcal{S},
\end{array}
\right.
\end{equation}
where $ \varphi_\mathcal{S}(y)=\ell_\varphi+k_\varphi y^\perp$, with
\begin{equation*}
\varphi(T,\cdot) = 0,\; \text{ in } \mathcal{F},\quad
\ell_\varphi(T)=0,\quad k_\varphi(T)=0.
\end{equation*}	
and 
\begin{equation}
\label{system2}
\left\{
\begin{array}{rl}
-\partial_t z-\nabla \cdot \mathbb{T}(z,q_z)=\rho F_1  & \text { in }(0,T)\times\mathcal{F},\\
\nabla \cdot z=0& \text { in }(0,T)\times\mathcal{F},\\
m\ell_z'(t)=\int_{\partial \mathcal{S}}\mathbb{T}( z, q_z)n\ d\Gamma+\rho F_2  & t\in(0,T),\\
J(k_z'(t))=\int_{\partial \mathcal{S}}y^\perp\cdot\mathbb{T}( z, q_z)n \ d\Gamma+\rho F_3& t\in(0,T),\\
\end{array}
\right.
\end{equation}
with the boundary conditions
\begin{equation*}
\left\{
\begin{array}{rl}
z_n = 0  &  \text { on }(0,T)\times\partial\Omega,\\
\left[ 2\nu D(z) n+\beta_{\Omega}  z\right]_{\tau}=0&  \text { on }(0,T)\times\partial\Omega,\\
(z- z_\mathcal{S})_n=0&   \text { on }(0,T)\times\partial \mathcal{S},\\
\left[ 2\nu D(z) n + \beta_{\mathcal{S}} (z- z_\mathcal{S})\right]_{\tau}=0 &  \text { on }(0,T)\times\partial \mathcal{S},
\end{array}
\right.
\end{equation*}
where $ z_\mathcal{S}(y)=\ell_z+k_z y^\perp$, completed with the initial condition
\begin{equation*}
z(T,\cdot) = 0,\; \text{ in }\mathcal{F},\quad
\ell_z(T)=0,\quad k_z(T)=0.
\end{equation*}	
Using Theorem \ref{regu}, we have
\begin{multline}
\label{reg}
\left\|z\right\|_{H^1(0,T;[L^2(\mathcal{F})]^2)}+\left\|z\right\|_{L^2(0,T;[H^2(\mathcal{F})]^2)}+\left\|\ell_z \right\|_{[H^1(0,T)]^2}+ \left\|k_z \right\|_{H^1(0,T)}\\\leq C \left( \left\| \rho F_1\right\|_{L^2(0,T;[L^2(\mathcal{F})]^2)} +\left\|\rho F_2\right\|_{[L^2(0,T)]^2}+\left\|\rho
F_3\right\|_{L^2(0,T)}    \right)  .
\end{multline}
\textbf{Step 2:}
In this part, we are going to obtain a Carleman estimate for the system \eqref{system1} by following the proof in \cite{MG}. However, we need to deal with the Navier boundary conditions \eqref{boundary1}.
 
We apply the curl operator to the first equation of \eqref{system1} in order to eliminate the pressure, to get
\begin{equation}
\label{eq1}
-\partial_t(\nabla\times\varphi)-\nu\Delta(\nabla \times\varphi)=-\rho'\nabla\times v\quad \text{in } (0,T)\times\mathcal{F}.
\end{equation}
We obtain a one dimensional heat equation. We recall that 
$$
\nabla\times\varphi=\frac{\partial \varphi_2}{\partial y_1}-\frac{\partial \varphi_1}{\partial y_2}.
$$ 
We apply Proposition \ref{pheat} replacing $\psi$ by $\nabla\times\varphi$. We get 
\begin{multline}
\label{eq9}
s\lambda^2\int_0^T\int_{\mathcal{F}}e^{-2s\beta}\xi\left|\nabla (\nabla\times\varphi)\right|^2  \ dy  dt+s^3\lambda^4\int_0^T\int_{\mathcal{F}}e^{-2s\beta}\xi^3\left|\nabla\times\varphi\right|^2 \ dy dt \\+s^3\lambda^3\int_0^T\int_{\partial\mathcal{F}}e^{-2s\widehat{\beta}}(\xi^*)^{3}\left|\nabla\times\varphi\right|^2d\Gamma dt \leq C\bigg( s^3\lambda^4\int_0^T\int_{\mathcal{O}_\eta} e^{-2s\beta}\xi^3\left|\nabla\times\varphi\right|^2 \ dy dt \\+s\lambda^2\int_0^T\int_{\mathcal{O}_\eta}e^{-2s\beta} \xi \left|\nabla(\nabla\times\varphi)\right|^2 \ dy dt+s\lambda\int_0^T\int_{\partial\mathcal{F}}e^{-2s\widehat{\beta}}\xi^*\left| \nabla (\nabla\times\varphi)\tau\right| ^2 d\Gamma dt\\+s^{-1}\lambda^{-1}\int_0^T\int_{\partial\mathcal{F}}e^{-2s\widehat{\beta}}(\xi^*)^{-1}\left|\nabla\times\partial_t\varphi\right|^2 \ d\Gamma dt+ \int_0^T\int_{\mathcal{F}}e^{-2s\beta}(\rho')^2\left|  \nabla\times v\right|^2  \ dydt  \bigg).
\end{multline}
Arguing as \cite[pp.7-8]{MG}, we treat the local terms appearing in the right hand side of \eqref{eq9}, we obtain
\begin{multline}
\label{eq4}
s\lambda^2\int_0^T\int_{\mathcal{F}} e^{-2s\beta} \xi\left|\nabla(\nabla\times\varphi)\right| ^2  \ dy dt+s^3\lambda^4\int_0^T\int_{\mathcal{F}} e^{-2s\beta}\xi^3\left|\nabla\times\varphi\right|^2  \ dydt\\+s^3\lambda^3\int_0^T\int_{\partial\mathcal{F}} e^{-2s\widehat{\beta}} (\xi^*)^3\left|\nabla\times\varphi \right| ^2 \ d\Gamma dt  \leq C\bigg( s^5\lambda^6\int_0^T\int_{\mathcal{O}}e^{-2s\beta}\xi^5\left|\varphi\right|^2  \ dy dt \\+s\lambda\int_0^T\int_{\partial\mathcal{F}}e^{-2s\widehat{\beta}}\xi^*\left|\nabla(\nabla\times\varphi)\tau\right|^2 \ d\Gamma dt +s^{-1}\lambda^{-1}\int_0^T\int_{\partial\mathcal{F}}e^{-2s\widehat{\beta}}(\xi^*)^{-1}\left|\nabla \times\partial_t\varphi \right|^2\ d\Gamma dt\\+ \int_0^T\int_{\mathcal{F}}e^{-2s\beta}(\rho')^2\left|  \nabla\times v\right|^2  \ dydt  
\bigg),
\end{multline}
for $\lambda\geq C$ and $s\geq C(T^N+T^{2N})$.
We notice that $\varphi$ satisfies the following problem
\begin{equation}
\label{sy}
\left\{
\begin{array}{rl}
\Delta \varphi=-\nabla\times (\nabla\times\varphi)&\text{in}\quad\mathcal{F}, \\
\nabla \cdot \varphi=0& \text{in}\quad\mathcal{F},\\  
\end{array}
\right.
\end{equation}
with the boundary conditions \eqref{boundary1}. 
Applying the Carleman inequality proved in Proposition \ref{CS}, we obtain 
\begin{multline}
\label{eq2}
s^2\lambda^2\int_{\Omega}e^{2se^{\lambda\eta}}e^{2\lambda\eta}\left|\nabla \varphi\right|^2   \ dy+s^4\lambda^4\int_{\Omega}e^{2se^{\lambda\eta}}e^{4\lambda\eta}\left|\varphi\right|^2  \ dy \\ \leq C\bigg( s^4\lambda^4 \int_{\mathcal{O}} e^{2se^{\lambda\eta}}e^{4\lambda\eta}\left|\varphi \right|^2  \ dy +s\int_{\mathcal{F}} e^{2se^{\lambda\eta}}e^{\lambda \eta}\left| \nabla\times(\nabla\times\varphi)\right|^2  \ dy\\+\beta_\mathcal{S}s^3\lambda^2 e^{2s}\int_{\partial\mathcal{S}}\left|(\varphi_\mathcal{S})_\tau \right|^2 \ d\Gamma +s^4\lambda^2e^{2s}\left\|\varphi_\mathcal{S}\cdot n \right\|_{H^{3/2}(\partial\mathcal{S})}^2\bigg),
\end{multline} 
where we have used that $a=\varphi_\mathcal{S}1_{\partial\mathcal{S}}$ and  $b=\beta_\mathcal{S}(\varphi_\mathcal{S})_\tau1_{\partial\mathcal{S}}$.
We replace $s$ in \eqref{eq2} by $\frac{se^{2N\lambda \left\| \eta\right\|_{L^\infty(\Omega)} }}{t^N(T-t)^N}$. Multiplying \eqref{eq2} by $$
e^{-2s\dfrac{e^{(2N+2)\lambda\left\| \eta\right\|_{L^\infty(\Omega)}}}{t^N(T-t)^N}},
$$ 
and integrating over $(0,T)$, we get
\begin{multline}
\label{eq3}
s^2\lambda^2\int_0^T\int_{\mathcal{F}}e^{-2s\beta}\xi^2\left|\nabla \varphi\right|^2   \ dydt+s^4\lambda^4\int_0^T\int_{\mathcal{F}}e^{-2s\beta}\xi^4\left|\varphi\right|^2  \ dydt \leq C\bigg( s^4\lambda^4 \int_0^T\int_{\mathcal{O}} e^{-2s\beta}\xi^3\left|\varphi \right|^2  \ dydt\\+s\int_0^T\int_{\mathcal{F}} e^{-2s\beta}\xi\left| \nabla\times(\nabla\times\varphi)\right|^2  \ dydt+\beta_\mathcal{S}s^3\lambda^2\int_0^Te^{-2s\widehat{\beta}}(\xi^*)^3\int_{\partial\mathcal{S}}\left|(\varphi_\mathcal{S})_\tau \right|^2 \ d\Gamma\ dt \\+s^4\lambda^2\int_0^Te^{-2s\widehat{\beta}}(\xi^*)^4\left\|\varphi_\mathcal{S}\cdot n \right\|_{H^{3/2}(\partial\mathcal{S})}^2\ dt \bigg).
\end{multline}
Applying the estimates obtained in \cite[Theorem 2.2]{MR2802919}, we get
\begin{equation}
\label{eq10}
\left\|\varphi\right\|^2_{[H^1(\mathcal{F})]^2}\leq C\left( \left\|\varphi\right\|^2_{[L^2(\mathcal{F})]^2}+\left\|\nabla\cdot \varphi\right\|^2_{L^2(\mathcal{F})}+\left\|\nabla\times\varphi\right\|^2_{L^2(\mathcal{F})}+
\left\| \varphi_\mathcal{S}\cdot n\right\|_{H^{1/2}(\partial\mathcal{S})}^2  \right).  
\end{equation} 
Then, we multiply \eqref{eq10} by $s^3\lambda^4e^{-2s\widehat{\beta}}(\xi^*)^3$, we get
\begin{multline}
\label{eq8}
s^3\lambda^4\int_0^T\int_{\mathcal{F}}e^{-2s\widehat{\beta}}(\xi^*)^3\left|\nabla\varphi\right|^2  \ dy dt\leq C\bigg(s^3\lambda^4\int_0^T\int_{\mathcal{F}} e^{-2s\widehat{\beta}}(\xi^*)^3\left|\varphi\right|^2 \ dydt\\+ s^3\lambda^4\int_0^T\int_{\mathcal{F}} e^{-2s\widehat{\beta}}(\xi^*)^3\left|\nabla\times\varphi\right|^2 \ dydt+s^3\lambda^4\int_0^Te^{-2s\widehat{\beta}}(\xi^*)^3\left\|\varphi_\mathcal{S}\cdot n \right\|_{H^{1/2}(\partial\mathcal{S})}^2\ dt \bigg).
\end{multline}
	
Adding \eqref{eq8}, \eqref{eq3} and \eqref{eq4}, we deduce
\begin{multline}
\label{eq5}
s^3\lambda^4\int_0^T\int_{\mathcal{F}}e^{-2s\widehat{\beta}}(\xi^*)^3\left|\nabla \varphi\right|^2   \ dydt+s^4\lambda^4\int_0^T\int_{\mathcal{F}}e^{-2s\beta}\xi^4\left|\varphi\right|^2  \ dydt+	s\lambda^2\int_0^T\int_{\mathcal{F}}e^{-2s\beta}\xi\left|\nabla (\nabla\times\varphi)\right|^2  \ dy  dt\\+s^3\lambda^3\int_0^T\int_{\partial\mathcal{F}} e^{-2s\widehat{\beta}} (\xi^*)^3\left|\nabla\times\varphi \right| ^2 \ d\Gamma dt \leq C\bigg( s^5\lambda^6 \int_0^T\int_{\mathcal{O}} e^{-2s\beta}\xi^5\left|\varphi \right|^2  \ dydt\\+s\lambda\int_0^T\int_{\partial\mathcal{F}}e^{-2s\widehat{\beta}}\xi^*\left|\nabla(\nabla\times\varphi)\tau\right|^2 \ d\Gamma dt +s^{-1}\lambda^{-1}\int_0^T\int_{\partial\mathcal{F}}e^{-2s\widehat{\beta}}(\xi^*)^{-1}\left|\nabla \times\partial_t\varphi \right|^2\ d\Gamma dt\\+\beta_\mathcal{S}s^3\lambda^2\int_0^T\int_{\partial\mathcal{S}}e^{-2s\widehat{\beta}}(\xi^*)^3\left|(\varphi_\mathcal{S})_\tau \right|^2 \ d\Gamma\ dt +s^4\lambda^2\int_0^Te^{-2s\widehat{\beta}}(\xi^*)^4\left\|\varphi_\mathcal{S}\cdot n \right\|_{H^{3/2}(\partial\mathcal{S})}^2\ dt\\+\int_0^T\int_{\mathcal{F}}e^{-2s\beta}(\rho')^2\left|  \nabla\times v\right|^2  \ dydt   \bigg).
\end{multline}
Taking $(s,\lambda)$ large enough, the fifth term in the right hand side of \eqref{eq5} can be transported to the left side. Indeed, since $\varphi_\mathcal{S}$ is rigid, from \cite[Lemma 2.2]{royT}, we have 
\begin{equation}
\label{eq21}
\int_{\mathcal{F}}\left|\varphi(t,\cdot)\right|^2  \ dy\geq C \left\|\varphi_\mathcal{S}(t)\cdot n \right\|_{H^{3/2}(\partial\mathcal{S})}^2,
\end{equation}
for any shape of the body $\mathcal{S}$.
Moreover, we have the following relation
\begin{equation}
\label{cc1}
\nabla\times\varphi=(\nabla\varphi n)\cdot\tau-((\nabla\varphi)^*n)\cdot \tau,\quad \text{on }\partial\mathcal{F},
\end{equation}
where $\tau=\begin{pmatrix}
-n_2\\ n_1
\end{pmatrix}$.
In the other hand, we have
\begin{equation}
\label{cc2}
((\nabla\varphi)^* n)\cdot\tau=\sum_{i,j}\partial_i\varphi_jn_j\tau_i=\sum_{i,j}\partial_i(\varphi_jn_j)\tau_i-\sum_{i,j}\partial_in_j\varphi_j\tau_i=\nabla(\varphi\cdot n)\cdot\tau-\nabla n\tau\cdot\varphi,\quad \text{on }\partial\mathcal{F}.
\end{equation}
Using the boundary conditions \eqref{boundary1}, we can write
\begin{equation}
\label{cc4}
\beta_\mathcal{S}(\varphi_\mathcal{S})_\tau=\nu\left(\nabla\varphi n+(\nabla\varphi)^*n \right)_\tau+\beta_\mathcal{S}\varphi_\tau,\quad \text{on } \partial\mathcal{S}.
\end{equation}
Using \eqref{cc4}, \eqref{cc2} and \eqref{cc1} ,
we get 
\begin{equation}
\label{cc3}
\beta_{\mathcal{S}}\int_{\partial\mathcal{S}}\left|(\varphi_\mathcal{S})_\tau\right|^2 \ d\Gamma\leq C \left(\int_{\partial\mathcal{F}}\left|\nabla\times\varphi\right|^2 \ d\Gamma+\int_{\partial\mathcal{F}}\left|\varphi\right|^2 \ d\Gamma +\left\|\varphi_\mathcal{S}\cdot n \right\|_{H^{3/2}(\partial\mathcal{S})}^2\right). 
\end{equation}
Multiplying \eqref{cc3} by $s^3\lambda^3e^{-2s\widehat{\beta}} (\xi^*)^3$ and integrating over $(0,T)$, we obtain
\begin{multline}
\label{eq19}
s^3\lambda^3\int_0^T\int_{\partial\mathcal{S}} e^{-2s\widehat{\beta}} (\xi^*)^3\left|(\varphi_\mathcal{S})_\tau \right| ^2 \ d\Gamma dt\leq C\bigg(  s^3\lambda^3\int_0^T\int_{\partial\mathcal{F}} e^{-2s\widehat{\beta}} (\xi^*)^3\left|\nabla\times\varphi \right| ^2 \ d\Gamma dt\\+s^3\lambda^3\int_0^T\int_{\partial\mathcal{F}} e^{-2s\widehat{\beta}} (\xi^*)^3\left|\varphi \right| ^2 \ d\Gamma dt+s^3\lambda^3\int_0^T e^{-2s\widehat{\beta}} (\xi^*)^3\left\|\varphi_\mathcal{S}\cdot n \right\|_{H^{3/2}(\partial\mathcal{S})}^2 \ dt\bigg) .
\end{multline}
Using the inequality in \cite[Theorem II.4.1]{galdi} with $r=2$, $q=2$, we obtain
\begin{equation}
\label{eq20}
s^3\lambda^3\int_{\partial\mathcal{F}} e^{-2s\widehat{\beta}} (\xi^*)^3\left|\varphi\right|^2 \ d\Gamma \leq C \left( s^3\lambda^3 \int_\mathcal{F}e^{-2s\widehat{\beta}} (\xi^*)^3\left|\nabla \varphi\right|^2 \ dy + s^3\lambda^3\int_\mathcal{F} e^{-2s\widehat{\beta}} (\xi^*)^3\left|\varphi\right|^2  \ dy\right) .
\end{equation}
Thus, adding \eqref{eq20}, \eqref{eq19}, \eqref{eq5} and using \eqref{eq21}, we get 
\begin{multline}
\label{17}
s^3\lambda^4\int_0^T\int_{\mathcal{F}}e^{-2s\widehat{\beta}}(\xi^*)^3\left|\nabla \varphi\right|^2   \ dydt+s^4\lambda^4\int_0^T\int_{\mathcal{F}}e^{-2s\beta}\xi^4\left|\varphi\right|^2  \ dydt+s^3\lambda^3\int_0^T\int_{\partial\mathcal{F}}e^{-2s\widehat{\beta}}(\xi^*)^3\left|\varphi_{\mathcal{S}}(t)  \right|^2\ d\Gamma dt\\+s\lambda^2\int_0^T\int_{\mathcal{F}}e^{-2s\beta}\xi\left|\nabla (\nabla\times\varphi)\right|^2  \ dy  dt\leq C\bigg( s^5\lambda^6 \int_0^T\int_{\mathcal{O}} e^{-2s\beta}\xi^5\left|\varphi \right|^2  \ dydt\\+s\lambda\int_0^T\int_{\partial\mathcal{F}}e^{-2s\widehat{\beta}}\xi^*\left|\nabla(\nabla\times\varphi)\tau\right|^2 \ d\Gamma dt +s^{-1}\lambda^{-1}\int_0^T\int_{\partial\mathcal{F}}e^{-2s\widehat{\beta}}(\xi^*)^{-1}\left|\nabla \times\partial_t\varphi \right|^2\ d\Gamma dt\\+\int_0^T\int_{\mathcal{F}}e^{-2s\beta}(\rho')^2\left|  \nabla\times v\right|^2  \ dydt   \bigg).
\end{multline}
Using \cite[lemma 1, section 4.1]{MR2526402}, we have that
\begin{equation}
\label{imeno}
\int_{\partial\mathcal{F}}\left|\varphi_{\mathcal{S}}(t)  \right|^2\ d\Gamma\geq C\left(\left|\ell_\varphi(t)\right|^2+\left| k_\varphi(t) \right|^2 \right). 
\end{equation}
Using \eqref{imeno}, \eqref{eq21} and combining with \eqref{17}, we obtain
\begin{multline}
\label{eq6}
s^3\lambda^4\int_0^T\int_{\mathcal{F}}e^{-2s\widehat{\beta}}(\xi^*)^3\left|\nabla \varphi\right|^2   \ dydt+s^4\lambda^4\int_0^T\int_{\mathcal{F}}e^{-2s\beta}\xi^4\left|\varphi\right|^2  \ dydt+s^3\lambda^3\int_0^Te^{-2s\widehat{\beta}}(\xi^*)^3 \left|\ell_\varphi(t)\right|^2\ dt\\+s\lambda^2\int_0^T\int_{\mathcal{F}}e^{-2s\beta}\xi\left|\nabla (\nabla\times\varphi)\right|^2  \ dy  dt+s^3\lambda^3\int_0^Te^{-2s\widehat{\beta}}(\xi^*)^3\left| k_\varphi(t) \right|^2\ dt\leq C\bigg( s^5\lambda^6 \int_0^T\int_{\mathcal{O}} e^{-2s\beta}\xi^5\left|\varphi \right|^2  \ dydt\\+s\lambda\int_0^T\int_{\partial\mathcal{F}}e^{-2s\widehat{\beta}}\xi^*\left|\nabla(\nabla\times\varphi)\tau\right|^2 \ d\Gamma dt +s^{-1}\lambda^{-1}\int_0^T\int_{\partial\mathcal{F}}e^{-2s\widehat{\beta}}(\xi^*)^{-1}\left|\nabla \times\partial_t\varphi \right|^2\ d\Gamma dt\\+\int_0^T\int_{\mathcal{F}}e^{-2s\beta}(\rho')^2\left|  \nabla\times v\right|^2  \ dydt   \bigg).
\end{multline}
	
Let us deal with the last term in the right hand side of \eqref{eq6}.
Noticing that $ \left|\rho' \right|\leq C s\rho (\xi^*)^{1+1/N}  $ and using \eqref{decomp}, we get
\begin{multline}
\int_0^T\int_{\mathcal{F}}e^{-2s\beta}(\rho')^2\left|  \nabla\times v\right|^2  \ dydt= \int_0^T\int_{\mathcal{F}}e^{-2s\beta}(\rho')^2(\rho)^{-2}\left|  \nabla\times \rho v\right|^2  \ dydt\\\leq Cs^2\int_0^T\int_{\mathcal{F}}e^{-2s\beta}(\xi^*)^3\left( \left|  \nabla\times \varphi\right|^2+\left|  \nabla\times z\right|^2\right)   \ dydt,
\end{multline}
for $N\geq 2$. Using that $ s^2e^{-2s\beta}(\xi^*)^3$ is bounded, applying \eqref{reg}, the inequality \eqref{eq6} is reduced to 
\begin{multline}
\label{eq14}
s^3\lambda^4\int_0^T\int_{\mathcal{F}}e^{-2s\widehat{\beta}}(\xi^*)^3\left|\nabla \varphi\right|^2   \ dydt+s^4\lambda^4\int_0^T\int_{\mathcal{F}}e^{-2s\beta}\xi^4\left|\varphi\right|^2  \ dydt+	s\lambda^2\int_0^T\int_{\mathcal{F}}e^{-2s\beta}\xi\left|\nabla (\nabla\times\varphi)\right|^2  \ dy  dt\\+s^3\lambda^3\int_0^Te^{-2s\widehat{\beta}}(\xi^*)^3 \left|\ell_\varphi(t)\right|^2 \ dt+s^3\lambda^3\int_0^Te^{-2s\widehat{\beta}}(\xi^*)^3\left| k_\varphi(t) \right|^2  \ dt \leq C\bigg( s^5\lambda^6 \int_0^T\int_{\mathcal{O}} e^{-2s\beta}\xi^5\left|\varphi \right|^2  \ dydt\\+s\lambda\int_0^T\int_{\partial\mathcal{F}}e^{-2s\widehat{\beta}}\xi^*\left|\nabla(\nabla\times\varphi)\tau\right|^2 \ d\Gamma dt +s^{-1}\lambda^{-1}\int_0^T\int_{\partial\mathcal{F}}e^{-2s\widehat{\beta}}(\xi^*)^{-1}\left|\nabla \times\partial_t\varphi \right|^2\ d\Gamma dt\\+\int_0^T\int_{\mathcal{F}}e^{-3s\widehat{\beta}}\left| F_1\right|^2  \ dydt+\int_0^Te^{-3s\widehat{\beta}} (\left| F_2\right|^2+\left| F_3\right|^2) \ dt \bigg),
\end{multline}
for $\lambda\geq C$ and $s\geq C(T^N+T^{2N})$.
	
\textbf{Step 3:}
Now, it remains to treat the two terms
\begin{equation}
\label{2terms}
s\lambda\int_0^T\int_{\partial\mathcal{F}}e^{-2s\widehat{\beta}}\xi^*\left|\nabla(\nabla\times\varphi)\tau\right|^2 \ d\Gamma dt,\quad s^{-1}\lambda^{-1}\int_0^T\int_{\partial\mathcal{F}}e^{-2s\widehat{\beta}}(\xi^*)^{-1}\left|\nabla \times\partial_t\varphi \right|^2\ d\Gamma dt.
\end{equation}
Using \eqref{cc1}, \eqref{cc2}, and the fact that
$$
(\nabla\varphi n)\cdot\tau=-\frac{\beta_\mathcal{S}}{\nu} (\varphi-\varphi_\mathcal{S})\cdot\tau-((\nabla\varphi)^* n)\cdot\tau,\quad \text{ on } \partial\mathcal{S},
$$
and
$$
(\nabla\varphi n)\cdot\tau=-\frac{\beta_\Omega}{\nu} \varphi\cdot\tau-((\nabla\varphi)^* n)\cdot\tau ,\quad \text{ on } \partial\Omega,
$$
we get,
\begin{multline}
\label{lol}
\nabla\times\varphi=-\frac{\beta_\mathcal{S}}{\nu}(\varphi-\varphi_S)\cdot\tau-2((\nabla\varphi)^*n)\cdot \tau\\=-\frac{\beta_\mathcal{S}}{\nu}(\varphi-\varphi_S)\cdot\tau-2 (\nabla(\varphi_\mathcal{S}\cdot n)\cdot\tau-\nabla n\tau\cdot\varphi),\quad \text{on }\partial\mathcal{S},
\end{multline}
\begin{equation}
\label{lol1}
\nabla\times\varphi=-\frac{\beta_\Omega}{\nu}\varphi\cdot\tau-2 ((\nabla\varphi)^*n)\cdot \tau=-\frac{\beta_\mathcal{S}}{\nu}\varphi\cdot\tau+2 (\nabla n\tau\cdot\varphi),\quad \text{on }\partial\Omega.
\end{equation}
Then, we have
$$
\nabla(\nabla\times\varphi)\tau=\nabla\left[-\frac{\beta_\mathcal{S}}{\nu}(\varphi-\varphi_S)\cdot\tau-2 (\nabla(\varphi_\mathcal{S}\cdot n)\cdot\tau-\nabla n\tau\cdot\varphi)\right]\tau,\quad \text{on }\partial\mathcal{S},
$$
and
$$
\nabla(\nabla\times\varphi)\tau=\nabla\left[-\frac{\beta_\Omega}{\nu}\varphi\cdot\tau+2 (\nabla n\tau\cdot\varphi)\right]\tau,\quad \text{on }\partial\Omega.
$$	
It implies
$$
\left| \nabla(\nabla\times\varphi)\tau\right| \leq C\left(\left| \nabla \varphi\right| +\left| \varphi_\mathcal{S}\right|+\left| \varphi\right| \right) ,\quad \text{on }\partial\mathcal{F}.
$$
Then
\begin{multline}
\label{eq16*}
s\lambda\int_0^T\int_{\partial\mathcal{F}}e^{-2s\widehat{\beta}}\xi^*\left| \nabla(\nabla\times\varphi)\tau\right|^2 \ d\Gamma  dt \leq C\bigg(s\lambda\int_0^T\int_{\partial\mathcal{F}}e^{-2s\widehat{\beta}}\xi^*\left| \nabla \varphi\right|^2  d\Gamma \ dt \\+s\lambda\int_0^T e^{-2s\widehat{\beta}}\xi^*\left( \left| \ell_\varphi\right|^2 +\left|k_\varphi\right|^2 \right)  \ dt+s\lambda\int_0^T\int_{\partial\mathcal{F}}e^{-2s\widehat{\beta}}\xi^*\left| \varphi\right|^2 \ d\Gamma  dt \bigg).
\end{multline}
The second and the third term in the right hand side of the above inequality can be absorbed using \eqref{eq20} by the left side of the inequality \eqref{eq14}.
To absorb the first term in the right hand side of the inequality \eqref{eq16*}, we use the elliptic estimate of the system \eqref{sy}, we obtain
	
\begin{multline}
\label{eq17*}
s\lambda\int_0^T\int_{\partial\mathcal{F}}e^{-2s\widehat{\beta}}\xi^*\left| \nabla \varphi\right|  d\Gamma \ dt  \leq C\bigg(s\lambda\int_0^T\int_{\mathcal{F}}e^{-2s\widehat{\beta}}\xi^*\left| \nabla(\nabla\times\varphi)\right|^2 \ dy dt\\+s\lambda\int_0^T e^{-2s\widehat{\beta}}\xi^*\left( \left| \ell_\varphi\right|^2 +\left|k_\varphi\right|^2 \right)  \ dt \bigg).
\end{multline}
The terms in the right hand side of \eqref{eq17*} can be absorbed by the left hand side of \eqref{eq14}, moreover the last term in the right side of \eqref{eq16*} can be manipulated as \eqref{eq20} and thus, it can be absorbed by the left side of \eqref{eq14}.	
To estimate the second term in \eqref{2terms}, observe that from \eqref{lol} and \eqref{lol1}, we have
\begin{multline}
\label{0}
s^{-1}\lambda^{-1}\int_0^T\int_{\partial\mathcal{F}}e^{-2s\widehat{\beta}}(\xi^*)^{-1}\left|\nabla \times\partial_t\varphi \right|^2\ d\Gamma dt\leq C\bigg(s^{-1}\lambda^{-1}\int_0^T\int_{\partial\mathcal{F}}e^{-2s\widehat{\beta}}(\xi^*)^{-1}\left|\partial_t\varphi \right|^2\ d\Gamma dt\\+s^{-1}\lambda^{-1}\int_0^Te^{-2s\widehat{\beta}}(\xi^*)^{-1}\left( \left| \ell'_\varphi\right|^2 +\left|k'_\varphi\right|^2 \right) \ dt\bigg). 
\end{multline}
	We take $\zeta_2(t)=s^{-1/2}\lambda^{-1/2}e^{-s\beta(t)}(\xi^*)^{-1/2}(t)$ and let consider the system
	\begin{equation*}
	\left\{
	\begin{array}{rl}
	-\partial_t (\zeta_2\varphi)-\nabla \cdot \mathbb{T}(\zeta_2\varphi,\zeta_2 q_\varphi)=-\zeta_2'\varphi -\zeta_2\rho' v & \text { in }(0,T)\times\mathcal{F},\\
	\nabla \cdot (\zeta_2\varphi)=0& \text { in }(0,T)\times\mathcal{F},\\
	m(\zeta_2 \ell_\varphi)'=\int_{\partial S}\mathbb{T}(\zeta_2\varphi,\zeta_2 q_\varphi)n\ d\Gamma+m\zeta_2'\ell_\varphi+m\zeta_2\rho' \ell_v  & t\in(0,T),\\
	J(\zeta_2 k_\varphi)'(t)=\int_{\partial S}y^\perp\cdot\mathbb{T}((\zeta_2\varphi),\zeta_2 q_\varphi)n \ d\Gamma+J\zeta_2'k_\varphi+J\zeta_2\rho' k_v& t\in(0,T),\\
	\end{array}
	\right.
	\end{equation*}
	with the boundary condition
	\begin{equation*}
	\left\{
	\begin{array}{rl}
	(\zeta_2\varphi)_n = 0  &  \text { on }(0,T)\times\partial\Omega,\\
	\left[ 2\nu D(\zeta_2\varphi) n+\beta_{\Omega} (\zeta_2\varphi)\right]_{\tau}=0&  \text { on }(0,T)\times\partial\Omega,\\
	(\zeta_2\varphi-\zeta_2 \varphi_\mathcal{S})_n=0&   \text { on }(0,T)\times\partial S,\\
	\left[ 2\nu D(\zeta_2\varphi) n + \beta_{\mathcal{S}} (\zeta_2\varphi-\zeta_2\varphi_\mathcal{S}  )\right]_{\tau}=0 &  \text { on }(0,T)\times\partial S,
	\end{array}
	\right.
	\end{equation*}
	completed with the initial condition
	\begin{equation*}
	(\zeta_2\varphi)(T,\cdot) =0,\; \text{ in } \mathcal{F},\quad
	(\zeta_2 \ell_\varphi)(T)=0, \quad (\zeta_2 k_\varphi)(T)=0.
	\end{equation*}
	We notice that in the above system all final conditions are equal to zero, then all the compatibility conditions mentioned in Proposition \ref{prop} are satisfied. To absorb the last two terms in the right hand side of \eqref{0}, we use $L^2$ regularity results of the system satisfied by $\zeta_2\varphi$. In fact, we have
	\begin{multline}
	\label{int*}
\left\|\zeta_2\varphi\right\|^2_{L^2(0,T;[H^{2}(\mathcal{F})]^2)}+\left\|\zeta_2\varphi\right\|^2_{H^1(0,T;[L^2(\mathcal{F})]^2)}+\left\|\zeta_2 \ell_\varphi \right\|^2_{[H^{1}(0,T)]^2}+\left\|\zeta_2 k_\varphi \right\|^2_{H^{1}(0,T)}\\\leq C \bigg( \left\| \zeta_2'\varphi\right\|^2_{L^2(0,T;[L^2(\mathcal{F})]^2)}+\left\|\zeta_2'\ell_\varphi \right\|^2_{[L^2(0,T)]^2}+\left\|\zeta_2'k_\varphi\right\|^2_{L^2(0,T)}\\+\left\| \zeta_2\rho' v\right\|^2_{L^2(0,T;[L^2(\mathcal{F})]^2)}+\left\|\zeta_2\rho '\ell_v \right\|^2_{[L^2(0,T)]^2}+\left\|\zeta_2\rho'k_v \right\|^2_{L^2(0,T)}  \bigg).  
	\end{multline}
We note that
\begin{equation}
\label{ah}
\left|\zeta_2'\right|\leq C s^{1/2}\lambda^{-1/2}(\xi^*)^{1/2+1/N}e^{-s\widehat{\beta}},\quad \left|\zeta_2\rho'\right|\leq C s^{1/2}\lambda^{-1/2}(\xi^*)^{1/2+1/N}e^{-s\widehat{\beta}}\rho.
\end{equation}	
Then, using \eqref{int*} and \eqref{ah}, we obtain
for $N\geq 1$
\begin{multline}
\label{ahh}
s^{-1}\lambda^{-1}\int_0^Te^{-2s\widehat{\beta}}(\xi^*)^{-1}\left( \left| \ell'_\varphi\right|^2 +\left|k'_\varphi\right|^2 \right) \ dt\leq C\bigg(s\lambda^{-1} \int_0^T\int_{\mathcal{F}} e^{-2s\widehat{\beta}} (\xi^*)^4 \left|\varphi\right|^2  \ dydt  \\+s\lambda^{-1}\int_0^Te^{-2s\widehat{\beta}}(\xi^*)^3\left|\ell_\varphi(t)\right|^2 dt+s\lambda^{-1}\int_0^Te^{-2s\widehat{\beta}}(\xi^*)^{3}\left| k_\varphi(t) \right|^2  \ dt\\+s\lambda^{-1} \int_0^T\int_{\mathcal{F}} e^{-2s\widehat{\beta}} (\xi^*)^4 \left|\rho v\right|^2  \ dydt \bigg).
\end{multline}
combining \eqref{decomp}, \eqref{reg} and using the fact that $s\lambda^{-1}e^{-2s\widehat{\beta}} (\xi^*)^4$ is bounded, we get
\begin{multline}
s\lambda^{-1} \int_0^T\int_{\mathcal{F}} e^{-2s\widehat{\beta}} (\xi^*)^4 \left|\rho v\right|^2  \ dydt\leq C \bigg(s\lambda^{-1} \int_0^T\int_{\mathcal{F}} e^{-2s\widehat{\beta}} (\xi^*)^4 \left|\varphi\right|^2  \ dydt+\int_0^T\int_{\mathcal{F}}e^{-3s\widehat{\beta}}\left| F_1 \right|^2  \ dydt\\+ \int_0^Te^{-3s\widehat{\beta}}\left| F_2 \right|^2 \ dt+ \int_0^Te^{-3s\widehat{\beta}}\left| F_3 \right|^2 \ dt\bigg).
\end{multline}
Then, we obtain
\begin{multline}
\label{ahhh}
s^{-1}\lambda^{-1}\int_0^Te^{-2s\widehat{\beta}}(\xi^*)^{-1}\left( \left| \ell'_\varphi\right|^2 +\left|k'_\varphi\right|^2 \right) \ dt\leq C\bigg(s\lambda^{-1} \int_0^T\int_{\partial\mathcal{F}} e^{-2s\widehat{\beta}} (\xi^*)^4 \left|\varphi\right|^2  \ dydt  \\+s\lambda^{-1}\int_0^Te^{-2s\widehat{\beta}}(\xi^*)^3\left|\ell_\varphi(t)\right|^2 dt+s\lambda^{-1}\int_0^Te^{-2s\widehat{\beta}}(\xi^*)^{3}\left| k_\varphi(t) \right|^2  \ dt+\int_0^T\int_{\mathcal{F}}e^{-3s\widehat{\beta}}\left| F_1 \right|^2  \ dydt\\+ \int_0^Te^{-3s\widehat{\beta}}\left| F_2 \right|^2 \ dt+ \int_0^Te^{-3s\widehat{\beta}}\left| F_3 \right|^2 \ dt\bigg).
\end{multline}
Now, we deal with the first term in the right hand side of \eqref{0}. Using \eqref{reg1} and \eqref{reg2} and by interpolation with parameter $3/8$, we get the following estimate
	\begin{multline}
	\label{int}
	\left\|\zeta_2\varphi\right\|^2_{H^1(0,T;[H^{3/4}(\mathcal{F})]^2)}\leq\left\|\zeta_2\varphi\right\|^2_{L^2(0,T;[H^{11/4}(\mathcal{F})]^2)}+\left\|\zeta_2\varphi\right\|^2_{H^{11/8}(0,T;[L^2(\mathcal{F})]^2)}+\left\|\zeta_2 \ell_\varphi \right\|^2_{[H^{11/8}(0,T)]^2}+\\\left\|\zeta_2 k_\varphi \right\|^2_{H^{11/8}(0,T)}\leq C \bigg( \left\| \zeta_2'\varphi\right\|^2_{H^{3/8}(0,T;[L^2(\mathcal{F})]^2)}+\left\| \zeta_2'\varphi\right\|^2_{L^2(0,T;[H^{3/4}(\mathcal{F})]^2)}+\left\|\zeta_2'\ell_\varphi \right\|^2_{[H^{3/8}(0,T)]^2}+\left\|\zeta_2'k_\varphi\right\|^2_{H^{3/8}(0,T)}\\+\left\| \zeta_2\rho' v\right\|^2_{H^{3/8}(0,T;[L^2(\mathcal{F})]^2)}+\left\| \zeta_2\rho' v\right\|^2_{L^{2}(0,T;[H^{3/4}(\mathcal{F})]^2)}+\left\|\zeta_2\rho '\ell_v \right\|^2_{[H^{3/8}(0,T)]^2}+\left\|\zeta_2\rho'k_v \right\|^2_{H^{3/8}(0,T)}  \bigg).  
	\end{multline}
	We note that
	$$
	\left|\zeta_2'\right|\leq C s^{1/2}\lambda^{-1/2}(\xi^*)^{1/2+1/N}e^{-s\widehat{\beta}},\quad
	\left|\zeta_2''\right|\leq C s^{3/2}\lambda^{-1/2}(\xi^*)^{3/2+2/N}e^{-s\widehat{\beta}},
	$$
		and
	$$
	\left|(\zeta_2\rho')'\right| \leq C s^{3/2}\lambda^{-1/2}(\xi^*)^{3/2+2/N}e^{-s\widehat{\beta}}\rho.
	$$
	Using the trace theorem, we have
	$$
	s^{-1}\lambda^{-1}\int_0^T\int_{\partial\mathcal{F}}e^{-2s\widehat{\beta}}(\xi^*)^{-1}\left|\partial_t\varphi \right|^2\ d\Gamma dt\leq C\left\|\zeta_2 \partial_t\varphi \right\|^2_{L^2(0,T,[H^{3/4}(\mathcal{F})]^2)}.
	$$
	Let estimate the terms in the right hand side of \eqref{int}. We have
	\begin{multline}
	\label{int2}
	\left\| \zeta_2'\varphi\right\|^2_{L^2(0,T;[H^{3/4}(\mathcal{F})]^2)}\leq \left\| \zeta_2'\varphi\right\|^2_{L^2(0,T;[H^{1}(\mathcal{F})]^2)}\leq C s\lambda^{-1}\bigg( \int_0^T\int_{\mathcal{F}} e^{-2s\widehat{\beta}} (\xi^*)^2 \left|\varphi\right|^2  \ dydt+\\ \int_0^T\int_{\mathcal{F}} e^{-2s\widehat{\beta}} (\xi^*)^{2} \left|\nabla\varphi\right|^2  \ dydt\bigg).
	\end{multline}
	The terms appearing in the right hand side of \eqref{int2} can be absorbed by the left hand side of the Carleman inequality \eqref{eq14}.
	In the other hand, we have 
	$$
	\left\|\zeta'_2\varphi\right\|^2_{L^2(0,T;[L^2(\mathcal{F})]^2)}\leq Cs\lambda^{-1}\int_0^T\int_{\mathcal{F}} e^{-2s\widehat{\beta}} (\xi^*)^{1+2/N} \left|\varphi\right|^2  \ dydt,
	$$
	and 
\begin{multline*}
\left\|\zeta'_2\varphi\right\|^2_{H^1(0,T;[L^2(\mathcal{F})]^2)}\leq C\bigg( s\lambda^{-1}\int_0^T\int_{\mathcal{F}} e^{-2s\widehat{\beta}} (\xi^*)^{1+2/N} \left|\partial_t\varphi\right|^2  \ dydt\\+s^3\lambda^{-1}\int_0^T\int_{\mathcal{F}} e^{-2s\widehat{\beta}} (\xi^*)^{3+4/N} \left|\varphi\right|^2  \ dydt\bigg) .
\end{multline*}
By an interpolation argument, we get
\begin{multline}
\label{16:56}
\left\|\zeta'_2\varphi\right\|^2_{H^{3/8}(0,T;[L^2(\mathcal{F})]^2)}\leq C\lambda^{-1}\bigg(\int_0^T\int_{\mathcal{F}} e^{-2s\widehat{\beta}} (s(\xi^*)^{1+2/N} \left|\partial_t\varphi\right|^2)^{3/8} (s(\xi^*)^{1+2/N} \left|\varphi\right|^2)^{5/8}  \ dydt\\+ \int_0^T\int_{\mathcal{F}} e^{-2s\widehat{\beta}} (s(\xi^*)^{1+2/N} \left|\varphi\right|^2)^{5/8} (s^3(\xi^*)^{3+4/N} \left|\varphi\right|^2)^{3/8}  \ dydt\bigg).
\end{multline}
	We rewrite the right hand side of the inequality \eqref{16:56} to obtain
	\begin{multline}
	\label{20:52}
	\left\|\zeta'_2\varphi\right\|^2_{H^{3/8}(0,T;[L^2(\mathcal{F})]^2)}\leq C\lambda^{-1}\bigg(\int_0^T\int_{\mathcal{F}} e^{-2s\widehat{\beta}} (s^{-1}(\xi^*)^{-1} \left|\partial_t\varphi\right|^2)^{3/8} (s^{11/5}(\xi^*)^{11/5+16/5N} \left|\varphi\right|^2)^{5/8}  \ dydt\\+ \int_0^T\int_{\mathcal{F}} e^{-2s\widehat{\beta}} (s(\xi^*)^{1+2/N} \left|\varphi\right|^2)^{5/8} (s^3(\xi^*)^{3+4/N} \left|\varphi\right|^2)^{3/8}  \ dydt\bigg).
	\end{multline}		
	Applying again Young's inequality, we get for $N\geq 4$
	\begin{multline}
	\label{1}
	\left\| \zeta_2'\varphi\right\|^2_{H^{3/8}(0,T;[L^2(\mathcal{F})]^2)}\leq C\bigg(  s^4\lambda^{-1}\int_0^T\int_{\mathcal{F}} e^{-2s\widehat{\beta}} (\xi^*)^4 \left|\varphi\right|^2  \ dydt\\+\varepsilon s^{-1}\lambda^{-1}\int_0^T\int_{\mathcal{F}} e^{-2s\widehat{\beta}} (\xi^*)^{-1} \left|\partial_t\varphi\right|^2  \ dydt\bigg).
	\end{multline}
	The first term in the right hand side of \eqref{1} can be absorbed by the left hand side of the Carleman inequality \eqref{eq14} while the second term is absorbed by the left hand side of \eqref{int}.
	
	Using again an interpolation argument, we obtain similarly,
	\begin{multline}
	\label{17:34}
	\left\|\zeta'_2\ell_\varphi\right\|^2_{H^{3/8}(0,T)}\leq  C\lambda^{-1}\bigg(\int_0^T e^{-2s\widehat{\beta}} (s(\xi^*)^{1+2/N} \left|\ell'_\varphi\right|^2)^{3/8} (s(\xi^*)^{1+2/N} \left|\ell_\varphi\right|^2)^{5/8}  \ dt\\+ \int_0^T e^{-2s\widehat{\beta}} (s(\xi^*)^{1+2/N} \left|\ell_\varphi\right|^2)^{5/8} (s^3(\xi^*)^{3+4/N} \left|\ell_\varphi\right|^2)^{3/8}  \ dt\bigg).
	\end{multline}
	The left hand side of \eqref{17:34} can be rewritten as
	\begin{multline}
	\label{18:05}
	\left\|\zeta'_2\ell_\varphi\right\|^2_{H^{3/8}(0,T)}\leq C\lambda^{-1}\bigg(\int_0^T e^{-2s\widehat{\beta}} (s^{-1}(\xi^*)^{-1} \left|\ell'_\varphi\right|^2)^{3/8} (s^{11/5}(\xi^*)^{11/5+16/5N} \left|\ell_\varphi\right|^2)^{5/8}  \ dt\\+ \int_0^T e^{-2s\widehat{\beta}} (s(\xi^*)^{1+22/5N} \left|\ell_\varphi\right|^2)^{5/8} (s^3(\xi^*)^{3} \left|\ell_\varphi\right|^2)^{3/8}  \ dt\bigg).
	\end{multline}
	Then, for $N\geq 4$ we get
	\begin{equation}
	\label{2}
	\left\|\zeta_2'\ell_\varphi \right\|^2_{[H^{3/8}(0,T)]^2}\leq C\bigg(  s^3\lambda^{-1}\int_0^T e^{-2s\widehat{\beta}} (\xi^*)^3 \left|\ell_\varphi\right|^2 \ dt\\+\varepsilon s^{-1}\lambda^{-1}\int_0^T e^{-2s\widehat{\beta}} (\xi^*)^{-1} \left|\ell'_\varphi\right|^2 \ dt\bigg).
	\end{equation}
	The first term in the right hand side of \eqref{2} can be absorbed by the left hand side of the Carleman inequality \eqref{eq14} while the second term is absorbed by the left hand side of \eqref{int}.
	
	In the other hand,
	\begin{multline}
	\label{3}
	\left\|\zeta_2'k_\varphi \right\|^2_{H^{3/8}(0,T)}\leq C\lambda^{-1}\bigg(\int_0^T e^{-2s\widehat{\beta}} (s^{-1}(\xi^*)^{-1} \left|k'_\varphi\right|^2)^{3/8} (s^{11/5}(\xi^*)^{11/5+16/5N} \left|k_\varphi\right|^2)^{5/8}  \ dt\\+ \int_0^T e^{-2s\widehat{\beta}} (s(\xi^*)^{1+2/N} \left|k_\varphi\right|^2)^{5/8} (s^3(\xi^*)^{3+4/N} \left|k_\varphi\right|^2)^{3/8}  \ dt\bigg),
	\end{multline}
	that can be rewritten as
	\begin{multline}
	\label{4}
	\left\|\zeta_2'k_\varphi \right\|^2_{H^{3/8}(0,T)}\leq C\lambda^{-1}\bigg(\int_0^T e^{-2s\widehat{\beta}} (s^{-1}(\xi^*)^{-1} \left|k'_\varphi\right|^2)^{3/8} (s^{11/5}(\xi^*)^{11/5+16/5N} \left|k_\varphi\right|^2)^{5/8}  \ dt\\+ \int_0^T e^{-2s\widehat{\beta}} (s(\xi^*)^{1+22/5N} \left|k_\varphi\right|^2)^{5/8} (s^3(\xi^*)^{3} \left|k_\varphi\right|^2)^{3/8}  \ dt\bigg).
	\end{multline}
	Then, for $N\geq 4$, we get
	\begin{equation}
	\label{5}
	\left\|\zeta_2'k_\varphi \right\|^2_{H^{3/8}(0,T)}\leq C\bigg(s^3\lambda^{-1}\int_0^T e^{-2s\widehat{\beta}} (\xi^*)^{3} \left|k_\varphi\right|^2 \ dt\\+    \varepsilon s^{-1}\lambda^{-1}\int_0^T e^{-2s\widehat{\beta}} (\xi^*)^{-1} \left|k'_\varphi\right|^2 \ dt\bigg).
	\end{equation}
	The first term in the right hand side of \eqref{5} can be absorbed by the left hand side of the Carleman inequality \eqref{eq14} while the second term is absorbed by the left hand side of \eqref{int}.
	On the other hand, we notice that $\left|\zeta_2\rho'\right|\leq Cs^{1/2}\lambda^{-1/2}e^{-s\widehat{\beta}}(\xi^*)^{1/2+1/N}\rho $ and we get as for \eqref{0}
	\begin{multline*}
	\int_0^T\left| \zeta_2\rho'\right|^2 \left\|  v\right\|^2_{[H^{3/4}(\mathcal{F})]^2)}\ dt\leq Cs\lambda^{-1}\int_0^T e^{-2s\widehat{\beta}}(\xi^*)^{1+2/N}\left\|\rho v\right\|^2_{[H^{3/4}(\mathcal{F})]^2} \ dt
	\\\leq C s\lambda^{-1}\bigg( \int_0^T\int_{\mathcal{F}} e^{-2s\widehat{\beta}} (\xi^*)^2 \left|\rho v\right|^2  \ dydt+ \int_0^T\int_{\mathcal{F}} e^{-2s\widehat{\beta}} (\xi^*)^{2} \left|\nabla(\rho v)\right|^2  \ dydt\bigg).
	\end{multline*}
	Using the decomposition \eqref{decomp} and the regularity estimate \eqref{reg}, we deduce from the above inequality 
	\begin{multline}
	\label{eq11}
	\left\| \zeta_2\rho' v\right\|^2_{L^2(0,T;[H^{3/4}(\mathcal{F})]^2)}\leq C\bigg( s\lambda^{-1} \int_0^T\int_{\mathcal{F}} e^{-2s\widehat{\beta}} (\xi^*)^2 \left|\varphi\right|^2  \ dydt+s\lambda^{-1} \int_0^T\int_{\mathcal{F}} e^{-2s\beta} (\xi^*)^2 \left|\nabla\varphi\right|^2  \ dydt\\+ \int_0^T\int_{\mathcal{F}}e^{-3s\widehat{\beta}}\left| F_1 \right|^2  \ dydt+ \int_0^Te^{-3s\widehat{\beta}}\left| F_2 \right|^2 \ dt+ \int_0^Te^{-3s\widehat{\beta}}\left| F_3 \right|^2 \ dt\bigg).
	\end{multline}
	The first and the second term in the right hand side of \eqref{eq11} can be absorbed by the left hand side of the Carleman inequality \eqref{eq14}.
	
	Since $\rho'=-\frac{3}{2}s(\widehat{\beta})'\rho$, we have
	\begin{multline*}
	\left\| \zeta_2\rho' v\right\|^2_{H^1(0,T;[L^2(\mathcal{F})]^2)}\leq C s^2\bigg( \left\|\zeta_2(\widehat{\beta})'\rho v \right\|_{L^2(0,T;[L^2(\mathcal{F})]^2)}^2+\left\|(\zeta_2(\widehat{\beta})')'\rho v \right\|_{L^2(0,T;[L^2(\mathcal{F})]^2)}^2\\+ \left\|\zeta_2(\widehat{\beta})'\partial_t(\rho v)\right\|^2_{L^2(0,T;[L^2(\mathcal{F})]^2)} \bigg)\\\leq C \left( s\lambda^{-1} \int_0^T\int_{\mathcal{F}} e^{-2s\widehat{\beta}} (\xi^*)^{1+2/N} \left|\partial_t(\rho v)\right|^2  \ dydt+s^3\lambda^{-1} \int_0^T\int_{\mathcal{F}} e^{-2s\widehat{\beta}} (\xi^*)^{3+4/N} \left|\rho v\right|^2  \ dydt \right),
	\end{multline*}
	where we have used that
	$$
	\left|\zeta_2(\widehat{\beta})'\right|^2 \leq Cs^{-1}\lambda^{-1} e^{-2s\widehat{\beta}}(\xi^*)^{1+2/N},
	$$
	and
	$$
	\left|(\zeta_2(\widehat{\beta})')'\right|^2 \leq Cs\lambda^{-1} e^{-2s\widehat{\beta}}(\xi^*)^{3+4/N}.
	$$
	Using interpolation arguments and the Young inequality, we find as for $\left\|\zeta'_2\varphi\right\|^2_{H^{3/8}(0,T;[L^2(\mathcal{F})]^2)}$
	\begin{multline}
	\label{eq12}
	\left\| \zeta_2\rho' v\right\|^2_{H^{3/8}(0,T;[L^2(\mathcal{F})]^2)}\leq C\bigg(  s^4\lambda^{-1}\int_0^T\int_{\mathcal{F}} e^{-2s\widehat{\beta}} (\xi^*)^4 \left|\rho v\right|^2  \ dydt\\+\varepsilon s^{-1}\lambda^{-1}\int_0^T\int_{\mathcal{F}} e^{-2s\widehat{\beta}} (\xi^*)^{-1} \left|\partial_t(\rho v)\right|^2  \ dydt\bigg).
	\end{multline}
We get also
	\begin{equation}
	\label{21}
	\left\|\zeta_2\rho'\ell_v \right\|^2_{[H^{3/8}(0,T)]^2}\leq C\bigg(  s^3\lambda^{-1}\int_0^T e^{-2s\widehat{\beta}} (\xi^*)^3 \left|\rho\ell_v\right|^2 \ dt\\+\varepsilon s^{-1}\lambda^{-1}\int_0^T e^{-2s\widehat{\beta}} (\xi^*)^{-1} \left|(\rho\ell_v)'\right|^2 \ dt\bigg).
	\end{equation}
	and
	\begin{equation}
	\label{31}
	\left\|\zeta_2\rho'k_v \right\|^2_{[H^{3/8}(0,T)]^2}\leq C\bigg(  \varepsilon s^{-1}\lambda^{-1}\int_0^T e^{-2s\widehat{\beta}} (\xi^*)^{-1} \left|(\rho k_v)'\right|^2 \ dt\\+ s^3\lambda^{-1}\int_0^T e^{-2s\widehat{\beta}} (\xi^*)^{3} \left|\rho k_v\right|^2 \ dt\bigg).
	\end{equation}
	From \eqref{decomp}, \eqref{reg}, \eqref{int}, \eqref{1}, \eqref{2}, \eqref{3}, \eqref{eq11}, \eqref{eq12}, \eqref{21}, \eqref{31} and \eqref{ahhh}, we deduce that for $N\geq 4$, $\lambda\geq C$ and $s\geq C(T^N+T^{2N})$, we obtain 
	\begin{multline}
	\label{eq15}
	s^{-1}\lambda^{-1}\int_0^T\int_{\partial\mathcal{F}}e^{-2s\widehat{\beta}}(\xi^*)^{-1}\left|\partial_t\varphi \right|^2\ d\Gamma dt\leq C\bigg( s^4\lambda^{-1} \int_0^T\int_{\mathcal{F}} e^{-2s\widehat{\beta}} (\xi^*)^4 \left|\varphi\right|^2  \ dydt  \\+s^3\lambda^{-1}\int_0^Te^{-2s\widehat{\beta}}(\xi^*)^3\left|\ell_\varphi(t)\right|^2 dt+s^3\lambda^{-1}\int_0^Te^{-2s\widehat{\beta}}(\xi^*)^{3}\left| k_\varphi(t) \right|^2  \ dt+\int_0^T\int_{\mathcal{F}}e^{-3s\widehat{\beta}}\left| F_1 \right|^2  \ dydt\\+ \int_0^Te^{-3s\widehat{\beta}}\left| F_2 \right|^2 \ dt+ \int_0^Te^{-3s\widehat{\beta}}\left| F_3 \right|^2 \ dt\bigg).
	\end{multline}

	Combining \eqref{eq14}, \eqref{eq15} and \eqref{decomp}, we get finally \eqref{carl} for $N\geq 4$, $\lambda\geq C$ and $s\geq C(T^N+T^{2N})$. 
\end{proof}
\section{Null controllability for the linearized system}
\label{contrl}
In this section, we prove the null controllability of the linear system 
\begin{equation}
\label{sfl1}
\left\{
\begin{array}{rl}
\partial_t \overline{u}-\nabla\cdot \mathbb{T}(\overline{u},\pi)= v^*1_{\mathcal{O}}+F_1  & \text { in }(0,T)\times\mathcal{F},\\
\nabla \cdot \overline{u}=0& \text { in }(0,T)\times\mathcal{F},\\
\end{array}
\right.
\end{equation}
\begin{equation}
\label{sfl2}
\left\{
\begin{array}{rl}
m\overline{h}''(t)=-\int_{\partial \mathcal{S}}\mathbb{T}(\overline{u},\pi)n \ d\Gamma +F_2 & t\in(0,T),\\
J\overline{\theta}''(t)=-\int_{\partial \mathcal{S}}y^\perp\cdot\mathbb{T}(\overline{u},\pi)n\ d\Gamma+F_3& t\in(0,T),\\
\end{array}
\right.
\end{equation}
\begin{equation}
\label{sfl3}
\left\{
\begin{array}{rl}
\overline{u}_{n}= 0  &  \text { on }(0,T)\times\partial\Omega,\\
\left[ 2\nu D(\overline{u})n+\beta_{\Omega} \overline{u}\right]_{\tau}=0&  \text { on }(0,T)\times\partial\Omega,\\
(\overline{u}-\overline{u}_\mathcal{S})_{n}=0&   \text { on }(0,T)\times\partial \mathcal{S},\\
\left[ 2\nu D(\overline{u}) n + \beta_{\mathcal{S}} \left(\overline{u}-\overline{u}_\mathcal{S} \right)\right]_{\tau}=0 &  \text { on }(0,T)\times\partial \mathcal{S},
\end{array}
\right.
\end{equation}
with the initial conditions
\begin{equation}
\label{initialc}
\overline{u}(0)=u_0,\quad \overline{h}'(0)=\ell^0,\quad \overline{\theta}'(0)=\omega^0,\quad \overline{h}(0)=h^0, \quad\overline{\theta}(0)=\theta^0.
\end{equation}
The system \eqref{sfl1},\eqref{sfl2}, \eqref{sfl3}, \eqref{initialc} can be written as
\begin{equation}
\label{ls}
\left\{
\begin{array}{c}
Z'(t)=AZ'(t)+Bv^*+F,\\
a'(t)=\mathcal{C}Z(t),\\
Z(0)=Z^0,\\
a(0)=a^0,
\end{array}
\right.
\end{equation}
where $A$ is defined as in section \ref{notation} and 
$$B=\mathbb{P}(v^*1_{\mathcal{O}}), \quad F=\mathbb{P}\left( F_11_{\mathcal{F}}+\left( \frac{F_2}{m}+\frac{F_3y^\perp}{J}\right) 1_{\mathcal{S}}\right).$$
The vector $a$ is defined by $a=(\overline{h},\overline{\theta})$ and we define the operator $\mathcal{C}$ for $Z\in\mathbb{H}$ as
$$
\mathcal{C}Z=(\ell_{\overline{u}},k_{\overline{u}}), \quad\text{ if } Z=\ell_{\overline{u}}+k_{\overline{u}} y^\perp \text{ in } \mathcal{S},
$$
where $\overline{h}'=\ell_{\overline{u}}$, $\overline{\theta}'=k_{\overline{u}}$. The equation $Z(0)=Z^0$ corresponds to the initial conditions $\overline{u}(0)=u_0$, $\overline{h}'(0)=\ell^0$, $ \overline{\theta}'(0)=\omega^0$ and  the vector $a^0$ corresponds to $a^0=(h^0,\theta^0)$.

Now, let us fix $\lambda\geq C$, $s\geq C(T^N+T^{2N})$ and let consider 
\begin{equation}
\rho_1(t)=\left\{
\begin{array}{cc}
s^{3/2}\lambda^{3/2}e^{-\frac{5}{2}s\widehat{\beta}(T/2)}(\xi^*(T/2))^{3/2} & t\in (0,T/2),\\
s^{3/2}\lambda^{3/2}e^{-\frac{5}{2}s\widehat{\beta}(t)}(\xi^*(t))^{3/2} & t\in (T/2,T),
\end{array}
\right.
\end{equation}
\begin{equation}
\rho_2(t)=\left\{
\begin{array}{cc}
e^{-\frac{3}{2}s\widehat{\beta}(T/2)} & t\in (0,T/2),\\
e^{-\frac{3}{2}s\widehat{\beta}(t)}& t\in (T/2,T),
\end{array}
\right.
\end{equation}
\begin{equation}
\rho_3(t)=\left\{
\begin{array}{cc}
s^{5/2}\lambda^3e^{-s\beta^*(T/2)-\frac{3}{2}s\widehat{\beta}(T/2)}(\widehat{\xi}(T/2))^{5/2} & t\in (0,T/2),\\
s^{5/2}\lambda^3e^{-s\beta^*(t)-\frac{3}{2}s\widehat{\beta}(t)}(\widehat{\xi}(t))^{5/2}& t\in (T/2,T),
\end{array}
\right.
\end{equation}
and
\begin{equation}
\rho_4(t)=\left\{
\begin{array}{cc}
e^{-\frac{11}{8}s\widehat{\beta}(T/2)}& t\in (0,T/2),\\
e^{-\frac{11}{8}s\widehat{\beta}(t)}& t\in (T/2,T).
\end{array}
\right.
\end{equation}
We notice that $\rho_i$ are continuous positive functions such that $\rho_i(T)=0$.

Let define the following spaces 
$$
\mathcal{H}=\left\lbrace f\in L^2(0,T;\mathbb{H})\;|\; \frac{f}{\rho_1}\in L^2(0,T;\mathbb{H})\right\rbrace,
$$
$$
\mathcal{Z}=\left\lbrace z\in L^2(0,T;\mathbb{H})\;|\; \frac{z}{\rho_2}\in L^2(0,T;\mathbb{H})\right\rbrace,
$$
$$
\mathcal{U}=\left\lbrace v^*\in L^2(0,T;[L^2(\mathcal{O})]^2)\;|\; \frac{v^*}{\rho_3}\in L^2(0,T;[L^2(\mathcal{O})]^2)\right\rbrace. 
$$
Now, we can state the null controllability of the linearized system \eqref{ls}
\begin{Proposition}
\label{nulc}
Let $\beta_{\mathcal{S}}>0$. There exists a linear bounded operator $E_T:\mathbb{H}\times\mathbb{R}^3\times \mathcal{F}\longrightarrow \mathcal{U}$ such that for any $(Z^0,a^0,F)\in \mathbb{H}\times\mathbb{R}^3\times \mathcal{F}$, the control $v^*=E_T(Z^0,a^0,F)$ is such that the solution $(Z,a)$ of \eqref{ls} satisfies $Z\in \mathcal{Z}$ and $a(T)=0$.
Moreover, if $Z^0\in \mathcal{D}((-A)^{1/2})$, then 
$$
\frac{Z}{\rho_4}\in L^2(0,T;\mathcal{D}(A))\cap C([0,T];\mathcal{D}((-A)^{1/2}))\cap H^1(0,T;\mathbb{H})
$$
and we have the estimate
\begin{equation}
\label{estim}
\left\|\frac{Z}{\rho_4}\right\|_{L^2(0,T;\mathcal{D}(A))\cap C([0,T];\mathcal{D}((-A)^{1/2}))\cap H^1(0,T;\mathbb{H})} \leq C \left( \left\|F \right\|_\mathcal{H}+\left|a^0 \right|+\left\|Z^0\right\|_{\mathcal{D}((-A)^{1/2})}    \right) .
\end{equation}
\end{Proposition}
\begin{proof}
The second part of Proposition \ref{nulc} comes from the fact that 
$$
\frac{\rho_i}{\rho_4}\in L^\infty(0,T) \text{ for } i=1,3,\quad \left| \frac{\rho_4'\rho_2}{(\rho_4)^2}\right| \leq \widehat{C} e^{-\frac{1}{8}s\widehat{\beta}(t)}\leq C ,\quad t\in [0,T].
$$
Then, using \cite[Corollary 4.3]{MR2317341}, we get
$$
\frac{Z}{\rho_4}\in L^2(0,T;\mathcal{D}(A))\cap C([0,T];\mathcal{D}((-A)^{1/2}))\cap H^1(0,T;\mathbb{H}),
$$
such that \eqref{estim} is satisfied.

Let us prove the first part. The proof is similar to \cite[Theorem 4.4]{MR2317341}. The adjoint system associated to the linear system \eqref{sfl1}, \eqref{sfl2} and \eqref{sfl3} can be written as
\begin{equation}
\label{hey}
\left\{
\begin{array}{rl}
-\partial_t v-\nabla \cdot \mathbb{T}(v,q)=\gamma^1  & \text { in }(0,T)\times\mathcal{F},\\
\nabla \cdot v=0& \text { in }(0,T)\times\mathcal{F},\\
m\ell_v'(t)=\int_{\partial \mathcal{S}}\mathbb{T}( v, q)n\ d\Gamma+\ell_{\gamma^1}+\ell  & t\in(0,T),\\
J(k_v'(t))=\int_{\partial \mathcal{S}}y^\perp\cdot\mathbb{T}( v, q)n \ d\Gamma+k_{\gamma^1}+k& t\in(0,T),\\
\end{array}
\right.
\end{equation}
\begin{equation}
\label{hey1}
\left\{
\begin{array}{rl}
v_n = 0  &  \text { on }(0,T)\times\partial\Omega,\\
\left[ 2\nu D( v) n+\beta_{\Omega}  v\right]_{\tau}=0&  \text { on }(0,T)\times\partial\Omega,\\
( v- v_\mathcal{S})_n=0&   \text { on }(0,T)\times\partial \mathcal{S},\\
\left[ 2\nu D( v) n + \beta_{\mathcal{S}} ( v- v_\mathcal{S})\right]_{\tau}=0 &  \text { on }(0,T)\times\partial \mathcal{S},
\end{array}
\right.
\end{equation}
where $ v_\mathcal{S}(y)=\ell_v+k_vy^\perp$, with
\begin{equation*}
v(T,\cdot) = 0,\; \text{ in } \mathcal{F},\quad
\ell_v(T)=0,\quad k_v(T)=0.
\end{equation*} 
Here $\gamma^2=(\ell,k)\in \mathbb{R}^3$ and  $\gamma^1\in L^2(0,T;\mathbb{H})$.
Following the arguments of \cite[Theorem 4.1]{MR2317341}, we show that the null controllability of the system \eqref{sfl1}, \eqref{sfl2}, \eqref{initialc} is equivalent to show the following observability inequality
\begin{equation}
\left|\gamma^2\right|^2+\left\|v(0)\right\|^2_\mathbb{H}+\int_0^T\left\| \rho_1 v\right\|^2_\mathbb{H} \ dt\leq C \left(\int_0^T\left\|\rho_2\gamma^1\right\|^2_\mathbb{H}\ dt+ \int_0^T\int_\mathcal{O}\left|\rho_3 v\right|^2\ dy dt   \right). 
\end{equation}
We set
\begin{equation}
\rho^*_i(t)=\left\{
\begin{array}{cc}
\rho_i(T-t) & t\in (0,T/2),\\
\rho_i(t)& t\in (T/2,T).
\end{array}
\right.
\end{equation}
Then, \eqref{carl} implies
\begin{equation}
\label{car1}
\int_0^T\left\| \rho_1^* v\right\|^2_{\mathbb{H}} \ dt\leq C \left(\int_0^T\left\|\rho^*_2(\gamma^1+\mathcal{C}^*\gamma^2)\right\|^2_{\mathbb{H}}\ dt+ \int_0^T\int_\mathcal{O}\left|\rho_3^* v\right|^2\ dy dt\right) .
\end{equation}
Next, we argue as \cite[Proposition 4]{MO}. Let consider a non negative function $\overline{\eta}\in C^1([0,T])$ such that 
$$
0\leq\overline{\eta}\leq1\text{ in } [0,T],\quad\overline{\eta}=1\text{ in } \left[ 0,T/2\right] ,\quad \overline{\eta}=0\text{ in } \left[ 3T/4,T\right].
$$
Then $(\overline{\eta}v,\overline{\eta}\ell_v,\overline{\eta}k_v)$ satisfies the energy estimates
\begin{multline*}
\left\| \overline{\eta}v\right\|^2_{L^2(0,T;[H^2(\mathcal{F})]^2)\cap C([0,T];[H^1(\mathcal{F})]^1)\cap H^1(0,T;[L^2(\mathcal{F})]^2)}+ \left\|\overline{\eta}\ell_v\right\|^2_{[H^1(0,T)]^2}+\left\|\overline{\eta}k_v\right\|^2_{H^1(0,T)} \\\leq C\left( \left\|\overline{\eta}'v\right\|^2_{L^2(0,T;[L^2(\mathcal{F})]^2)}+\left\|\overline{\eta}'l_v\right\|^2_{[L^2(0,T)]^2}+\left\|\overline{\eta}'k_v\right\|^2_{L^2(0,T)}  +\left\|\overline{\eta}(\gamma^1+\mathcal{C}^*\gamma^2)\right\|^2_{L^2(0,T;\mathbb{H})}\right). 
\end{multline*}
Whence
\begin{multline}
\label{eq17}
\left\| v\right\|^2_{L^2(0,T/2;[H^2(\mathcal{F})]^2)\cap C([0,T/2];[H^1(\mathcal{F})]^2)\cap H^1(0,T/2;[L^2(\mathcal{F})]^2)}+ \left\|\ell_v\right\|^2_{[H^1(0,T/2)]^2}+\left\|k_v\right\|^2_{H^1(0,T/2)} \\\leq C\left( \left\|v\right\|^2_{L^2(T/2,3T/4;[L^2(\mathcal{F})]^2)}+\left\|l_v\right\|^2_{[L^2(T/2,3T/4)]^2}+\left\|k_v\right\|^2_{L^2(T/2,3T/4)}  +\left\|(\gamma^1+\mathcal{C}^*\gamma^2)\right\|^2_{L^2(0,3T/4;\mathbb{H})}\right). 
\end{multline}
Using \eqref{car1}, we have
\begin{multline*}
 \left\|v\right\|^2_{L^2(T/2,3T/4;[L^2(\mathcal{F})]^2)}+\left\|\ell_v\right\|^2_{[L^2(T/2,3T/4)]^2}+\left\|k_v\right\|^2_{L^2(T/2,3T/4)}\leq C\int_0^T\left\| \rho_1^* v\right\|^2_{\mathbb{H}} \ dt\\\leq C \left(\int_0^T\left\|\rho^*_2(\gamma^1+\mathcal{C}^*\gamma^2)\right\|^2_{\mathbb{H}}\ dt+ \int_0^T\int_\mathcal{O}\left|\rho_3^* v\right|^2\ dy dt\right),
\end{multline*}
where we have used that the function $\frac{1}{\rho_1^*}$ is bounded in $[T/2,3T/4]$.

Since $\rho_1$ is constant in $(0,T/2)$ and $\frac{1}{\rho_2}$ is bounded in $(0,3T/4)$, the equation \eqref{eq17} gives
\begin{multline*}
\left\|v(0)\right\|^2_{\mathbb{H}}+\int_0^{T/2}\left\| \rho_1 v\right\|^2_{\mathbb{H}} \ dt\leq C \bigg(\int_0^{3T/4}\left\|\rho_2(\gamma^1+\mathcal{C}^*\gamma^2)\right\|^2_{\mathbb{H}}\ dt\\+\int_0^T\left\|\rho^*_2(\gamma^1+\mathcal{C}^*\gamma^2)\right\|^2_{\mathbb{H}}\ dt+  \int_0^T\int_\mathcal{O}\left|\rho_3^* v\right|^2\ dy dt\bigg) .
\end{multline*}
 Thus, using that $\rho^*_1=\rho_1$ in $(T/2,T)$, the estimate \eqref{car1} and using the fact that $\rho_i^*\leq \rho_i$ for $i=1,2,3$ in $[0,T]$, we get
\begin{equation*}
\left\|v(0)\right\|^2_{\mathbb{H}}+ \int_0^T\left\| \rho_1 v\right\|^2_{\mathbb{H}} \ dt\leq C \left(\int_0^T\left\|\rho_2(\gamma^1+\mathcal{C}^*\gamma^2)\right\|^2_{\mathbb{H}}\ dt+ \int_0^T\int_\mathcal{O}\left|\rho_3 v\right|^2\ dy dt\right) .
\end{equation*}
It remains to prove
\begin{equation}
\label{ing}
\left|\gamma^2\right|^2\leq C \left(\int_0^T\left\|\rho_2\gamma^1\right\|^2_\mathbb{H}\ dt+ \int_0^T\int_\mathcal{O}\left|\rho_3 v\right|^2\ dy dt   \right).    
\end{equation}
The result follows by using a contradiction argument. In fact, assume that \eqref{ing} is false. Then there exists a sequence $(\gamma^1_\kappa,\gamma^2_\kappa)_\kappa\in L^2(0,T;\mathbb{H})\times\mathbb{R}^3$ such that
\begin{equation}
\label{abs}
\left|\gamma^2_\kappa\right|=1 ,
\end{equation}
 and 
\begin{equation}
\label{limit}
\int_0^T\left\|\rho_2\gamma_\kappa^1\right\|^2_\mathbb{H} \ dydt+\int_0^T\int_\mathcal{O} \left|\rho_3v_\kappa \right|^2 \ dy dt  \longrightarrow 0.
\end{equation}
Thus, for $\varepsilon >0$
$$
\gamma_\kappa^1\rightharpoonup 0\text{ in } L^2(0,T-\varepsilon;\mathbb{H}),
$$
and
$$
\gamma^2_\kappa\longrightarrow \gamma^2 \text{ in } \mathbb{R}^3.
$$
On the other hand, the solution $(v_\kappa)_\kappa$ of the system \eqref{hey}, \eqref{hey1} associated with $(\gamma^1_\kappa,\gamma^2_\kappa)$ verifies the estimate (arguing always like \cite[Proposition 4]{MO} taking $\overline{\eta}=1$ in $[0,T-\varepsilon]$ and $\overline{\eta}=0$ in $[T-\frac{\varepsilon}{2},T]$)
\begin{multline*}
\left\| v_\kappa\right\|^2_{L^2(0,T-\varepsilon;[H^2(\mathcal{F})]^2)\cap H^1(0,T-\varepsilon;[L^2(\mathcal{F})]^2)}+ \left\| q_\kappa\right\|^2_{L^2(0,T-\varepsilon;H^1(\mathcal{F}))}+\left\|(\ell_v)_\kappa\right\|^2_{[H^1(0,T-\varepsilon)]^2}+\left\|(k_v)_\kappa\right\|^2_{ H^1(0,T-\varepsilon)} \\\leq C\left(\int_0^T\left\|\rho_2(\gamma_\kappa^1+\mathcal{C}^*\gamma_\kappa^2)\right\|^2_{\mathbb{H}}\ dt+ \int_0^T\int_\mathcal{O}\left|\rho_3 v_\kappa\right|^2\ dy dt\right).
\end{multline*}
From the above inequality, we have 
$$
v_\kappa\rightharpoonup v \text{ in } L^2(0,T-\varepsilon;\mathcal{D}(A))\cap H^1(0,T-\varepsilon;\mathbb{H}),
$$
and
$$
q_\kappa\rightharpoonup q \text{ in } L^2(0,T-\varepsilon;H^1(\mathcal{F})).
$$
Therefore, the couple $(v,q)$ satisfies the system 
\begin{equation*}
\left\{
\begin{array}{rl}
-\partial_t v-\nabla \cdot \mathbb{T}(v,q)=0 & \text { in }(0,T-\varepsilon)\times\mathcal{F},\\
\nabla \cdot v=0& \text { in }(0,T-\varepsilon)\times\mathcal{F},\\
m\ell_v'(t)=\int_{\partial \mathcal{S}}\mathbb{T}( v, q)n\ d\Gamma+\ell  & t\in(0,T-\varepsilon),\\
J(k_v'(t))=\int_{\partial \mathcal{S}}y^\perp\cdot\mathbb{T}( v, q)n \ d\Gamma+k& t\in(0,T-\varepsilon),\\
\end{array}
\right.
\end{equation*}
\begin{equation*}
\left\{
\begin{array}{rl}
v_n = 0  &  \text { on }(0,T-\varepsilon)\times\partial\Omega,\\
\left[ 2\nu D( v) n+\beta_{\Omega}  v\right]_{\tau}=0&  \text { on }(0,T-\varepsilon)\times\partial\Omega,\\
( v- v_\mathcal{S})_n=0&   \text { on }(0,T-\varepsilon)\times\partial \mathcal{S},\\
\left[ 2\nu D( v) n + \beta_{\mathcal{S}} ( v- v_\mathcal{S})\right]_{\tau}=0 &  \text { on }(0,T-\varepsilon)\times\partial \mathcal{S}.
\end{array}
\right.
\end{equation*}
Moreover, from \eqref{limit}, we have $v=0$ in $(0,T-\varepsilon)\times\mathcal{O}$. 
Then, using the unique continuity property of the Stokes system (see for instance \cite{MR1387461}), we get 
$$
v=\nabla q=0 \text{ in }(0,T-\varepsilon)\times\mathcal{F}.
$$ 
The boundary conditions read to 
$$
(\ell_v+k_vy^\perp)_n=0,\quad \beta_{\mathcal{S}}(\ell_v+k_vy^\perp)_\tau=0,\quad y\in\partial\mathcal{S}.
$$
Since $\beta_\mathcal{S}>0$, we get that $\ell_v=0$ and $k_v=0$ in $(0,T-\varepsilon)$.
Then, we obtain in particular that $\gamma^2=(\ell,k)=0$ from the equations of the structure motion which contradicts \eqref{abs}.

\end{proof}
\section{Fixed point}
\label{fip}
In this section, we prove Theorem \ref {tp} by applying a fixed-point argument. For this purpose, we follow the same steps as \cite{MR2317341}. First, we give some estimates on the terms appearing in the system \eqref{sf1}, \eqref{sf2} and \eqref{sf3}.
We have the following lemma that is proved in \cite{MR2029294}.
\begin{Lemma}
\label{lem1}
Let $X$ and $Y$ satisfying the properties given in Section \ref{chg}. We obtain for all $(\overline{u},\pi)\in [H^2(\mathcal{F})]^2\times H^1(\mathcal{F})$, the following estimates, for all $t\in [0,T]$
\begin{equation*}
\frac{1}{\rho_4(t)^2}\left\|(\mathcal{L}-\Delta)\overline{u} \right\|_{[L^2(\mathcal{F})]^2}\leq C \left( \left\|\frac{\overline{h}'}{\rho_4}\right\|_{[L^\infty(0,T)]^2}+\left\|\frac{\overline{\theta}'}{\rho_4}\right\|_{[L^2(0,T)]^2} \right) \left\|\frac{\overline{u}}{\rho_4(t)}\right\|_{[H^2(\mathcal{F})]^2} ,
\end{equation*}
\begin{equation*}
\frac{1}{\rho_4(t)^2}\left\|\mathcal{N}\overline{u} \right\|_{[L^2(\mathcal{F})]^2}\leq C \left(1+ \left\|\frac{\overline{h}'}{\rho_4}\right\|_{[L^2(0,T)]^2}+\left\|\frac{\overline{\theta}'}{\rho_4}\right\|_{[L^2(0,T)]^2} \right) \left\|\frac{\overline{u}}{\rho_4(t)}\right\|_{[H^1(\mathcal{F})]^2}\left\|\frac{\overline{u}}{\rho_4(t)}\right\|_{[L^2(\mathcal{F})]^2} ,
\end{equation*}
\begin{equation*}
\frac{1}{\rho_4(t)^2}\left\|\mathcal{M}\overline{u} \right\|_{[L^2(\mathcal{F})]^2}\leq C\left( \left\|\frac{\overline{h}'}{\rho_4}\right\|_{[L^2(0,T)]^2}+\left\|\frac{\overline{\theta}'}{\rho_4}\right\|_{[L^2(0,T)]^2} \right) \left\|\frac{\overline{u}}{\rho_4(t)}\right\|_{[H^1(\mathcal{F})]^2} ,
\end{equation*}
\begin{equation*}
\frac{1}{\rho_4(t)^2}\left\|(\mathcal{G}-\nabla)\pi \right\|_{[L^2(\mathcal{F})]^2}\leq C \left( \left\|\frac{\overline{h}'}{\rho_4}\right\|_{[L^\infty(0,T)]^2}+\left\|\frac{\overline{\theta}'}{\rho_4}\right\|_{[L^2(0,T)]^2} \right) \left\|\frac{\pi}{\rho_4(t)}\right\|_{[H^1(\mathcal{F})]^2}.
\end{equation*}
\end{Lemma}
We have also 
\begin{Lemma}
\label{lem2}
Let $X$ and $Y$ satisfying the properties given in Section \ref{chg}. We obtain for all $(\overline{u},\pi)\in [H^2(\mathcal{F})]^2\times H^1(\mathcal{F})$ the following estimates, for all $t\in [0,T]$
\begin{equation*}
\frac{1}{\rho_4(t)^2}\left\|(\mathcal{L}^{(1)}-\mathcal{L}^{(2)})\overline{u} \right\|_{[L^2(\mathcal{F})]^2}\leq C \left( \left\|\frac{\overline{h}'^{(1)}-\overline{h}'^{(2)}}{\rho_4}\right\|_{[L^\infty(0,T)]^2}+\left\|\frac{\overline{\theta}'^{(1)}-\overline{\theta}'^{(2)}}{\rho_4}\right\|_{[L^2(0,T)]^2} \right) \left\|\frac{\overline{u}}{\rho_4}\right\|_{[H^2(\mathcal{F})]^2} ,
\end{equation*}
\begin{multline*}
\frac{1}{\rho_4(t)^2}\left\|(\mathcal{N}^{(1)}-\mathcal{N}^{(2)})\overline{u} \right\|_{[L^2(\mathcal{F})]^2}\\\leq C \left( \left\|\frac{\overline{h}'^{(1)}-\overline{h}'^{(2)}}{\rho_4}\right\|_{[L^2(0,T)]^2}+\left\|\frac{\overline{\theta}'^{(1)}-\overline{\theta}'^{(2)}}{\rho_4}\right\|_{[L^2(0,T)]^2} \right) \left\|\frac{\overline{u}}{\rho_4(t)}\right\|_{[H^1(\mathcal{F})]^2}\left\|\frac{\overline{u}}{\rho_4(t)}\right\|_{[L^2(\mathcal{F})]^2} ,
\end{multline*}
\begin{equation*}
\frac{1}{\rho_4(t)^2}\left\|(\mathcal{M}^{(1)}-\mathcal{M}^{(2)})\overline{u} \right\|_{[L^2(\mathcal{F})]^2}\leq C \left( \left\|\frac{\overline{h}'^{(1)}-\overline{h}'^{(2)}}{\rho_4}\right\|_{[L^2(0,T)]^2}+\left\|\frac{\overline{\theta}'^{(1)}-\overline{\theta}'^{(2)}}{\rho_4}\right\|_{[L^2(0,T)]^2} \right) \left\|\frac{\overline{u}}{\rho_4(t)}\right\|_{[H^1(\mathcal{F})]^2},
\end{equation*}
\begin{equation*}
\frac{1}{\rho_4(t)^2}\left\|(\mathcal{G}^{(1)}-\mathcal{G}^{(2)})\pi \right\|_{[L^2(\mathcal{F})]^2}\leq C \left( \left\|\frac{\overline{h}'^{(1)}-\overline{h}'^{(2)}}{\rho_4}\right\|_{[L^\infty(0,T)]^2}+\left\|\frac{\overline{\theta}'^{(1)}-\overline{\theta}'^{(2)}}{\rho_4}\right\|_{[L^2(0,T)]^2} \right) \left\|\frac{\pi}{\rho_4(t)}\right\|_{[H^1(\mathcal{F})]^2}.
\end{equation*}
\end{Lemma}
Now, we are in position to prove Theorem \ref{tp}.  
\begin{proof}[Proof of Theorem \ref {tp}]
For all $r>0$, let us set
$$
\mathcal{K}_r=\left\lbrace F\in \mathcal{H}\;|\; \left\|\frac{F}{\rho_1}\right\|_{L^2(0,T;\mathbb{H})}\leq r \right\rbrace. 
$$
Let $F\in \mathcal{K}_r$, and assume that
\begin{equation}
\label{assum}
\left\|u^0 \right\|_{[H^1(\mathcal{F})]^2} +\left|\widetilde{\ell^0}\right|+\left|h^0\right|+\left|\omega^0 \right|+\left|\theta^0 \right|\leq r.
\end{equation}
From Proposition \ref{nulc}, the solution $(\overline{u},\pi,\overline{h},\overline{\theta})$ of the linear system \eqref{sfl1}, \eqref{sfl2}, \eqref{initialc} with $v^*=E_T(Z^0,a^0,F)$ satisfies $\overline{h}(T)=0$, $\overline{\theta}(T)=0$ and 
$$
\frac{\overline{u}}{\rho_4}\in L^2(0,T;[H^2(\mathcal{F})]^2)\cap H^1(0,T;[L^2(\mathcal{F})]^2),\quad \frac{\nabla \pi}{\rho_4} \in L^2(0,T;L^2(\mathcal{F})),
$$ 
$$
\frac{\overline{h}'}{\rho_4}\in [H^1(0,T)]^2,\quad\frac{\overline{\theta}'}{\rho_4} \in H^1(0,T).
$$
Using \eqref{assum} and \eqref{estim}, we get
\begin{equation}
\label{assum2}
\left\|\frac{\overline{u}}{\rho_4}\right\|_{L^2(0,T;\mathcal{D}(A))\cap H^1(0,T;\mathbb{H})\cap C(0,T;(\mathcal{D}(-A))^{1/2})}+\left\|\frac{\nabla \pi}{\rho_4}\right\|_{[L^2(\mathcal{F})]^2} \leq Cr. 
\end{equation}

Using the condition \eqref{assum2}, we can construct the change of variables defined in Section \ref{chg}. 

We can thus, define the mapping $\Phi: \mathcal{K}_r\longrightarrow \mathcal{K}_r$, that associates $F\in \mathcal{K}_r$, we set
$$
\Phi(F)=
\left\{
\begin{array}{cc}
 \nu(\mathcal{L}-\Delta)\overline{u}-\mathcal{M}\overline{u}-\mathcal{N}\overline{u}+(\nabla-\mathcal{G})\pi & \text{ in }\mathcal{F},\\
 -mk_{\overline{u}}\ell_{\overline{u}}^\perp & \text{ in }\mathcal{S},
 \end{array}
 \right.
$$
where $(\overline{u},\pi,\overline{h},\overline{\theta})$ is the solution of the linear system \eqref{sfl1}, \eqref{sfl2} and \eqref{sfl3}.
Combining Lemma \ref{lem1}, the estimate \eqref{assum2} and 
$$
\frac{(\rho_4)^2}{\rho_1} \in L^\infty(0,T),
$$
we obtain
$$
\left\|\frac{\Phi(F)}{\rho_1} \right\|_{L^2(0,T;[L^2(\mathcal{F})]^2)} \leq C (1+r)^{d+1}r^2.
$$

Then, for $r$ small enough, we get $\Phi(\mathcal{K}_r)\subset \mathcal{K}_r$.
Similarly, using Lemma \ref{lem2}, we get that 
$$
\left\|\frac{\Phi(F^1)-\Phi(F^2)}{\rho_1} \right\|_{L^2(0,T;[L^2(\mathcal{F})]^2)} \leq C (1+r)^{d+1}r\left\|\frac{F^1-F^2}{\rho_1}\right\|_{L^2(0,T;[L^2(\mathcal{F})]^2)} . 
$$
Thus for $r$ small enough, we obtain that $\Phi|_{\mathcal{K}_r}$ is a contraction. Then $\Phi$ admits a fixed point associated to $(u,p,h,\theta)$, the unique solution of \eqref{sf1}, \eqref{sf2}, \eqref{sf3}, \eqref{sf4}, which ends the proof of Theorem \ref {tp}.
\end{proof}
\appendix
\section{Carleman estimate for the heat operator}
Let define the function $\eta$ like \eqref{eta}.
Let take $\beta$ and $\xi$ as \eqref{beta}, \eqref{ksi}.  Let $\psi$ be a function defined on $\mathcal{F}$ that can be either a vector valued function or a scalar function. 
\begin{Proposition}
\label{pheat}
Suppose that the function $\psi$ verifies the heat equation
\begin{equation}
\label{heat}
-\partial_t\psi-\nu\Delta\psi=f,\quad \text{ in } (0,T)\times\mathcal{F}.
\end{equation}
Then, there exist $C(\mathcal{F},\mathcal{O}_\eta)>0$, $s_1$ and $\lambda_1$ where $\lambda_1= \lambda_1(\mathcal{F},\mathcal{O}_\eta)\geq 1$ and $s_1=\lambda_1(T^N+T^{2N})$, 
such that the following estimate holds
\begin{multline}
\label{car}
s\lambda^2\int_0^T\int_{\mathcal{F}}e^{-2s\beta}\xi \left|\nabla \psi\right|^2  \ dy dt+s^3\lambda^4\int_0^T\int_{\mathcal{F}}e^{-2s\beta}\xi^3\left|\psi\right|^2 \ dy dt +s^3\lambda^3\int_0^T\int_{\partial\mathcal{F}}e^{-2s\widehat{\beta}}(\xi^*)^3\left|\psi\right|^2d\Gamma dt \\\leq C\bigg( s^3\lambda^4\int_0^T\int_{\mathcal{O}_\eta} e^{-2s\beta}\xi^3\left|\psi \right|^2 \ dy dt +s\lambda^2\int_0^T\int_{\mathcal{O}_\eta}e^{-2s\beta} \xi \left|\nabla\psi \right|^2 \ dy dt+\int_0^T\int_{\mathcal{F}} e^{-2s\beta}\left| f\right|^2 \ dy dt \\+s\lambda\int_0^T\int_{\partial\mathcal{F}}e^{-2s\widehat{\beta}}\xi^*\left| \nabla \psi\tau\right| ^2 d\Gamma+s^{-1}\lambda^{-1}\int_0^T\int_{\partial\mathcal{F}}e^{-2s\widehat{\beta}}(\xi^*)^{-1}\left|\partial_t\psi\right|^2 \ d\Gamma dt\bigg),
\end{multline}
for any $s\geq s_1$ and $\lambda\geq \lambda_1$, where $f\in [L^2(\mathcal{F})]^2$.	
\end{Proposition}
\begin{proof}
For the sake of simplicity, we take $\nu=1$. We set $w=e^{-s\beta}\psi$. The equation \eqref{heat} becomes
\begin{equation*}
-s\partial_t\beta w-\partial_t w-\Delta w-s^2\lambda^2\xi^2\left|\nabla \eta\right|^2w+2s\lambda \xi (\nabla \eta\cdot\nabla)w+s\lambda^2\xi\left|\nabla\eta\right|^2w+s\lambda \xi \Delta \eta w =e^{-s\beta}f,
\end{equation*} 
which is equivalent to	
\begin{equation}
\label{Ceq1}
Pw+Qw=e^{-s\beta}f+s\lambda^2\xi\left|\nabla\eta\right|^2w-s\lambda \xi \Delta \eta w	+s\partial_t\beta w,
\end{equation}
where 
\begin{equation}
\label{Pw}
Pw=-\Delta w-s^2\lambda^2\xi^2\left|\nabla \eta\right|^2w,
\end{equation}
and 
\begin{equation}
\label{Qw}
Qw=2s\lambda \xi \nabla w \nabla\eta+2s\lambda^2\xi\left|\nabla\eta\right|^2w-\partial_t w.
\end{equation}		
We multiply \eqref{Ceq1} by itself, we have
\begin{multline*}
\int_0^T\left\|Pw\right\|^2_{[L^2(\mathcal{F})]^2} \ dt+\int_0^T\left\|Qw\right\|^2_{[L^2(\mathcal{F})]^2} \ dt +2\int_0^T\sum_{i}^{4}\sum_{j}^2\left\langle(Pw)_i, (Qw)_j \right\rangle_{[L^2(\mathcal{F})]^2} \ dt\\  \leq C\left(\int_0^T\int_\mathcal{F} e^{-2s\beta}\left| f\right|^2 \ dy dt +s^2\lambda^4 \int_0^T\int_\mathcal{F} \left| w\right|^2 \ dy dt\right) .
\end{multline*}  
Thus, it suffices to treat $\int_0^T\left\langle (Pw)_i,(Qw_j)\right\rangle_{[L^2(\mathcal{F})]^2} \ dt $ since
$$
2\int_0^T\sum_{i}^{2}\sum_{j}^2\left\langle(Pw)_i, (Qw)_j \right\rangle_{[L^2(\mathcal{F})]^2} \ dt  \leq C\left( \int_0^T\int_\mathcal{F} e^{-2s\beta}\left| f\right|^2 \ dy dt +s^2\lambda^4 \int_0^T\int_\mathcal{F} \left| w\right|^2 \ dy dt\right).
$$
Here, we denote by $(Pw)_i$ and $(Qw_j)$, the $i$ th and the $j$ th term in the expression \eqref{Pw} and \eqref{Qw} respectively. First, we need to obtain an inequality with two global terms of $\left|w\right|^2$ and $\left| \nabla w\right|^2 $ on the left hand side namely
$$
s\lambda^2\int_0^T\int_{\mathcal{F}}\xi\left|\nabla w\right|^2  \ dy dt,\quad s^3\lambda^4\int_0^T\int_{\mathcal{F}}\xi^3\left|w\right|^2 \ dy dt,
$$
and associate local terms on the right hand side.
First, we have
\begin{multline*}
\int_0^T\left\langle (Pw)_1,(Qw)_1\right\rangle_{[L^2(\mathcal{F})]^2} \ dt=-2s\lambda\int_0^T\int_{\mathcal{F}}\xi \Delta w \cdot \nabla w\nabla \eta\ dy dt\\=-2s\lambda\sum_{i,j,k}\int_0^T\int_{\mathcal{F}}\xi\partial_j\eta \partial_j w_i \partial^2_{kk}w_i\ dydt=2s\lambda^2\sum_{i}\int_0^T\int_{\mathcal{F}}\xi\left|\nabla \eta\cdot\nabla w_i \right|^2  dydt\\+2s\lambda\sum_{i,j,k}\int_0^T\int_{\mathcal{F}}\xi\partial^2_{jk}\eta\partial_jw_i \partial_kw_i dydt+2s\lambda\sum_{i,j,k}\int_0^T\int_{\mathcal{F}}\xi\partial_{j}\eta \partial^2_{kj}w_i  \partial_kw_i \ dy dt\\-2s\lambda\sum_{i,j,k}\int_0^T\int_{\partial\mathcal{F}}\xi^*\partial_{j}\eta \partial_jw_i  \partial_kw_in_k \ d\Gamma dt.
\end{multline*}
On the other hand, we have
\begin{multline*}
2s\lambda\int_0^T\int_{\mathcal{F}}\xi\partial_{j}\eta \ \partial^2_{kj}w_i \ \partial_kw_i \ dydt=2s\lambda\int_0^T\int_{\partial\mathcal{F}}\xi^*\partial_j\eta  \partial_kw_i  \partial_kw_i n_j d\Gamma dt-2s\lambda^2\int_0^T\int_{\mathcal{F}}\xi\left|\nabla \eta \right|^2 \left| \nabla w_i\right|^2 \ dydt\\-2s\lambda\int_0^T\int_{\mathcal{F}}\xi\Delta\eta \left| \nabla w_i\right| ^2 \ dydt-2s\lambda\int_0^T\int_{\mathcal{F}}\xi\partial_{j}\eta \ \partial^2_{kj}w_i  \partial_kw_i \ dydt.
\end{multline*}	
It follows that
\begin{multline*}
2s\lambda\int_0^T\int_{\mathcal{F}}\xi\partial_{j}\eta \ \partial^2_{kj}w_i \ \partial_kw_i \ dydt=s\lambda\int_0^T\int_{\partial\mathcal{F}}\xi^*\partial_j\eta \ \partial_kw_i \ \partial_kw_i\ n_j d\Gamma dt-s\lambda^2\int_0^T\int_{\mathcal{F}}\xi\left|\nabla \eta \right|^2 \left| \nabla w_i\right|^2 \ dydt\\-s\lambda\int_0^T\int_{\mathcal{F}}\xi\Delta\eta \ \left| \nabla w_i\right| ^2 \ dydt=s\lambda\int_0^T\int_{\partial\mathcal{F}}\xi^*(\nabla\eta\cdot n)\left|\nabla w_i\right|^2  d\Gamma dt-s\lambda^2\int_0^T\int_{\mathcal{F}}\xi\left|\nabla \eta \right|^2 \left| \nabla w_i\right|^2 \ dy dt\\-s\lambda\int_0^T\int_{\mathcal{F}}\xi\Delta\eta \ \left| \nabla w_i\right| ^2 \ dydt.
\end{multline*}		
Then
\begin{multline*}
\int_0^T\left\langle (Pw)_1,(Qw)_1\right\rangle_{[L^2(\mathcal{F})]^2} \ dt\geq -s\lambda^2\int_0^T\int_{\mathcal{F}}\xi\left|\nabla \eta \right|^2 \left|\nabla w\right|^2 \ d\Gamma dt- Cs\lambda\int_0^T\int_{\mathcal{F}}\xi\left|\nabla w\right| ^2 dydt \\+s\lambda\sum_{i}\int_0^T\int_{\partial\mathcal{F}}\xi^*(\nabla\eta\cdot n)\left|\nabla w_i\right|^2  d\Gamma dt -2s\lambda\sum_{i}\int_0^T\int_{\partial\mathcal{F}}\xi^*\left|\nabla w_i\cdot n  \right|^2( \nabla\eta\cdot n ) \ d\Gamma dt . 
\end{multline*}
Using that 
$
\left|\nabla w_i\right|^2=\left|\nabla w_i\cdot n\right|^2+\left|\nabla w_i\cdot \tau\right|^2,   
$
we get
\begin{multline}
\label{p1q1}
\int_0^T	\left\langle (Pw)_1,(Qw)_1\right\rangle_{[L^2(\mathcal{F})]^2} \ dt\geq -s\lambda^2\int_0^T\int_{\mathcal{F}}\xi\left|\nabla \eta \right|^2 \left|\nabla w\right|^2 \ d\Gamma dt- Cs\lambda\int_0^T\int_{\mathcal{F}}\xi\left|\nabla w\right| ^2 dydt \\-Cs\lambda\int_0^T\int_{\partial\mathcal{F}}\xi^*\left|\nabla w\tau\right|^2  d\Gamma dt -s\lambda\int_0^T\int_{\partial\mathcal{F}}\xi^*\left|\nabla w n  \right|^2( \nabla\eta\cdot n ) \ d\Gamma dt  . 
\end{multline}
We have also
\begin{multline*}
\label{Ceq2}
\int_0^T\left\langle (Pw)_1,(Qw)_2\right\rangle_{[L^2(\mathcal{F})]^2}  dt=-2s\lambda^2 \int_0^T\int_{\mathcal{F}}\xi \left| \nabla\eta\right|^2 \Delta w \cdot w \ dydt = 2s\lambda^2\sum_{i}^3\int_0^T\int_{\mathcal{F}}\xi \left|\nabla \eta\right|^2 \left| \nabla w_i\right|^2  \ dy dt\\+4s\lambda^2\sum_{i,j}^3\int_0^T\int_{\mathcal{F}}\xi \  \partial_i\eta \ \partial^2_{ij}\eta \ w\cdot \ \partial_jw \ dydt +2s\lambda^3\sum_{i}^3\int_0^T\int_{\mathcal{F}}\xi \left|\nabla \eta \right|^2(\nabla\eta\cdot \nabla w_i)w_i \ dydt\\-2s\lambda^2\sum_{i,j}^3\int_0^T\int_{\partial\mathcal{F}}\xi^* \left| \nabla \eta\right|^2 w_i\partial_jw_in_j\ d\Gamma dt .   
\end{multline*}
Using the Cauchy inequality, we get
\begin{equation*}
2s\lambda^3\sum_{i}^3\int_0^T\int_{\mathcal{F}}\xi \left|\nabla \eta \right|^2(\nabla\eta\cdot \nabla w_i)w_i \ dy dt\geq -Cs\lambda^4 \int_0^T\int_{\mathcal{F}} \xi\left|w\right|^2 \  dydt -\varepsilon s\lambda^2\int_0^T\int_{\mathcal{F}}\xi\left| \nabla \eta \right|^2\left| \nabla w\right|^2 \ dydt,   
\end{equation*}
for $\varepsilon >0$, and 
\begin{equation*}
4s\lambda^2\sum_{i,j}^3\int_0^T\int_{\mathcal{F}}\xi \  \partial_i\eta \partial^2_{ij}\eta w_i \partial_jw_i \ dy dt\geq -Cs\lambda^4 \int_0^T\int_{\mathcal{F}} \xi\left|w\right|^2 \  dy dt -\varepsilon s\lambda^2\int_0^T\int_{\mathcal{F}}\xi\left| \nabla \eta \right|^2\left| \nabla w\right|^2 \ dy dt.   
\end{equation*}
We obtain
\begin{multline}
\label{P1Q}
\int_0^T\left\langle (Pw)_1,(Qw)_2\right\rangle_{[L^2(\mathcal{F})]^2} \ dt\geq s\lambda^2\int_0^T\int_{\mathcal{F}}\xi \left| \nabla\eta\right|^2 \left|\nabla w\right|^2 \ dy dt\\-C\bigg(  s^2\lambda^4\int_0^T\int_{\mathcal{F}} \xi \left|w \right|^2 \ dy dt+ s\lambda \int_0^T\int_{\mathcal{F}} \xi \left|\nabla w \right|^2 \ dy dt+ \varepsilon s\lambda\int_0^T\int_{\partial\mathcal{F}}\xi^*\left|  \nabla wn\right|^2   \ d\Gamma dt\\+s\lambda^3\int_0^T\int_{\partial\mathcal{F}}\xi^*\left|  w\right|^2 \ d\Gamma dt\bigg) ,
\end{multline}
where $s>1$ and $\lambda>1$. 
Using the inequality in \cite[Theorem II.4.1]{galdi}, we have
$$
s\lambda^3\int_{\partial\mathcal{F}}\xi^*\left|w\right|^2 \ dy dt\leq \varepsilon s\lambda^2 \int_\mathcal{F} \xi^*\left|\nabla w\right|^2 \ dy dt+Cs\lambda^4\int_\mathcal{F}\xi^* \left|w\right|^2 \ dy dt.
$$	
We have
\begin{multline}
\label{P1Q3}
\int_0^T\left\langle (Pw)_1,(Qw)_3\right\rangle_{[L^2(\mathcal{F})]^2} \ dt=\int_0^T\int_\mathcal{F}\Delta w\cdot\partial_t w \ dy dt=-\int_0^T\int_\mathcal{F} \nabla w:\partial_t \nabla w \ dy dt\\+\int_0^T\int_{\partial\mathcal{F}} \nabla w n\cdot \partial_t w \ d\Gamma dt=-\frac{1}{2}\int
_0^T \frac{d}{dt}\int_\mathcal{F}\left|\nabla w\right|^2 \ dy dt+\int_0^T\int_{\partial\mathcal{F}} \nabla w n\cdot \partial_t w \ d\Gamma dt\\=\int_0^T\int_{\partial\mathcal{F}} \nabla w n\cdot \partial_t w \ d\Gamma dt,
\end{multline}
where we have used that $\nabla w(0,\cdot)=\nabla w(T,\cdot)=0$.
Applying again the Green formula, we get:
\begin{multline*}
\int_0^T\left\langle (Pw)_2,(Qw)_1\right\rangle_{[L^2(\mathcal{F})]^2} \ dt=-2s^3\lambda^3 \int_0^T\int_{\mathcal{F}}\xi^3 \left|\nabla \eta \right|^2 \nabla w\nabla \eta\cdot w \ dy dt\\=3s^3\lambda^4\int_0^T\int_{\mathcal{F}}\left|\nabla \eta\right|^4\xi^3\left|w\right|^2 \ dy dt+ s^3\lambda^3\int_0^T\int_{\mathcal{F}}\Delta \eta\left|\nabla \eta\right|^2\xi^3\left|w\right|^2 \ dy dt\\+2s^3\lambda^3\int_0^T\int_{\mathcal{F}}\partial_i\eta \  \partial_{i,j}^2\eta \ \partial_j\eta \ \xi^3 \left|w\right|^2 \ dy dt-s^3\lambda ^3\int_0^T\int_{\partial\mathcal{F}} (\xi^*)^3\left|\nabla \eta\right|^2(\nabla\eta\cdot n)\left|w\right|^2 \ d\Gamma dt  . 
\end{multline*}
It follows that
\begin{multline*}
\int_0^T	\left\langle (Pw)_2,(Qw)_1\right\rangle_{[L^2(\mathcal{F})]^2} \ dt\geq 3s^3\lambda^4\int_0^T\int_{\mathcal{F}}\left|\nabla \eta\right|^4\xi^3\left|w\right|^2 \ dydt-C s^3\lambda^3\int_0^T\int_{\mathcal{F}}\xi^3\left|w\right|^2 \ dydt\\-s^3\lambda ^3\int_0^T\int_{\partial\mathcal{F}} (\xi^*)^3\left|\nabla \eta\right|^2(\nabla\eta\cdot n)\left|w\right|^2 \ d\Gamma dt  . 
\end{multline*}
Thus
\begin{multline*}
\int_0^T	\left\langle (Pw)_2,(Qw)_1\right\rangle_{[L^2(\mathcal{F})]^2} \ dt\geq 3s^3\lambda^4\int_0^T\int_{\mathcal{F}}\left|\nabla \eta\right|^4\xi^3\left|w\right|^2 \ dydt-C s^3\lambda^3\int_0^T\int_{\mathcal{F}}\xi^3\left|w\right|^2 \ dydt\\-s^3\lambda ^3\int_0^T\int_{\partial\mathcal{F}} (\xi^*)^3\left|\nabla \eta\right|^2(\nabla\eta\cdot n)\left|w\right|^2 \ d\Gamma dt . 
\end{multline*}
We have 
\begin{equation*}
\int_0^T\left\langle (Pw)_2,(Qw)_2\right\rangle_{[L^2(\mathcal{F})]^2} \ dt=-2s^3\lambda^4\int_0^T\int_{\mathcal{F}}\left|\nabla \eta\right|^4\xi^3\left|w\right|^2 \ dy dt.
\end{equation*}
We obtain also
\begin{multline*}
\int_0^T\left\langle (Pw)_2,(Qw)_3\right\rangle_{[L^2(\mathcal{F})]^2} \ dt=s^2\lambda^2\int_0^T\int_{\mathcal{F}}\left|\nabla \eta\right|^2\xi^2w\cdot\partial_t w \ dy dt\\=\frac{1}{2}s^2\lambda^2\int_0^T\int_{\mathcal{F}}\left|\nabla \eta\right|^2\xi^2\partial_t\left|w\right|^2  \ dy dt=-\frac{1}{2}s^2\lambda^2\int_0^T\int_{\mathcal{F}}\partial_t(\left|\nabla \eta\right|^2\xi^2)\left|w\right|^2  \ dy dt\\
=-s^2\lambda^2\int_0^T\int_{\mathcal{F}}\left|\nabla \eta\right|^2\xi\partial_t\xi\left|w\right|^2  \ dy dt,
\end{multline*}
thanks to the fact that $w(0,\cdot)=w(T,\cdot)=0$.
Thus, we have
\begin{multline}
\label{P2Q}
\int_0^T	\left\langle (Pw)_2,Qw\right\rangle_{[L^2(\mathcal{F})]^2} \ dt\geq s^3\lambda^4\int_0^T\int_{\mathcal{F}}\left|\nabla \eta\right|^4\xi^3\left|w\right|^2 \ dydt-Cs^3\lambda^3\int_0^T\int_{\mathcal{F}}\xi^3\left| w\right|^2 \ dydt\\-s^3\lambda ^3\int_0^T\int_{\partial\mathcal{F}} (\xi^*)^3\left|\nabla \eta\right|^2(\nabla\eta\cdot n)\left|w\right|^2 \ d\Gamma dt .
\end{multline}
	
Thus, combining \eqref{p1q1}, \eqref{P1Q}, \eqref{P2Q}, we get 
\begin{multline*}
s\lambda^2\int_0^T\int_{\mathcal{F}}\xi\left| \nabla\eta\right|^2 \left|\nabla w\right|^2  \ dy dt+s^3\lambda^4\int_0^T\int_{\mathcal{F}}\left|\nabla \eta\right|^4\xi^3\left|w\right|^2 \ dy dt +s^3\lambda^3\int_0^T\int_{\partial\mathcal{F}}(\xi^*)^3\left|w\right|^2 \ d\Gamma dt \\+s\lambda\int_0^T\int_{\partial\mathcal{F}}\left|\nabla w n \right|^2  \ d\Gamma dt \leq C\bigg( (s^3\lambda^3 +s^2\lambda^4)\int_0^T\int_{\mathcal{F}} \xi^3\left|w \right|^2 \ dy dt \\+s\lambda\int_0^T\int_{\mathcal{F}} \xi \left|\nabla w \right|^2 \ dydt+\int_0^T\int_{\mathcal{F}} e^{-2s\beta}\left| f\right|^2 \ dy dt +s\lambda\int_0^T\int_{\partial\mathcal{F}}e^{-2s\widehat{\beta}}\xi^*\left| \nabla w\tau\right| ^2 d\Gamma dt\\+\varepsilon s\lambda\int_0^T\int_{\partial\mathcal{F}}\xi^*\left|\nabla w n \right|^2  \ d\Gamma dt+s^{-1}\lambda^{-1}\int_0^T\int_{\partial\mathcal{F}}\xi^{-1}\left|\partial_tw\right|^2 \ d\Gamma dt \bigg),
\end{multline*}
for all $\varepsilon>0$. 
Using the fact that $\left| \nabla\eta\right|>0 $ on $\mathcal{F} \backslash \mathcal{O}_\eta$, we obtain from the last inequality
\begin{multline*}
s\lambda^2\int_0^T\int_{\mathcal{F}\backslash \mathcal{O}_\eta}\xi \left|\nabla w\right|^2  \ dy dt+s^3\lambda^4\int_0^T\int_{\mathcal{F}\backslash \mathcal{O}_\eta}\xi^3\left|w\right|^2 \ dy dt +s^3\lambda^3\int_0^T\int_{\partial\mathcal{F}}(\xi^*)^3\left|w\right|^2 \ d\Gamma dt\\+s\lambda\int_0^T\int_{\partial\mathcal{F}}\left|\nabla w n \right|^2  \ d\Gamma dt \leq C\bigg( (s^3\lambda^3 +s^2\lambda^4)\int_0^T\int_{\mathcal{F}} \xi^3\left|w \right|^2 \ dy dt +s\lambda\int_0^T\int_{\mathcal{F}} \xi \left|\nabla w \right|^2 \ dy dt\\+\int_0^T\int_{\mathcal{F}} e^{-2s\beta}\left| f\right|^2 \ dy dt +s\lambda\int_0^T\int_{\partial\mathcal{F}}e^{-2s\widehat{\beta}}\xi^*\left| \nabla w\tau\right| ^2 \ d\Gamma dt+\varepsilon s\lambda\int_0^T\int_{\partial\mathcal{F}}\xi^*\left|\nabla w n \right|^2  \ d\Gamma dt\\+s^{-1}\lambda^{-1}\int_0^T\int_{\partial\mathcal{F}}(\xi^*)^{-1}\left|\partial_tw\right|^2 \ d\Gamma dt \bigg).
\end{multline*}
We add the term 
$$
s\lambda^2\int_0^T\int_{\mathcal{O}_\eta}\xi \left|\nabla w\right|^2  \ dy dt+s^3\lambda^4\int_0^T\int_{\mathcal{O}_\eta}\xi^3\left|w\right|^2 \ dy dt,
$$
on the both sides of the previous equation, we get
\begin{multline*}
s\lambda^2\int_0^T\int_{\mathcal{F}}\xi \left|\nabla w\right|^2  \ dy dt+s^3\lambda^4\int_0^T\int_{\mathcal{F}}\xi^3\left|w\right|^2 \ dy dt +s^3\lambda^3\int_0^T\int_{\partial\mathcal{F}}(\xi^*)^3\left|w\right|^2 \ d\Gamma dt\\+s\lambda\int_0^T\int_{\partial\mathcal{F}}\left|\nabla w n \right|^2  \ d\Gamma dt \leq C\bigg( (s^3\lambda^3 +s^2\lambda^4)\int_0^T\int_{\mathcal{F}} \xi^3\left|w \right|^2 \ dy dt +s\lambda\int_0^T\int_{\mathcal{F}} \xi \left|\nabla w \right|^2 \ dy dt\\s\lambda^2\int_0^T\int_{\mathcal{O}_\eta}\xi \left|\nabla w\right|^2  \ dy dt+s^3\lambda^4\int_0^T\int_{\mathcal{O}_\eta}\xi^3\left|w\right|^2 \ dy dt+\int_0^T\int_{\mathcal{F}} e^{-2s\beta}\left| f\right|^2 \ dy dt\\ +s\lambda\int_0^T\int_{\partial\mathcal{F}}e^{-2s\widehat{\beta}}\xi^*\left| \nabla w\tau\right| ^2 \ d\Gamma dt+\varepsilon s\lambda\int_0^T\int_{\partial\mathcal{F}}\xi^*\left|\nabla w n \right|^2  \ d\Gamma dt\\+s^{-1}\lambda^{-1}\int_0^T\int_{\partial\mathcal{F}}(\xi^*)^{-1}\left|\partial_tw\right|^2 \ d\Gamma dt \bigg).
\end{multline*}
We choose $( s,\lambda)$ sufficiently large to have 
$$
s\lambda^2-Cs\lambda\geq\frac{ s\lambda^2}{2},\quad s^3\lambda^4-C(s^{3}\lambda^{3} +s^2\lambda^4)\geq\frac{s^3\lambda^4}{2},
$$ 
to absorb the first and the second terms of the right hand side of the previous equation, we get
\begin{multline*}
s\lambda^2\int_0^T\int_{\mathcal{F}}\xi \left|\nabla w\right|^2  \ dy  dt+s^3\lambda^4\int_0^T\int_{\mathcal{F}}\xi^3\left|w\right|^2 \ dy dt +s^3\lambda^3\int_0^T\int_{\partial\mathcal{F}}(\xi^*)^3\left|w\right|^2 \ d\Gamma dt \\\leq C\bigg( s^3\lambda^4\int_0^T\int_{\mathcal{O}_\eta} \xi^3\left|w \right|^2 \ dy dt +s\lambda^2\int_0^T\int_{\mathcal{O}_\eta} \xi \left|\nabla w \right|^2 \ dy dt+\int_0^T\int_{\mathcal{F}} e^{-2s\beta}\left| f\right|^2 \ dy dt \\+s\lambda\int_0^T\int_{\partial\mathcal{F}}\xi^*\left| \nabla w\tau\right| ^2 d\Gamma dt+s^{-1}\lambda^{-1}\int_0^T\int_{\partial\mathcal{F}}(\xi^*)^{-1}\left|\partial_tw\right|^2 \ d\Gamma dt\bigg).
\end{multline*}
Coming back to $\psi$, the last inequality writes
\begin{multline*}
s\lambda^2\int_0^T\int_{\mathcal{F}}e^{-2s\beta}\xi \left|\nabla \psi\right|^2  \ dy dt+s^3\lambda^4\int_0^T\int_{\mathcal{F}}e^{-2s\beta}\xi^3\left|\psi\right|^2 \ dy dt +s^3\lambda^3\int_0^T\int_{\partial\mathcal{F}}e^{-2s\widehat{\beta}}(\xi^*)^3\left|\psi\right|^2d\Gamma dt \\\leq C\bigg( s^3\lambda^4\int_0^T\int_{\mathcal{O}_\eta} e^{-2s\beta}\xi^3\left|\psi \right|^2 \ dy dt +s\lambda^2\int_0^T\int_{\mathcal{O}_\eta}e^{-2s\beta} \xi \left|\nabla\psi \right|^2 \ dy dt+\int_0^T\int_{\mathcal{F}} e^{-2s\beta}\left| f\right|^2 \ dy dt \\+s\lambda\int_0^T\int_{\partial\mathcal{F}}e^{-2s\widehat{\beta}}\xi^*\left| \nabla \psi\tau\right| ^2 d\Gamma+s^{-1}\lambda^{-1}\int_0^T\int_{\partial\mathcal{F}}e^{-2s\widehat{\beta}}(\xi^*)^{-1}\left|\partial_t\psi\right|^2 \ d\Gamma dt\bigg),
\end{multline*}
where we have used that
\begin{multline*}
s^{-1}\lambda^{-1}\int_0^T\int_{\partial\mathcal{F}}(\xi^*)^{-1}\left|\partial_tw\right|^2 \ d\Gamma dt\leq C\bigg( s^{-1}\lambda^{-1}\int_0^T\int_{\partial\mathcal{F}}e^{-2s\widehat{\beta}}(\xi^*)^{-1}\left|\partial_t\psi\right|^2 \ d\Gamma dt\\+s^{-1}\lambda^{-1}\int_0^T\int_{\partial\mathcal{F}}\partial_t(e^{-2s\widehat{\beta}})(\xi^*)^{-1}\left|\psi\right|^2 \ d\Gamma dt\bigg),
\end{multline*}
and the fact that $|\partial_t\widehat{\beta}|\leq C(T+T^2 )(\xi^*)^{1+1/N}$. We obtain \eqref{car}.
\end{proof}


\bibliographystyle{plain}
\bibliography{biblio}
	
\end{document}